\documentclass[11pt]{gsm-l}

\usepackage{extsizes}

\usepackage{hyperref}

\usepackage{tikz, tikz-cd}
\usepackage{rotating}

\usepackage{enumitem}
\setlist[itemize]{topsep=0pt,itemsep=0pt}
\setlist[enumerate]{topsep=0pt,itemsep=0pt}

\usepackage{scalerel}
\usepackage{tkz-euclide}
\usetikzlibrary{arrows, arrows.meta, bending, decorations.pathmorphing, decorations.markings} 

\usepackage{auto-pst-pdf}

\usepackage{pict2e, amssymb,latexsym,amstext,amsmath,amsthm,epsfig,ifthen}
\usepackage{color}
\usepackage{blkarray}
\usepackage[capitalize, nameinlink, noabbrev, nosort]{cleveref}
\usepackage{youngtab}
\usepackage[boxsize=.35 em]{ytableau}

\usepackage[dvips,centering,includehead,width=12.61cm,height=22.1cm,paperheight=23.8cm,paperwidth=17cm]{geometry}

\input epsf

\newtheorem{theorem}{Theorem}[section]
\newtheorem{proposition}[theorem]{Proposition}

\newtheorem{corollary}[theorem]{Corollary}
\newtheorem{lemma}[theorem]{Lemma}

\theoremstyle{definition}

\newtheorem{remark}[theorem]{Remark}
\newtheorem{example}[theorem]{Example}
\newtheorem{definition}[theorem]{Definition}
\newtheorem{problem}[theorem]{Problem}

\newtheorem{exercise}[theorem]{Exercise}

\numberwithin{section}{chapter}
\numberwithin{equation}{section}
\numberwithin{figure}{chapter}

\def\endproof{\hfill$\square$\medskip}

\def\ZZ{\mathbb{Z}}

\def\CC{\mathbb{C}}

\def\RR{\mathbb{R}}
\def\QQ{\mathbb{Q}}
\def\TT{\mathbb{T}}

\def\AA{\mathcal{A}}
\def\FFcal{\mathcal{F}}

\def\xx{\mathbf{x}}
\def\yy{\mathbf{y}}
\def\Xcal{\mathcal{X}}
\def\Gr{\operatorname{Gr}}
\def\SL{\operatorname{SL}}

\def\Qsf{\QQ_{\,\rm sf}}

\newcommand{\light}[1]{{#1}}
\newcommand{\dark}[1]{{#1}}

\newcommand{\overunder}[2]{
\!\begin{array}{c}
\scriptstyle{#1}\\[-.1in]
-\!\!\!-\!\!\!-\\[-.1in]
\scriptstyle{#2}
\end{array}
\!
}

\newcommand{\smalloverunder}[2]{
\!\!\begin{array}{c}
\scriptstyle{#1}\\[-.1in]
-\!\!\!-\\[-.1in]
\scriptstyle{#2}
\end{array}
\!\!
}

\def\Trop{\operatorname{Trop}}

\newcommand{\notch}{\scriptstyle\bowtie}

\newsavebox{\digon}
\savebox{\digon}(10,10)[bl]{
\qbezier(5,0)(0,5)(5,10)
\qbezier(5,0)(10,5)(5,10)
\put(5,0){\circle*{1}} 
\put(5,5){\circle*{1}} 
\put(5,10){\circle*{1}} 
}

\newsavebox{\lowbar}
\savebox{\lowbar}(10,10)[bl]{
\put(5,0){\line(0,1){5}}
}

\newsavebox{\lowtag}
\savebox{\lowtag}(10,10)[bl]{
\put(5,3.5){\makebox(0,0){$\notch$}}
}

\newsavebox{\highbar}
\savebox{\highbar}(10,10)[bl]{
\put(5,5){\line(0,1){5}}
}

\newsavebox{\hightag}
\savebox{\hightag}(10,10)[bl]{
\put(5,6.5){\makebox(0,0){$\notch$}}
}

\begin{document}


\frontmatter

\title{
\textsf{\Huge \hbox{Introduction to Cluster Algebras}}\\[.1in]
\textsf{\Huge Chapters 1--3} \\[.2in]
{\rm\textsf{\LARGE (preliminary version)}}
}

\author{\Large \textsc{Sergey Fomin}}

\author{\Large \textsc{Lauren Williams}}

\author{\Large \textsc{Andrei Zelevinsky}}







\maketitle

\noindent
\textbf{\Huge Preface}

\vspace{1in}


\noindent

\makeatletter
\newcommand{\manuallabel}[2]{\phantomsection\def\@currentlabel{#2}\label{#1}}
\makeatother

\manuallabel{ch:finitetype}{5}
\manuallabel{ch:rings}{6}
\manuallabel{sec:rings-baseaffine}{6.5}
\manuallabel{sec:rings-matrices}{6.6}
\manuallabel{sec:plucker-rings}{6.7}
\manuallabel{sec:generators+relations}{6.8}
\manuallabel{ex:Gr36}{6.8.13}
\manuallabel{ch:plabic}{7}
\manuallabel{ch:Grassmannians}{8}
\manuallabel{ch:surfaces}{10}
\manuallabel{ch:dynamical}{11}
\manuallabel{ch:general-cluster-algebras}{12}


This is a preliminary draft of Chapters 1--3 of our forthcoming textbook
\textsl{Introduction to cluster algebras}, 
joint with Andrei Zelevinsky (1953--2013). 
Other chapters have been posted as 
\begin{itemize}[leftmargin=.2in]
	\item \href{https://arxiv.org/abs/1707.07190}{(Chapters~4--5)}, 
	\item \href{https://arxiv.org/abs/2008.09189}{(Chapter~6)}, and
	\item \href{https://arxiv.org/abs/2106.02160}{(Chapter~7)}. 
\end{itemize}
We~expect to post additional chapters in the not so distant future. 

\medskip

This book grew from the ten lectures given by Andrei at the
NSF CBMS conference on Cluster Algebras and Applications at North
Carolina State University 
in June 2006. 
The material of his lectures is much expanded but we still follow
the original plan aimed at giving an accessible introduction to the subject
for a general mathematical audience.

Since its inception in \cite{ca1}, the theory of cluster algebras has
been actively developed in many directions. 
We do not attempt to give a comprehensive treatment of the many
connections and applications of this young theory.
Our choice of topics reflects our personal taste; 
much of the book is based on the work done by Andrei and ourselves.

\medskip


Comments and suggestions are welcome. 

\bigskip

\rightline{Sergey Fomin}
\rightline{Lauren Williams}

\vfill

\noindent
Partially supported by the NSF grants DMS-1049513, 
DMS-1361789, DMS-1664722, DMS-215299, and DMS-2348501. 
\medskip


\bigskip

\noindent
\copyright \ 2016--2025 by 
Sergey Fomin, Lauren Williams, and Andrei Zelevinsky

\tableofcontents

\newpage

\thispagestyle{empty}

\section*{Acknowledgments}

We are grateful to Ari Krishna, Gregory Li, 
Stella Li, Annabel Ma, Jacob Paltrowitz, and Katherine Tung for a number
of valuable suggestions on the earlier version of Chapters 1--3.

\mainmatter



\chapter{\hbox{Total positivity}}
\label{ch:tp-examples}

Total positivity, along with G.~Lusztig's theory of
canonical bases, was one of the main motivations
for the development of cluster algebras. 
In this chapter, we present the basic notions of total positivity,
focusing on three important examples (to be re-examined again in
the future):
square matrices, Grassmannians, and basic affine spaces. 
As our main goal here is to provide motivation rather than develop a
rigorous theory, the exposition is somewhat informal.
Additional details and
  references can be found in~\cite{sf-icm}. 

\section{Totally positive matrices} 
\label{sec:TP}

An $n \times n$ matrix with real entries is called  
\emph{totally positive} 
\index{matrix!totally positive}
(TP for short) if all its minors---that is, determinants of square
submatrices---are positive. 
A real matrix is called \emph{totally nonnegative} (or~TNN) if all
its minors are nonnegative. 
The first systematic study of these classes of matrices was conducted
in the 1930s by F.~Gantmacher and M.~Krein~\cite{gantmacher-krein},  
following the pioneering work of I.~Schoenberg~\cite{schoenberg}. 
In particular, they 
showed that the eigenvalues of an $n\times n$ totally positive matrix
are real, positive, and distinct. 

Total positivity is a remarkably widespread phenomenon:
TP and TNN matrices play an important role, \emph{inter alia}, 
in classical mechanics, 
probability, 
discrete potential theory, 
asymptotic representation theory, 
algebraic and enumerative combinatorics, 
and the theory of integrable systems. 

The \emph{Binet-Cauchy Theorem} implies that TP 
(resp., TNN) matrices in $G=\operatorname{SL}_n$ are closed under
multiplication; they thus form a
  multiplicative semigroup, denoted by~$G_{>0}$ (resp.,~$G_{\ge 0}$). 
The following ``splitting lemma'' due to C.~Cryer \cite{cryer,
  cryer76} shows that the study of $G_{\ge 0}$ can be reduced
to the investigation of its subsemigroup  
of upper-triangular unipotent TNN matrices, i.e. upper-triangular
TNN matrices with $1$'s on the diagonal: 
\begin{lemma}
\label{lem:cryer}
A matrix $z\in \operatorname{SL}_n$ 
is totally
nonnegative if and only if it has a 
Gaussian decomposition
\[
z=\left[\!\!\begin{array}{ccccc}
1      & 0      & 0      & \cdots & 0      \\
*      & 1      & 0      & \cdots & 0      \\
*      & *      & 1      & \cdots & 0      \\
\vdots & \vdots & \vdots & \ddots & \vdots \\
*      & *      & *      & \cdots & 1 \\
\end{array}\!\right]
\left[\!\!\begin{array}{ccccc}
*      & 0      & 0      & \cdots & 0      \\
0      & *      & 0      & \cdots & 0      \\
0      & 0      & *      & \cdots & 0      \\
\vdots & \vdots & \vdots & \ddots & \vdots \\
0      & 0      & 0      & \cdots & * \\
\end{array}\!\right]
\left[\!\!\begin{array}{ccccc}
1      & *      & *      & \cdots & *      \\
0      & 1      & *      & \cdots & *      \\
0      & 0      & 1      & \cdots & *      \\
\vdots & \vdots & \vdots & \ddots & \vdots \\
0      & 0      & 0      & \cdots & 1 \\
\end{array}\!\right]
\]
in which all three factors 
(lower-triangular unipotent, diagonal, and upper-tri\-an\-gu\-lar unipotent)
are totally nonnegative. 
\end{lemma}

There is also a counter\-part of this statement for totally positive
matrices. 

The \emph{Loewner-Whitney Theorem} \cite{loewner, whitney}
  identifies the infinitesimal generators of $G_{\ge 0}$ 
as the \emph{Chevalley generators} of the corresponding Lie algebra. 
In pedestrian terms, each  
$n\times n$ TNN matrix can be written as a product of
matrices of the form $x_i(t)$, $y_i(t)$, and $z_i(t)$, where 
each parameter $t$ is positive, 
the matrices $x_i(t)$ are defined by 
\[
x_i(t)=
\left[\begin{array}{cccccc}
1      & \cdots & 0      & 0      & \cdots & 0      \\
\vdots & \ddots & \vdots & \vdots & \ddots & \vdots \\
0      & \cdots & 1      & t      & \cdots & 0      \\
0      & \cdots & 0      & 1      & \cdots & 0      \\
\vdots & \ddots & \vdots & \vdots & \ddots & \vdots \\
0      & \cdots & 0      & 0      & \cdots & 1 \\
\end{array}\right],
\]
where the $t$ is in row $i$ of $x_i(t)$,
$y_i(t)$ is the transpose of $x_i(t)$, and $z_i(t)$ is the 
diagonal matrix with diagonal entries $(1,\dots,1, t, t^{-1}, 1, \dots, 1)$
where $t$ and $t^{-1}$ are in positions $i$ and $i+1$.
This led G.~Lusztig \cite{lusztig} to the idea of extending the notion
of total positivity to 
other semisimple Lie groups~$G$, by defining the set~$G_{\ge 0}$ of
TNN elements in~$G$ as the semigroup generated by the
Chevalley generators. 
Lusztig proved that $G_{\ge 0}$ is a semialgebraic subset of~$G$,
and described it by inequalities of the form $\Delta(x)\ge 0$ where
$\Delta$ lies in the appropriate \emph{dual canonical basis};
see \cite[Section~5]{lusztig-survey}. 
A~simpler description in terms of  \emph{generalized~minors}~\cite{fz-dbc}
was given in~\cite{fz-osc}. 

A yet more general (if informal) concept is one of a 
\emph{totally positive} (or \emph{totally nonnegative}) (sub)variety
of a given complex algebraic variety~$Z$. 
Vaguely, the idea is this: 
suppose that $Z$ comes equipped with a family $\mathbf\Delta$ of
``important'' regular functions on~$Z$.
The corresponding TP (resp., TNN) 
variety $Z_{>0}$ (resp., $Z_{\ge 0}$)
is the set of points 
at which all of these functions
take positive (resp., nonnegative) real values: 
\[
Z_{>0}=\{z\in Z:  \Delta(z)>0\ \text{for all $\Delta\in\mathbf\Delta$}\}. 
\]
If $Z$ is the affine space of $n\times n$ matrices (or 
$Z=\operatorname{GL}_n(\CC)$, or $Z=\operatorname{SL}_n(\CC)$), 
and $\mathbf\Delta$ is the set of all
minors, then we recover the classical notions. 
One can restrict this construction to matrices lying in a given
stratum of a Bruhat decomposition, or in a given \emph{double Bruhat
  cell}~\cite{fz-dbc, lusztig}. 
Another important example is the \emph{totally positive (resp.,
  nonnegative) Grassmannian}
consisting of the points in a usual Grassmann manifold where all 
Pl\"ucker coordinates can be chosen to be positive (resp.,
nonnegative). 
This construction can be extended to arbitrary partial flag manifolds,
and more generally to homogeneous spaces $G/P$ associated to
semisimple complex Lie groups. 

We note that in each of the examples alluded to above, the notion of positivity depends on 
a particular choice of a coordinate system: a basis in a vector
space allows us to view linear transformations as matrices,
determines a system of Pl\"ucker coordinates, etc. 

Among many questions which one may ask about totally
positive/non\-negative varieties $Z_{>0}$ and~$Z_{\ge 0}\,$,  
let us restrict our attention to the problem of \emph{efficient TP
  testing}:
how many inequalities (and which ones) does one need to check in order to
ascertain that a given point in~$Z$ is totally positive? 
In~particular, are there efficient ways for testing a very large matrix for total 
positivity? 
(It is not hard to see that an $n \times n$ matrix
has altogether $\binom{2n}{n} - 1$ minors,
a number which grows exponentially in~$n$.) 
Examples~\ref{ex:SL_2} and~\ref{ex:U-in-SL_3} provide a glimpse into the
tricks used to construct efficient TP
criteria. 

\begin{example}
\label{ex:SL_2}
A $2 \times 2$ matrix $z = \left[\begin{smallmatrix}{a}&{b}\\{c}& {d}\end{smallmatrix}\right]$ has
five minors: the matrix entries $a,b,c,d$ and the determinant $\Delta = \det(z)=ad - bc$.
Now the identity
\begin{equation}
\label{eq:2by2-3term-relation}
ad = \Delta + bc 
\end{equation}
shows that we do not have to check all five minors:
if $a$, $b$, $c$, and~$\Delta$ are positive, then so is 
$d = (\Delta + bc)/a$.
(Alternatively, test the minors $d, b, c, \Delta$.) 

\end{example}

\begin{example}
\label{ex:U-in-SL_3}
Now let $n=3$.
To keep it simple (cf.\ also Lemma~\ref{lem:cryer}), 
let us consider the subgroup 
of unipotent upper triangular matrices
\begin{equation*}
\label{eq:3by3}
z= \left[\begin{smallmatrix}
1  & a & b \\[2pt]
0  & 1 & c \\[2pt]
0  & 0 & 1
\end{smallmatrix}\right] \in \operatorname{SL}_3.
\end{equation*}
Since some of the entries of $z$ are equal to $0$, we modify the
definition of  total positivity by requiring that 
$\Delta(z) > 0$ for each minor $\Delta$
which does not identically vanish on the subgroup.
This leaves us with four minors to check for positivity: the
matrix entries $a, b, c$, and the $2 \times 2$ minor
$P = ac-b$.
Again we can reduce the number of needed checks from $4$ to $3$
using the~identity 
\begin{equation}
\label{eq:3by3-unipotent-3term-relation}
ac = P + b \, .
\end{equation}
Thus each of the sets $\{a,b,P\}$ and $\{b,c,P\}$ provides
an efficient TP test. 
\end{example}

We note that in each of the above examples, the number of checks
involved in each positivity test 
was equal to the dimension of the variety at
hand. It seems implausible that one could do better.

\section{The Grassmannian of $2$-planes in $m$-space}
\label{sec:Ptolemy}

Before developing efficient total positivity tests for square
matrices, we shall discuss the somewhat simpler case of Grassmannians
of $2$-planes. 

Recall that the complex \emph{Grassmann manifold} (the \emph{Grassmannian}
for short), 
denoted $\Gr_{k,m}=\Gr_{k,m}(\mathbb{C})$, 
is the variety of all  
\hbox{$k$-dimensional} subspaces in an $m$-dimensional complex vector
space. 
Let us fix a basis in this space, thereby identifying it with~$\mathbb{C}^m$. 
Now any $k \times m$ matrix~$z$ of rank~$k$ defines a point
$[z]\in\Gr_{k,m}$,
the \emph{row span} of~$z$. 

Given a $k$-element subset 
$J\subset \{1,\dots,m\}$, 
the \emph{Pl\"ucker coordinate} $P_J(z)$ (evaluated at a matrix~$z$ as
above) is, by definition, the $k \times k$ minor
of $z$ determined by the column set~$J$.
The collection $(P_J(z))_{|J|=k}$ 
only depends on the row span $[z]$ (up to common rescaling),
and in fact provides an embedding
of  $\Gr_{k,m}$ into the complex projective space of dimension
$\binom{m}{k}-1$, called the \emph{Pl\"ucker embedding}.

The \emph{Pl\"ucker ring} $R_{k,m}$ is the ring generated by the 
 Pl\"ucker coordinates $P_J$ for all $k$-element subsets $J$ of $\{1,\dots,m\}$.
This is the homogeneous coordinate ring of~$\Gr_{k,m}$ with respect to the
Pl\"ucker embedding. 
The ideal of relations that the generators $P_{J}$ satisfy is
generated by certain quadratic  relations called
\emph{Grassmann-Pl\"ucker relations}.

The Pl\"ucker coordinates are used to define the totally 
positive points of the Grassmannian, as follows.

\begin{definition}
The \emph{totally positive Grassmannian} $\Gr_{k,m}^+$ 
is the subset of $\Gr_{k,m}$ consisting of points 
whose Pl\"ucker coordinates can be chosen so that all of them are
positive real numbers. 
(Recall that Pl\"ucker coordinates are defined up to a common
rescaling.) 
\end{definition}

In simple terms, an element $[z]\in \Gr_{k,m}$ 
defined by a full-rank $k\times m$ matrix~$z$ 
(without loss of generality, $z$~can be assumed to have real entries)
is TP if all the maximal (i.e., $k\times k$) minors $P_J(z)$
are of the same sign. 
We can begin by checking one particular value $P_J(z)$, 
and if it happens to be negative, replace the top row of $z$ by its negative.
Thus the problem at hand can be restated as follows:
find an efficient method for checking whether all 
maximal minors of a given $k\times m$ matrix are positive. 
A~brute force test requires $\binom{m}{k}$ checks.
Can this number be reduced?

In this section, we systematically study 
the case $k=2$ (Grassmannians of $2$-planes). 
Arbitrary Grassmannians $\Gr_{k,m}$ will be treated in Chapter~\ref{ch:Grassmannians}.

In the case of the Grassmannian $\Gr_{2,m}$, 
there are $\binom{m}{2}$ Pl\"ucker coordinates $P_{ij}=P_{\{i,j\}}$
labeled by pairs of integers $1\le i<j\le m$. 
It turns out however that in order to verify that all the 
$2\times 2$ minors $P_{ij}(z)$ of a given $2\times m$ matrix~$z$ are positive, 
it suffices to check the positivity of only $2m-3$ special minors.
(Note that $2m-3
$ is the dimension of the affine cone over~$\Gr_{2,m}$.)

\begin{exercise}
Show that $2 \times 2$ minors of a $2 \times m$ matrix
(equivalently, the Pl\"ucker coordinates~$P_{ij}$)
satisfy 
the
\emph{three-term Grass\-mann-Pl\"ucker relations} 
\begin{equation}
\label{eq:grassmann-plucker-3term}
P_{ik}\,P_{jl} = P_{ij}\,P_{kl} + P_{il}\,P_{jk}\qquad (1 \leq i < j <
k < l \leq m). 
\end{equation}
\end{exercise}

We are going to construct a family of ``optimal" tests for 
total positivity in $\Gr_{2,m}$ using 
triangulations of an $m$-gon.
Consider a convex $m$-gon~$\mathbf{P}_{m}$ with
its vertices labeled clockwise.
We associate the Pl\"ucker coordinate $P_{ij}$ with the chord
(i.e., a side or a diagonal of~$\mathbf{P}_{m}$)
whose endpoints are $i$ and~$j$. 

Now, let $T$ be a \emph{triangulation} of $\mathbf{P}_m$ 
by pairwise noncrossing diagonals.
(Common endpoints are allowed.) 
We view~$T$ as a maximal collection of 
noncrossing chords; as such, it consists of 
$m$ sides and $m-3$ diagonals, giving rise to a collection 
$\mathbf{\tilde x}(T)$ of $2m-3$ Pl\"ucker coordinates, which we call an
\emph{extended cluster}.  The Pl\"ucker coordinates
corresponding to the sides of $\mathbf{P}_m$ are called 
\emph{frozen variables}. 
They are present in every extended cluster~$\mathbf{\tilde x}(T)$, 
hence the term ``frozen;'' 
an alternative terminology is \emph{coefficient variables}. 
The remaining $m-3$ Pl\"ucker coordinates corresponding
to diagonals of $\mathbf{P}_m$ are called \emph{cluster variables};
they form a \emph{cluster}.\linebreak[3]
Thus each extended cluster consists of $m-3$ cluster variables
and $m$ frozen variables.
We note that these $2m\!-\!3$ quantities are algebraically independent. 
See Figure~\ref{fig:octagon}. 

\begin{figure}[ht] 
\begin{center} 
\vspace{-5pt}
\setlength{\unitlength}{2.4pt} 
\begin{picture}(60,65)(0,-3) 
\thicklines 
  \multiput(0,20)(60,0){2}{\line(0,1){20}} 
  \multiput(20,0)(0,60){2}{\line(1,0){20}} 
  \multiput(0,40)(40,-40){2}{\line(1,1){20}} 
  \multiput(20,0)(40,40){2}{\line(-1,1){20}} 
 
  \multiput(20,0)(20,0){2}{\circle*{1}} 
  \multiput(20,60)(20,0){2}{\circle*{1}} 
  \multiput(0,20)(0,20){2}{\circle*{1}} 
  \multiput(60,20)(0,20){2}{\circle*{1}} 
 
\thinlines 
\put(40,0){\line(1,2){20}} 
\put(0,40){\line(1,0){60}} 
\put(0,20){\line(2,-1){40}} 
\put(0,40){\line(1,-1){40}} 
\put(40,60){\line(-2,-1){40}} 
 
\put(15,16){\makebox(0,0){$P_{46}$}} 
\put(23,22){\makebox(0,0){$P_{47}$}} 
\put(47,22){\makebox(0,0){$P_{24}$}} 
\put(34,42.5){\makebox(0,0){$P_{27}$}} 
\put(23,48){\makebox(0,0){$P_{17}$}} 
 
\put(40,-3){\makebox(0,0){$4$}} 
\put(63,20){\makebox(0,0){$3$}} 
\put(63,40){\makebox(0,0){$2$}} 
\put(40,63){\makebox(0,0){$1$}} 
\put(20,63){\makebox(0,0){$8$}} 
\put(-3,40){\makebox(0,0){$7$}} 
\put(-3,20){\makebox(0,0){$6$}} 
\put(20,-3){\makebox(0,0){$5$}}

\end{picture} 
\vspace{-10pt}
\end{center} 
\caption{A triangulation $T$ of an octagon~$\mathbf{P}_8\,$. 
The extended cluster $\mathbf{\tilde x}(T)$ consists of the cluster variables $P_{17},
P_{24}, P_{27}, P_{46}, P_{47}$ and the frozen variables
$P_{12},P_{23},\dots,P_{78}, P_{18}\,$.} 
\label{fig:octagon} 
\end{figure} 

\begin{theorem}
\label{th:Grassmannian-TP-criterion}
Each Pl\"ucker coordinate $P_{ij}$ 
can be written as a sub\-trac\-tion-free 
rational expression in the 
elements of
a given extended cluster~$\mathbf{\tilde x}(T)\,$. 
Thus, if the $2m-3$ Pl\"ucker coordinates $P_{ij}\in \mathbf{\tilde x}(T)$
evaluate positively at a given $2\times m$ matrix~$z$,
then all $2\times 2$ minors of $z$ are positive. 
\end{theorem}

To clarify, a \emph{subtraction-free} expression is a formula involving
variables and positive integers that uses only the 
operations of addition, multiplication, and division.

\subsection*{Proof of Theorem \ref{th:Grassmannian-TP-criterion}}
Let us visualize the three-term relations
\eqref{eq:grassmann-plucker-3term}
using the $m$-gon~$\mathbf{P}_m\,$.
Take four vertices $i<j<k<l$ of~$\mathbf{P}_m\,$,
cf.\ Figure~\ref{fig:type-a-exch}.
Then the relation \eqref{eq:grassmann-plucker-3term}
is reminiscent of the 
classical Ptolemy Theorem which asserts that for an inscribed
quadrilateral, 
the products of the lengths of two pairs of
opposite sides add up to the product of the two diagonals.

\begin{figure}[ht]
\begin{center}
\vspace{-10pt}
\setlength{\unitlength}{1pt}
\begin{picture}(60,62)(0,2)
\thicklines
  \put(0,20){\line(1,2){20}}
  \put(0,20){\line(1,-1){20}}
  \put(0,20){\line(3,1){60}}
  \put(20,0){\line(0,1){60}}
  \put(20,0){\line(1,1){40}}
  \put(20,60){\line(2,-1){40}}

  \put(20,0){\circle*{1}}
  \put(20,60){\circle*{1}}
  \put(0,20){\circle*{1}}
  \put(60,40){\circle*{1}}

\put(20,67){\makebox(0,0){$i$}}
\put(65,40){\makebox(0,0){$j$}}
\put(20,-6){\makebox(0,0){$k$}}
\put(-4,20){\makebox(0,0){$l$}}

\end{picture}
\vspace{-10pt}
\end{center}
\caption{Three-term Grassmann-Pl\"ucker (or Ptolemy)
  relation~\eqref{eq:grassmann-plucker-3term}. 
}
\label{fig:type-a-exch}
\end{figure}

\noindent
Now Theorem \ref{th:Grassmannian-TP-criterion}
is an immediate consequence of the following three facts:
\begin{enumerate}[leftmargin=.25in]
\item Every Pl\"ucker coordinate appears as an element of 
an extended cluster $\mathbf{\tilde x}(T)$ for some triangulation $T$ of the polygon~$\mathbf{P}_m$.
\item
Any two triangulations of $\mathbf{P}_m$ can be
transformed into each other by a sequence of 
\emph{flips}.
Each flip removes a diagonal of a triangulation 
to create a quadrilateral, then
replaces it with the other diagonal of the
same quadrilateral.
See Figure~\ref{fig:flip}.

\begin{figure}[ht]
\begin{center}
\vspace{-15pt}
\setlength{\unitlength}{1pt}
\begin{picture}(60,58)(0,2)
\thicklines
  \put(0,20){\line(1,2){20}}
  \put(0,20){\line(1,-1){20}}
  \put(20,0){\line(0,1){60}}
  \put(20,0){\line(1,1){40}}
  \put(20,60){\line(2,-1){40}}

  \put(20,0){\circle*{1}}
  \put(20,60){\circle*{1}}
  \put(0,20){\circle*{1}}
  \put(60,40){\circle*{1}}

\end{picture}
\begin{picture}(40,66)(0,0)
\put(20,30){\makebox(0,0){$\longrightarrow$}}
\end{picture}
\begin{picture}(60,66)(0,0)
\thicklines
  \put(0,20){\line(1,2){20}}
  \put(0,20){\line(1,-1){20}}
  \put(0,20){\line(3,1){60}}
  \put(20,0){\line(1,1){40}}
  \put(20,60){\line(2,-1){40}}

  \put(20,0){\circle*{1}}
  \put(20,60){\circle*{1}}
  \put(0,20){\circle*{1}}
  \put(60,40){\circle*{1}}

\end{picture}
\vspace{-15pt}
\end{center}
\caption{A flip of a diagonal in a quadrilateral.}
\label{fig:flip}
\end{figure}

\item
Each flip, say one involving the diagonals $ik$ and~$jl$, 
acts on extended clusters by
exchanging the Pl\"ucker coordinates $P_{ik}$ and~$P_{jl}$. 
This exchange can be viewed as a subtraction-free
transformation determined by the corresponding 
three-term 
relation~\eqref{eq:grassmann-plucker-3term}. 
\hfil \qed
\end{enumerate}

\begin{remark}
\label{rem:laurent-pos-gr2m}
In fact something stronger than Theorem
\ref{th:Grassmannian-TP-criterion} holds: 
every Pl\"ucker coordinate can be written as a 
\emph{Laurent polynomial} with 
\emph{positive coefficients}
in the  Pl\"ucker coordinates from~$\mathbf{\tilde x}(T)$. 
This is an instance of very general phenomena of 
Laurentness and positivity in cluster algebras,
which will be discussed later in the book. 
\end{remark}

\pagebreak[3]

The combinatorics of flips is captured by 
the graph whose vertices are labeled by the triangulations of the
polygon~$\mathbf{P}_m$ and whose edges correspond to flips.
Each vertex of this graph
has degree~$m-3$.  
Moreover, this graph
 is the $1$-skeleton of an $(m-3)$-dimensional convex polytope 
(discovered by J.~Stasheff \cite{stasheff}) 
called the \emph{associahedron}. 
See Figure~\ref{fig:A3assoc_poly}.

\begin{figure}[ht!]
\begin{center}
\raisebox{-0.7mm}{\includegraphics[width=100mm]{a3.eps}}
\end{center}
\caption{The 3-dimensional associahedron.  Each $2$-dimensional face corresponds to a 
	diagonal $d$ in a hexagon; the vertices of that face are labeled by 
	the triangulations containing $d$.}
\label{fig:A3assoc_poly}
\end{figure}

In the forthcoming terminology of cluster algebras, 
this graph is an example of an \emph{exchange graph}.  
Its vertices correspond to 
extended clusters (all of which have the same cardinality) 
while its edges correspond to \emph{exchange
  relations}~\eqref{eq:grassmann-plucker-3term}:
adjacent extended clusters are related to each other 
by interchanging the cluster variables appearing on the left-hand
side of an exchange relation.

Recall that the Pl\"ucker ring $R_{2,m}$ is generated by the Pl\"ucker
coordinates~$P_{ij}$ subject to the three-term
relations~\eqref{eq:grassmann-plucker-3term}.
The combinatorics of extended clusters provides an elegant way to
construct a linear basis for this ring. 
Define a \emph{cluster monomial} to be 
a monomial in (i.e., a product of a multiset of) 
cluster and/or frozen variables, all of which belong to one extended cluster; 
in other words, the corresponding collection of arcs 
does not contain two  diagonals which cross each other.
These monomials appeared already in the classical $19^{\rm th}$ century
literature on invariant theory; 
in particular, it is known \cite{kungrota,stur} that the set
of all cluster monomials  form a
linear basis of~$R_{2,m}$. 
Later on we shall discuss a far-reaching generalization of this
result in the context of cluster algebras.

As mentioned above, the ideas we have discussed
for the Grassmannian $\Gr_{2,m}$ can be generalized
to a beautiful theory that works for 
arbitrary Grassmannians $\Gr_{k,m}$, {\it cf}.\ Chapter
\ref{ch:Grassmannians}. 
In Chapter \ref{ch:surfaces}, we shall describe a generalization of the 
$\Gr_{2,m}$ example in a different direction,
which involves the combinatorics of flips for triangulations of a
Riemann surface with boundary and punctures. 
That construction has an intrinsic interpretation 
in hyperbolic geometry where an analogue of the Ptolemy relation
holds for the exponentiated hyperbolic distances between horocycles drawn
around vertices of a polygon with geodesic sides and cusps at the
vertices (the ``Penner coordinates'' on the corresponding
\emph{decorated Teichm\"uller space}~\cite{penner}).

\section{The basic affine space}
\label{sec:baseaffine}

We next turn our attention to total positivity criteria for square 
matrices. 
In view of Lemma~\ref{lem:cryer}, it makes sense to study
lower- and upper-triangular matrices first, and then proceed to the whole
group~$\operatorname{SL}_n$. 
We choose a related but different strategy:
first study matrices for which a certain subset of ``flag minors'' are
positive, then treat the general case. 

\begin{definition}
The \emph{flag minor} $P_J$ labeled by a nonempty proper subset
$J \subsetneq \{1,\dots,n\}$ 
is a function on the special linear group $G=\SL_n$
defined by
\[
P_J: z=(z_{ij}) \mapsto \det(z_{ij} \ \vert \ i \leq |J|, 
  j \in J).
\]
Thus $P_J(z)$ is a $|J| \times |J|$ minor of $z$ that 
occupies the top $|J|$ rows and the columns in $J$.
Since an $n$-element set has $2^n-2$ proper nonempty subsets,
a matrix in $\SL_n$ has $2^n-2$ flag minors.
\end{definition}

Let $U \subset G$ be the subgroup of 
unipotent lower-triangular matrices, i.e. the lower-triangular
matrices with $1$'s on the diagonal. 
The group $U$ acts on~$G$ 
by multiplication on the left.
It consequently acts on the coordinate ring~$\CC[G]$, that is,
the ring of polynomials in the matrix entries of a generic matrix of
determinant~1. 
It is easy to see that each flag minor $P_J$ is an invariant of this action: 
for any $z\in G$ and $y\in U$, we have $P_J(yz)=P_J(z)$. 
Similarly to the case of Pl\"ucker coordinates in a Pl\"ucker ring,
we have 
(thanks to the appropriate versions of the First and Second Fundamental Theorems
of invariant theory): 
\begin{enumerate}
\item
the flag minors generate the ring $\CC[G]^U$ 
of $U$-invariant polynomials in the matrix entries, and 
\item
the ideal of relations among the flag minors is generated by 
certain quadratic \emph{generalized Pl\"ucker relations}.
\end{enumerate}
The ring $\CC[G]^U$ plays an important role in the representation theory
of the semisimple Lie group~$G$:
it naturally carries all irreducible representations of~$G$, each with
multiplicity~1. (We will not rely on this in what follows.) 
This ring is the coordinate ring of the
\emph{basic affine space}, the (Geometric Invariant Theory) 
quotient~$U\backslash G$.  
This space is also known as the \emph{base affine space}, the 
\emph{fundamental affine space}, and the \emph{principal 
affine space}.
In this section, this space plays the role analogous to the role that
the Grassmannians played in Section~\ref{sec:Ptolemy}. 
As~before, we can state everything we need in an elementary fashion,
in terms of matrices and their~minors.

\begin{definition}
An element $z\in G$ is \emph{flag totally positive} (FTP) if all 
flag minors $P_J$ take positive values at~$z$.
\end{definition}

We would like to detect flag total positivity in an efficient way, 
by testing as few of the $2^n-2$ flag minors as possible. 
It turns out that the optimal test of this kind probes only
$\frac{(n-1)(n+2)}{2}$ flag minors.
We note that $\frac{(n-1)(n+2)}{2}=n^2-1-\binom{n}{2}$\
is the dimension of the basic affine space. 

Following \cite{bfz-tp}, we construct a family of tests for flag total positivity 
labeled by combinatorial objects called \emph{wiring diagrams}
which play the same role as triangulations did in Section~\ref{sec:Ptolemy}.
 
This notion is best explained by an
example such as the one in Figure~\ref{fig:wiring}.
A~wiring diagram consists of a family of $n$ piecewise-straight 
lines, considered up to isotopy, which can be viewed as graphs of $n$ continuous
piecewise-linear functions defined on the same interval. 
The lines are labeled $1,\dots,n$ as shown in Figure~\ref{fig:wiring}.  
The key requirement is that each pair of lines intersects exactly once. 

The notion of a wiring diagram is closely related to that of a \emph{reduced word}
for an element of maximal length in the symmetric group~$\mathcal{S}_n$. 
See \cite[Section~2.3]{bfz-tp} for detailed explanations. 

\begin{figure}[hb]
\setlength{\unitlength}{1.2pt}
\begin{center}
\vspace{-10pt}
\begin{picture}(180,45)(6,0)
\thicklines
\dark{

  \put(0,0){\line(1,0){40}}
  \put(60,0){\line(1,0){100}}
  \put(180,0){\line(1,0){10}}
  \put(0,20){\line(1,0){40}}
  \put(60,20){\line(1,0){10}}
  \put(90,20){\line(1,0){70}}
  \put(180,20){\line(1,0){10}}
  \put(0,40){\line(1,0){70}}
  \put(90,40){\line(1,0){100}}

  \put(40,0){\line(1,1){20}}
  \put(70,20){\line(1,1){20}}
  \put(160,0){\line(1,1){20}}

  \put(40,20){\line(1,-1){20}}
  \put(70,40){\line(1,-1){20}}
  \put(160,20){\line(1,-1){20}}
  \put(193,-4){$\mathbf{3}$}
  \put(193,16){$\mathbf{2}$}
  \put(193,36){$\mathbf{1}$}

  \put(-6,-4){$\mathbf{1}$}
  \put(-6,16){$\mathbf{2}$}
  \put(-6,36){$\mathbf{3}$}
}

\end{picture}
\vspace{-10pt}
\end{center}
\caption{A wiring diagram.}
\label{fig:wiring}
\end{figure}

A \emph{chamber} of a wiring diagram 
is a connected component of its complement inside the vertical strip; we 
usually don't include the top and bottom regions.
We label each chamber by a subset
of 
$[1,n]\!=\!\{1,\dots,n\}$ indicating which lines 
pass \emph{below} that chamber; 
cf.\ Figure~\ref{fig:chamber-sets1}.

\begin{figure}[ht]
\setlength{\unitlength}{1.2pt}
\begin{center}
\vspace{-10pt}
\begin{picture}(180,44)(6,1)
\thicklines
\dark{

  \put(0,0){\line(1,0){40}}
  \put(60,0){\line(1,0){100}}
  \put(180,0){\line(1,0){10}}
  \put(0,20){\line(1,0){40}}
  \put(60,20){\line(1,0){10}}
  \put(90,20){\line(1,0){70}}
  \put(180,20){\line(1,0){10}}
  \put(0,40){\line(1,0){70}}
  \put(90,40){\line(1,0){100}}

  \put(40,0){\line(1,1){20}}
  \put(70,20){\line(1,1){20}}
  \put(160,0){\line(1,1){20}}

  \put(40,20){\line(1,-1){20}}
  \put(70,40){\line(1,-1){20}}
  \put(160,20){\line(1,-1){20}}

  \put(193,-4){$\mathbf{3}$}
  \put(193,16){$\mathbf{2}$}
  \put(193,36){$\mathbf{1}$}

  \put(-6,-4){$\mathbf{1}$}
  \put(-6,16){$\mathbf{2}$}
  \put(-6,36){$\mathbf{3}$}
}

  \put(18,10){$_{\mathbf{\dark{1}}}$}
  \put(76,10){$_{\mathbf{\dark{2}}}$}
  \put(181,10){$_{\mathbf{\dark{3}}}$}

  \put(46,30){$_{\mathbf{\dark{12}}}$}
  \put(102,30){$_{\mathbf{\dark{23}}}$}

\end{picture}
\vspace{-10pt}
\end{center}
\caption{Chamber minors $P_1$, $P_2$, $P_3$, $P_{12}$, and $P_{23}$.}
\label{fig:chamber-sets1}
\end{figure}

Thus every chamber is naturally associated with a flag minor $P_J$, 
called a \emph{chamber minor},
that occupies the columns specified by the set~$J$ in the chamber, 
and the rows 
$1,2,\dots,|J|$.
The total number of chamber minors is always~$\frac{(n-1)(n+2)}{2}$.

The chamber minors associated to a wiring diagram
make up an \emph{extended cluster}. 
Such an extended cluster will always contain the $2n-2$ flag minors
associated to the unbounded chambers: 
\[
P_{1}, P_{1,2}, \dots, P_{1,2,\dots,n-1}
\text{\ and\ }  
P_{n}, P_{n-1,n}, \dots, P_{2,3,\dots,n};
\] 
these are the \emph{frozen variables}. 
The $\binom{n-1}{2}$ chamber minors 
associated with bounded chambers are the \emph{cluster variables};
they form the corresponding \emph{cluster}.

\begin{theorem}[\cite{bfz-tp}]
\label{th:diagram-TP-criterion}
Every flag minor 
can be written as a subtraction-free
rational expression in the chamber minors of any given wiring
diagram.
Thus, if these $\frac{(n-1)(n+2)}{2}$ chamber minors 
evaluate positively at a matrix $z\in\SL_n$, then $z$ is FTP. 
\end{theorem}

\begin{proof}
Theorem \ref{th:diagram-TP-criterion}
is implied by the following three facts:
\begin{enumerate}[leftmargin=.25in]
\item Each flag minor appears as a chamber minor in
some wiring diagram.

\item
Any two wiring diagrams can be
transformed into each other by a sequence of local \emph{braid moves}
of the form shown in Figure~\ref{fig:moves1}.

\item
Under each braid move, the
corresponding collections of chamber minors are obtained from each
other by exchanging the minors $Y$ and~$Z$
(cf.\ Figure~\ref{fig:moves1}), and these minors satisfy
the identity
\begin{equation}
\label{eq1:ac+bd=yz}
YZ = AC+BD \ .
\end{equation}
\end{enumerate}
Statement (1) is easily checked by direct inspection.
Statement (2) is a theorem of G.~Ringel~\cite{ringel}; 
it is also equivalent to a well known property 
of reduced words in the symmetric group (Tits' Lemma):
any two such words are related by braid moves.
Finally, formula \eqref{eq1:ac+bd=yz} is one of the 
aforementioned generalized Pl\"ucker relations; 
see \cref{ex:identity-Plucker} below.
\end{proof}

\begin{figure}[ht]
\setlength{\unitlength}{1pt}
\begin{center}
\vspace{-10pt}
\begin{picture}(120,45)(10,0)
\thicklines
\dark{

  \put(0,0){\line(1,0){10}}
  \put(30,0){\line(1,0){40}}
  \put(50,0){\line(1,0){10}}
  \put(90,0){\line(1,0){10}}
  \put(0,20){\line(1,0){10}}
  \put(30,20){\line(1,0){10}}
  \put(60,20){\line(1,0){10}}
  \put(90,20){\line(1,0){10}}
  \put(0,40){\line(1,0){40}}
  \put(60,40){\line(1,0){40}}

  \put(10,0){\line(1,1){20}}
  \put(40,20){\line(1,1){20}}
  \put(70,0){\line(1,1){20}}

  \put(10,20){\line(1,-1){20}}
  \put(40,40){\line(1,-1){20}}
  \put(70,20){\line(1,-1){20}}
}
  \put(115,24){\vector(1,0){10}}
  \put(125,16){\vector(-1,0){10}}

  \put(2,8){$A$}
  \put(46,8){$Y$}
  \put(89,8){$D$}

  \put(16,28){$B$}
  \put(76,28){$C$}

\end{picture}
\begin{picture}(100,40)(-5,0)
\thicklines
\dark{

  \put(0,40){\line(1,0){10}}
  \put(30,40){\line(1,0){40}}
  \put(50,40){\line(1,0){10}}
  \put(90,40){\line(1,0){10}}
  \put(0,20){\line(1,0){10}}
  \put(30,20){\line(1,0){10}}
  \put(60,20){\line(1,0){10}}
  \put(90,20){\line(1,0){10}}
  \put(0,0){\line(1,0){40}}
  \put(60,0){\line(1,0){40}}

  \put(10,20){\line(1,1){20}}
  \put(40,0){\line(1,1){20}}
  \put(70,20){\line(1,1){20}}

  \put(10,40){\line(1,-1){20}}
  \put(40,20){\line(1,-1){20}}
  \put(70,40){\line(1,-1){20}}
}

  \put(2,28){$B$}
  \put(46,28){$Z$}
  \put(89,28){$C$}

  \put(16,8){$A$}
  \put(76,8){$D$}

\end{picture}
\vspace{-10pt}
\end{center}
\caption{A braid move.}
\label{fig:moves1}
\end{figure}

To prove \eqref{eq1:ac+bd=yz}, we will need
Propositions~\ref{prop:Muir}--\ref{prop:Muir-flag}~below.

\begin{proposition}[Muir's Law of extensible minors~\cite{muir1960}]
\label{prop:Muir}
Suppose that $(I)$ is a polynomial identity involving minors $\Delta_{A,B}$ of a
generic matrix, which is homogeneous in that  every term in $(I)$
is a product of the same number of determinants.  Let $R$ and $C$
be finite sets of positive integers such that~$R$ (resp.,~$C$) is disjoint from
every row set $A$ (resp., every column set~$B$) appearing in a determinant in~$(I)$.
Then one can get a new identity $(I')$ from $I$ by replacing each minor
$\Delta_{A,B}$ by $\Delta_{A \cup R, B \cup C}$.
\end{proposition}

\begin{proposition}
\label{prop:Muir-flag}
Suppose $(I)$ is a polynomial identity involving flag minors
$P_{B}$ of a
generic matrix, such that every term in $(I)$  
is a product of the same number of flag minors.  Let $C$
be a finite set of positive integers disjoint from
every column set $B$ appearing  in a flag minor $P_B$ in~$(I)$.
 Then one
can get a new identity $(I^{''})$ from $(I)$ by replacing each term
$P_{B}$ by $P_{B \cup C}$.
\end{proposition}

\begin{proof}
Let $b = |B|$ and $c=|C|$.
Recall that the flag minor $P_B$ is equal to 
$\Delta_{\{1,2,\dots,b\},B}$.
Note that if the identity $(I)$ is true, then the identity $(I')$ is true, 
where $(I')$ is obtained from $(I)$ by replacing each term 
$\Delta_{\{1,2,\dots,b\},B}$ by the term
$\Delta_{\{c+1,c+2,\dots,c+b\},B}$.
But now if we apply Muir's Law with 
$R=\{1,2,\dots,c\}$ and $C$, then we get a new identity $(I^{''})$
from $(I')$ by replacing each term 
$\Delta_{\{c+1,c+2,\dots,c+b\},B}$ by 
$\Delta_{\{1,2,\dots,c+b\},B \cup C}= P_{B\cup C}$. 
\end{proof}

\begin{exercise}\label{ex:identity-Plucker}
Prove \eqref {eq1:ac+bd=yz}, thereby completing
the proof of
\cref{th:diagram-TP-criterion}.
\end{exercise}

\begin{remark}
\label{rem:laurent-pos-base}
Just as in the case of $\Gr_{2,m}$
(cf.\ Remark~\ref{rem:laurent-pos-gr2m}), something much 
stronger than Theorem \ref{th:diagram-TP-criterion} is true:
each flag minor can be written
as a \emph{Laurent polynomial with positive coefficients} in the chamber
minors of a given wiring diagram.
\end{remark}

Before moving on to our last example (total positivity for square matrices), 
let us pause to make 
a few observations  based on our study of total positivity in the
Grassmannian~$\Gr_{2,m}$, 
flag total positivity in $G=\SL_n$, 
and the related rings $R_{2,m}$ and~$\CC[G]^U$.
In both cases, we observed the following common features:
\begin{itemize}[leftmargin=.15in]
\item a family of distinguished generators of the 
ring (Pl\"ucker coordinates and flag minors, respectively);
\item a finite subset of \emph{frozen} generators;
\item a grouping of the generators into overlapping \emph{extended
  clusters} all of which have the same size;
each extended cluster contains the frozen generators; 
\item combinatorial data accompanying each extended cluster
(triangulations and wiring diagrams, respectively); 
\item \emph{exchange relations} that can be written using those data;
they lead to subtraction-free birational maps relating
extended clusters to each other; 
\item a ``mutation rule" for producing new combinatorial data 
from the current one (flips of triangulations and braid moves, respectively).
\end{itemize}

In the case of the Grassmannian $\Gr_{2,m}$, we defined a graph
whose vertices are indexed by the set of triangulations of $\mathbf{P}_m$, 
and whose edges correspond to flips.  
This graph is \emph{regular}, i.e., all its vertices have the same degree. 
Indeed, in any triangulation, we can flip any of the participating diagonals. 
Put another way, given an 
extended cluster and a cluster variable within it, there is a unique
way to construct a new extended cluster by replacing that cluster variable by 
another one.

We could construct an analogous graph related to the flag TP elements 
of $\SL_n$, with 
vertices corresponding to wiring diagrams, and edges corresponding to 
braid moves. However, this graph is not regular for $n \geq 4$.
The framework of cluster algebras will rectify this issue 
by providing a recipe for constructing the missing cluster variables and 
clusters.

Another important property that we observed for the Grassmannian 
$\Gr_{2,m}$ concerned cluster monomials. Recall that a cluster monomial is a product of 
cluster and frozen (=coefficient) variables (not necessarily distinct) all of which belong to 
one extended cluster. Cluster monomials form a linear
basis for the Pl\"ucker ring~$R_{2,m}\,$. 
Unfortunately the analogous statement does not hold for the ring~$\CC[G]^U$. 
However, the cluster monomials are 
still linearly independent, and hence can be included in an additive 
basis for the ring. 
Explicit constructions of such additive bases that possess ``nice''
properties (e.g., various versions of positivity) 
remain at the center of current research on cluster algebras. 

\vspace{-10pt}

\section{The general linear group}
\label{sec:matrices}

We now turn to total positivity criteria 
for the general linear group, or equivalently,
for square matrices of arbitrary size.
It turns out that to test whether a given $n \times n$
matrix is TP it suffices
to check the positivity of only $n^2$ special minors.

For an $n\times n$ matrix~$z$, let $\Delta_{I,J}(z)$ denote the minor
of $z$ determined by the row set~$I$ 
and the column set~$J$;
here $I$ and $J$ are nonempty
subsets of $[1,n] = \{1, \dots, n\}$ of the same cardinality.
Thus $z$ is TP if and only if $\Delta_{I,J}(z) > 0$
for all such $I$ and~$J$.

Following \cite{fz-dbc,fz-intel}
we construct a family of ``optimal" tests for total positivity,
labeled by combinatorial objects called
\emph{double wiring diagrams}. 
They generalize the wiring diagrams
we saw in the previous section.
A double wiring diagram 
is basically a superposition of two ordinary wiring diagrams, each
colored in its own color (`thin' or `thick'); 
see Figure~\ref{fig:double-wiring}.

\begin{figure}[ht]
\setlength{\unitlength}{1.2pt}
\begin{center}
\begin{picture}(180,45)(6,0)
\thicklines
\dark{

  \put(0,0){\line(1,0){40}}
  \put(60,0){\line(1,0){100}}
  \put(180,0){\line(1,0){10}}
  \put(0,20){\line(1,0){40}}
  \put(60,20){\line(1,0){10}}
  \put(90,20){\line(1,0){70}}
  \put(180,20){\line(1,0){10}}
  \put(0,40){\line(1,0){70}}
  \put(90,40){\line(1,0){100}}

  \put(40,0){\line(1,1){20}}
  \put(70,20){\line(1,1){20}}
  \put(160,0){\line(1,1){20}}

  \put(40,20){\line(1,-1){20}}
  \put(70,40){\line(1,-1){20}}
  \put(160,20){\line(1,-1){20}}
  \put(193,-6){$\mathbf{3}$}
  \put(193,14){$\mathbf{2}$}
  \put(193,34){$\mathbf{1}$}

  \put(-6,-6){$\mathbf{1}$}
  \put(-6,14){$\mathbf{2}$}
  \put(-6,34){$\mathbf{3}$}
}

\light{
\thinlines

  \put(0,2){\line(1,0){100}}
  \put(120,2){\line(1,0){70}}
  \put(0,22){\line(1,0){10}}
  \put(30,22){\line(1,0){70}}
  \put(120,22){\line(1,0){10}}
  \put(150,22){\line(1,0){40}}
  \put(0,42){\line(1,0){10}}
  \put(30,42){\line(1,0){100}}
  \put(150,42){\line(1,0){40}}

  \put(10,22){\line(1,1){20}}
  \put(100,2){\line(1,1){20}}
  \put(130,22){\line(1,1){20}}

  \put(10,42){\line(1,-1){20}}
  \put(100,22){\line(1,-1){20}}
  \put(130,42){\line(1,-1){20}}

  \put(193,2){${1}$}
  \put(193,22){${2}$}
  \put(193,42){${3}$}

  \put(-6,2){${3}$}
  \put(-6,22){${2}$}
  \put(-6,42){${1}$}
}
\end{picture}
\vspace{-10pt}
\end{center}
\caption{A double wiring diagram.}
\label{fig:double-wiring}
\end{figure}

The lines in a double wiring diagram are numbered separately within
each color.
(Note the difference in the numbering schemes for the two colors.) 
We then assign to every \emph{chamber} of a diagram a pair of subsets
of $[1,n]$: each subset indicates which lines of the corresponding
color pass below that chamber; see Figure~\ref{fig:chamber-sets}.
Thus every chamber is naturally associated with a minor $\Delta_{I,J}$
(again called a \emph{chamber minor})
that occupies the rows and columns of an $n\times n$ matrix 
specified by the sets $I$ and $J$ written into that chamber.
The total number of chamber minors in a given double wiring diagram is always~$n^2$.

\begin{figure}[ht]
\setlength{\unitlength}{1.2pt}
\begin{center}
\begin{picture}(180,50)(6,0)
\thicklines
\dark{

  \put(0,0){\line(1,0){40}}
  \put(60,0){\line(1,0){100}}
  \put(180,0){\line(1,0){10}}
  \put(0,20){\line(1,0){40}}
  \put(60,20){\line(1,0){10}}
  \put(90,20){\line(1,0){70}}
  \put(180,20){\line(1,0){10}}
  \put(0,40){\line(1,0){70}}
  \put(90,40){\line(1,0){100}}

  \put(40,0){\line(1,1){20}}
  \put(70,20){\line(1,1){20}}
  \put(160,0){\line(1,1){20}}

  \put(40,20){\line(1,-1){20}}
  \put(70,40){\line(1,-1){20}}
  \put(160,20){\line(1,-1){20}}

  \put(193,-6){$\mathbf{3}$}
  \put(193,14){$\mathbf{2}$}
  \put(193,34){$\mathbf{1}$}

  \put(-6,-6){$\mathbf{1}$}
  \put(-6,14){$\mathbf{2}$}
  \put(-6,34){$\mathbf{3}$}
}

\light{
\thinlines

  \put(0,2){\line(1,0){100}}
  \put(120,2){\line(1,0){70}}
  \put(0,22){\line(1,0){10}}
  \put(30,22){\line(1,0){70}}
  \put(120,22){\line(1,0){10}}
  \put(150,22){\line(1,0){40}}
  \put(0,42){\line(1,0){10}}
  \put(30,42){\line(1,0){100}}
  \put(150,42){\line(1,0){40}}

  \put(10,22){\line(1,1){20}}
  \put(100,2){\line(1,1){20}}
  \put(130,22){\line(1,1){20}}

  \put(10,42){\line(1,-1){20}}
  \put(100,22){\line(1,-1){20}}
  \put(130,42){\line(1,-1){20}}

  \put(193,2){${1}$}
  \put(193,22){${2}$}
  \put(193,42){${3}$}

  \put(-6,2){${3}$}
  \put(-6,22){${2}$}
  \put(-6,42){${1}$}

  \put(14,10){$_{\light{3},\mathbf{\dark{1}}}$}
  \put(76,10){$_{\light{3},\mathbf{\dark{2}}}$}
  \put(136,10){$_{\light{1},\mathbf{\dark{2}}}$}
  \put(181,10){$_{\light{1},\mathbf{\dark{3}}}$}

  \put(-1,30){$_{\light{23},\mathbf{\dark{12}}}$}
  \put(42,30){$_{\light{13},\mathbf{\dark{12}}}$}
  \put(102,30){$_{\light{13},\mathbf{\dark{23}}}$}
  \put(162,30){$_{\light{12},\mathbf{\dark{23}}}$}

  \put(84,50){$_{\light{123},\mathbf{\dark{123}}}$}
}
\end{picture}
\end{center}
\caption{
This double wiring diagram has $3^2=9$ chamber minors:
$\Delta_{3,1}\,$, $\Delta_{3,2}\,$, $\Delta_{1,2}\,$, $\Delta_{1,3}\,$,
$\Delta_{23,12}\,$, $\Delta_{13,12}\,$, $\Delta_{13,23}\,$,
$\Delta_{12,23}\,$, and~$\Delta_{123,123}$.}
\label{fig:chamber-sets}
\end{figure}

\begin{theorem}[\cite{fz-dbc}]
\label{th:double-diagram-TP-criterion}
Every minor of a 
square matrix can be written as a subtraction-free
rational expression in the chamber minors of a given double wiring
diagram.
Thus, if these $n^2$ chamber minors evaluate positively at a given
$n\times n$ matrix~$z$, 
then $z$ is totally positive.
\end{theorem}

By now the reader can  guess the strategy for proving this theorem.

\begin{proof}
Theorem \ref{th:double-diagram-TP-criterion}
is a consequence of the following facts:
\begin{enumerate}[leftmargin=.25in]
\item Every minor is a chamber minor for
some double wiring diagram.

\item
Any two double wiring diagrams are related to each other via a
sequence of \emph{local moves} of 
three different kinds, shown in Figure~\ref{fig:moves}.

\item
Under each local move, the corresponding collections of chamber minors
transform by exchanging the minors $Y$ and~$Z$, and these minors
satisfy the identity
\begin{equation}
\label{eq:ac+bd=yz}
YZ = AC+BD \ .
\end{equation}
\end{enumerate}
		(By convention, the ``chamber minor'' associated to the region
		at the very bottom of the double wiring diagram is equal to $1$.)

\begin{figure}[ht]
\setlength{\unitlength}{1pt}
\begin{center}
\begin{picture}(120,45)(10,0)
\thicklines
\dark{

  \put(0,0){\line(1,0){10}}
  \put(30,0){\line(1,0){40}}
  \put(50,0){\line(1,0){10}}
  \put(90,0){\line(1,0){10}}
  \put(0,20){\line(1,0){10}}
  \put(30,20){\line(1,0){10}}
  \put(60,20){\line(1,0){10}}
  \put(90,20){\line(1,0){10}}
  \put(0,40){\line(1,0){40}}
  \put(60,40){\line(1,0){40}}

  \put(10,0){\line(1,1){20}}
  \put(40,20){\line(1,1){20}}
  \put(70,0){\line(1,1){20}}

  \put(10,20){\line(1,-1){20}}
  \put(40,40){\line(1,-1){20}}
  \put(70,20){\line(1,-1){20}}
}
  \put(115,24){\vector(1,0){10}}
  \put(125,16){\vector(-1,0){10}}

\thinlines
\light{
  \put(0,2){\line(1,0){100}}
  \put(0,22){\line(1,0){100}}
  \put(0,42){\line(1,0){100}}
}

  \put(2,8){$A$}
  \put(46,8){$Y$}
  \put(89,8){$D$}

  \put(16,28){$B$}
  \put(76,28){$C$}

\end{picture}
\begin{picture}(100,40)(-5,0)
\thicklines
\dark{

  \put(0,40){\line(1,0){10}}
  \put(30,40){\line(1,0){40}}
  \put(50,40){\line(1,0){10}}
  \put(90,40){\line(1,0){10}}
  \put(0,20){\line(1,0){10}}
  \put(30,20){\line(1,0){10}}
  \put(60,20){\line(1,0){10}}
  \put(90,20){\line(1,0){10}}
  \put(0,0){\line(1,0){40}}
  \put(60,0){\line(1,0){40}}

  \put(10,20){\line(1,1){20}}
  \put(40,0){\line(1,1){20}}
  \put(70,20){\line(1,1){20}}

  \put(10,40){\line(1,-1){20}}
  \put(40,20){\line(1,-1){20}}
  \put(70,40){\line(1,-1){20}}
}

\thinlines
\light{
  \put(0,2){\line(1,0){100}}
  \put(0,22){\line(1,0){100}}
  \put(0,42){\line(1,0){100}}
}
  \put(2,28){$B$}
  \put(46,28){$Z$}
  \put(89,28){$C$}

  \put(16,8){$A$}
  \put(76,8){$D$}

\end{picture}
\end{center}
\begin{center}
\begin{picture}(120,60)(10,0)
\thinlines
\light{

  \put(0,2){\line(1,0){10}}
  \put(30,2){\line(1,0){40}}
  \put(50,2){\line(1,0){10}}
  \put(90,2){\line(1,0){10}}
  \put(0,22){\line(1,0){10}}
  \put(30,22){\line(1,0){10}}
  \put(60,22){\line(1,0){10}}
  \put(90,22){\line(1,0){10}}
  \put(0,42){\line(1,0){40}}
  \put(60,42){\line(1,0){40}}

  \put(10,2){\line(1,1){20}}
  \put(40,22){\line(1,1){20}}
  \put(70,2){\line(1,1){20}}

  \put(10,22){\line(1,-1){20}}
  \put(40,42){\line(1,-1){20}}
  \put(70,22){\line(1,-1){20}}
}
  \put(2,8){$A$}
  \put(46,8){$Y$}
  \put(89,8){$D$}

  \put(16,28){$B$}
  \put(76,28){$C$}

\thicklines
  \put(115,24){\vector(1,0){10}}
  \put(125,16){\vector(-1,0){10}}
\dark{
  \put(0,0){\line(1,0){100}}
  \put(0,20){\line(1,0){100}}
  \put(0,40){\line(1,0){100}}
}

\end{picture}
\begin{picture}(100,40)(-5,0)
\thinlines
\light{

  \put(0,42){\line(1,0){10}}
  \put(30,42){\line(1,0){40}}
  \put(50,42){\line(1,0){10}}
  \put(90,42){\line(1,0){10}}
  \put(0,22){\line(1,0){10}}
  \put(30,22){\line(1,0){10}}
  \put(60,22){\line(1,0){10}}
  \put(90,22){\line(1,0){10}}
  \put(0,2){\line(1,0){40}}
  \put(60,2){\line(1,0){40}}

  \put(10,22){\line(1,1){20}}
  \put(40,2){\line(1,1){20}}
  \put(70,22){\line(1,1){20}}

  \put(10,42){\line(1,-1){20}}
  \put(40,22){\line(1,-1){20}}
  \put(70,42){\line(1,-1){20}}
}

\thicklines
\dark{
  \put(0,0){\line(1,0){100}}
  \put(0,20){\line(1,0){100}}
  \put(0,40){\line(1,0){100}}
}
  \put(2,28){$B$}
  \put(46,28){$Z$}
  \put(89,28){$C$}

  \put(16,8){$A$}
  \put(76,8){$D$}

\end{picture}
\end{center}
\begin{center}
\begin{picture}(100,60)(5,-10)
\thicklines
\dark{

  \put(0,0){\line(1,0){10}}
  \put(30,0){\line(1,0){40}}

  \put(0,20){\line(1,0){10}}
  \put(30,20){\line(1,0){40}}

  \put(10,0){\line(1,1){20}}
  \put(10,20){\line(1,-1){20}}
}

\thinlines
\light{

  \put(0,2){\line(1,0){40}}
  \put(60,2){\line(1,0){10}}

  \put(0,22){\line(1,0){40}}
  \put(60,22){\line(1,0){10}}

  \put(40,2){\line(1,1){20}}
  \put(40,22){\line(1,-1){20}}
}
  \put(2,8){$A$}
  \put(32,8){$Y$}
  \put(62,8){$C$}

  \put(32,28){$B$}
  \put(32,-12){$D$}

\thicklines
  \put(100,15){\vector(1,0){10}}
  \put(110,7){\vector(-1,0){10}}

\end{picture}
\begin{picture}(100,60)(-30,-10)
\thicklines
\dark{
  \put(0,0){\line(1,0){40}}
  \put(60,0){\line(1,0){10}}

  \put(0,20){\line(1,0){40}}
  \put(60,20){\line(1,0){10}}

  \put(40,0){\line(1,1){20}}
  \put(40,20){\line(1,-1){20}}
}

\thinlines
\light{

  \put(0,2){\line(1,0){10}}
  \put(30,2){\line(1,0){40}}

  \put(0,22){\line(1,0){10}}
  \put(30,22){\line(1,0){40}}

  \put(10,2){\line(1,1){20}}
  \put(10,22){\line(1,-1){20}}
}
  \put(2,8){$A$}
  \put(32,8){$Z$}
  \put(62,8){$C$}

  \put(32,28){$B$}
  \put(32,-12){$D$}

\end{picture}
\end{center}
\caption{Local ``moves.''}
\label{fig:moves}
\end{figure}

Statements (1) and (2) can be easily derived from their counterparts
for ordinary wiring diagrams, which appeared in the proof of Theorem
\ref{th:diagram-TP-criterion}. 
Each instance of the relation \eqref{eq:ac+bd=yz} is a well known
determinantal identity; the readers may enjoy finding its proof on
their own. 
The identities corresponding to the top two cases in
Figure~\ref{fig:moves} are nothing but the three-term 
relation \eqref{eq1:ac+bd=yz} which we discussed earlier;
the third one is sometimes called the ``Lewis Carroll identity,''
due to the role it plays in C.~L.~Dodgson's condensation method
\cite[pp.\ 170--180]{dodgson}. 
All of these identities were proved by P.~Desnanot as early as in 1819,
see \cite[pp.\ 140--142]{muir1960}.
\end{proof}

\begin{exercise}
\label{exercise:solid-minors}
A minor $\Delta_{I,J}$ is called \emph{solid}
if both $I$ and $J$ consist of  consecutive indices. 
It is easy to see that 
an $n\times n$ matrix~$z$ has $n^2$ solid minors $\Delta_{I,J}$ 
such that $I\cup J$ contains~$1$
(see Figure~\ref{fig:gasca}).   
Show that $z$ is TP if and only if all these $n^2$ minors are positive. 
\end{exercise}

\begin{figure}[ht]
\setlength{\unitlength}{1pt} 

\begin{center}
\begin{picture}(80,80)(0,3)
\thicklines

  \put(0,0){\line(1,0){80}}
  \put(0,0){\line(0,1){80}}
  \put(80,0){\line(0,1){80}}
  \put(0,80){\line(1,0){80}}

\dark{
  \put(40,50){\line(1,0){28}}
  \put(40,50){\line(0,1){28}}
  \put(40,78){\line(1,0){28}}
  \put(68,50){\line(0,1){28}}
}

\light{
  \put( 2,10){\line(1,0){58}}
  \put( 2,10){\line(0,1){58}}
  \put( 2,68){\line(1,0){58}}
  \put(60,10){\line(0,1){58}}
}
\end{picture}
\end{center}
\caption{
Solid minors $\Delta_{I,J}$ with $1\in I\cup J$. 
}
\label{fig:gasca}

\end{figure}

\begin{remark}
\label{rem:laurent-pos-gln}
Similarly to Remarks \ref{rem:laurent-pos-gr2m} and~\ref{rem:laurent-pos-base}, 
Theorem~\ref{th:double-diagram-TP-criterion} can be strengthened as
follows: 
every minor of a 
square matrix can be written as a \emph{Laurent polynomial with 
positive coefficients}
in the chamber minors of a given double wiring diagram.  
\end{remark}

The algebraic/combinatorial construction described above possesses the
same key features that we identified in the two previous settings. 
The collections of chamber minors associated to double wiring
diagrams are again examples of \emph{extended clusters} in the future
cluster algebra setup.
As noted above, all these extended clusters have the same cardinality~$n^2$.
Each of them contains the $2n-1$ minors of the form
$\Delta_{[1,p],[n-p+1,n]}$ and/or $\Delta_{[n-p+1,n],[1,p]}$ (for $p \in [1,n]$),
which correspond to unbounded chambers.
These $2n-1$ minors are the \emph{frozen variables}.
Removing them from an extended cluster, we obtain a
\emph{cluster} consisting of $(n-1)^2$ \emph{cluster
  variables},
the chamber minors associated with the bounded chambers.
A ``mutation'' of one of the three kinds depicted in
Figure~\ref{fig:moves} replaces a single cluster variable in a cluster by a new
one; the product of these two cluster variables (minors) appears on
the left-hand side of the corresponding \emph{exchange
  relation}~\eqref{eq:ac+bd=yz}. 

For $n = 3$, there are $34$ clusters corresponding to double wiring diagrams.
They are shown in Figure~\ref{fig:schemes} 
as vertices of a graph
whose edges correspond to local moves.
Looking closely at this graph, we
see that it is not regular: of the $34$ vertices, $18$ have
degree~$4$, and $16$ have degree~$3$.
Thus, for each of the $16$ clusters corresponding to vertices of degree~$3$,
there is one minor that cannot be exchanged from this cluster to
form another cluster.
This ``irregularity" can be repaired using 
two additional polynomials in the matrix entries,
see Exercise~\ref{exercise:KL}.

\begin{figure}[ht!]
\setlength{\unitlength}{1.45pt}
\begin{center}
\begin{picture}(200,322)(-100,-164)

\put(0,-165){\line(2,3){110}}
\put(0,-165){\line(-2,3){110}}
\put(0,-165){\line(0,1){55}}

\put(0,-110){\line(0,1){30}}
\put(0,-110){\line(2,1){40}}
\put(0,-110){\line(-2,1){40}}

\put(40,-90){\line(2,3){40}}
\put(40,-90){\line(0,1){30}}
\put(40,-90){\line(-2,1){40}}

\put(0,-80){\line(-2,1){40}}
\put(0,-80){\line(2,1){40}}

\put(-40,-90){\line(2,1){40}}
\put(-40,-90){\line(0,1){30}}
\put(-40,-90){\line(-2,3){40}}

\put(0,-70){\line(0,1){30}}

\put(40,-60){\line(-2,1){40}}
\put(-40,-60){\line(2,1){40}}
\put(40,-60){\line(0,1){30}}
\put(-40,-60){\line(0,1){30}}

\put(0,-40){\line(0,1){20}}

\put(0,-20){\line(1,1){20}}
\put(0,-20){\line(-1,1){20}}
\put(40,-30){\line(2,3){20}}
\put(40,-30){\line(-2,3){20}}
\put(80,-30){\line(1,1){30}}
\put(80,-30){\line(-2,3){20}}
\put(-40,-30){\line(2,3){20}}
\put(-40,-30){\line(-2,3){20}}
\put(-80,-30){\line(2,3){20}}
\put(-80,-30){\line(-1,1){30}}

\put(0,165){\line(2,-3){110}}
\put(0,165){\line(-2,-3){110}}
\put(0,165){\line(0,-1){55}}

\put(0,110){\line(0,-1){30}}
\put(0,110){\line(2,-1){40}}
\put(0,110){\line(-2,-1){40}}

\put(40,90){\line(2,-3){40}}
\put(40,90){\line(0,-1){30}}
\put(40,90){\line(-2,-1){40}}

\put(0,80){\line(-2,-1){40}}
\put(0,80){\line(2,-1){40}}

\put(-40,90){\line(2,-1){40}}
\put(-40,90){\line(0,-1){30}}
\put(-40,90){\line(-2,-3){40}}

\put(0,70){\line(0,-1){30}}

\put(40,60){\line(-2,-1){40}}
\put(-40,60){\line(2,-1){40}}
\put(40,60){\line(0,-1){30}}
\put(-40,60){\line(0,-1){30}}

\put(0,40){\line(0,-1){20}}

\put(0,20){\line(1,-1){20}}
\put(0,20){\line(-1,-1){20}}
\put(40,30){\line(2,-3){20}}
\put(40,30){\line(-2,-3){20}}
\put(80,30){\line(1,-1){30}}
\put(80,30){\line(-2,-3){20}}
\put(-40,30){\line(2,-3){20}}
\put(-40,30){\line(-2,-3){20}}
\put(-80,30){\line(2,-3){20}}
\put(-80,30){\line(-1,-1){30}}

{\small
\put(3,-22){$abcG$}
\put(-17,-1){$acFG$}
\put(-37,-32){$ceFG$}
\put(-34,-61){$cdeG$}
\put(2,-40){$bcdG$}
\put(42,-58){$bdfG$}
\put(12,-31){$bfEG$}
\put(23,-1){$abEG$}
\put(-9,-88){$defG$}
\put(2,-69){$bcdA$}
\put(-34,-92){$cdeA$}
\put(2,-114){$defA$}
\put(15,-93){$bdf\!A$}
\put(3,-166){$efgA$}
\put(-107,-1){$egAB$}
\put(-77,-32){$ceAB$}
\put(-57,-1){$ceBF$}
\put(64,1){$bfCE$}
\put(53,-33){$bfAC$}
\put(80,-6){$fgAC$}

\put(2,20){$aEFG$}

\put(-77,29){$egBF$}
\put(-37,29){$acBF$}
\put(14,29){$abCE$}
\put(51,28){$fgCE$}

\put(3,37){$aDEF$}

\put(-60,49){$aBDF$}
\put(28,50){$aCDE$}

\put(-13.5,62){$gD\,EF$}

\put(-13,82){$aB\,CD$}

\put(-34,89){$gBDF$}
\put(8.5,90.5){$gCDE$}

\put(2,112){$gBCD$}
\put(2,165){$gABC$}

\put(0,-20){\circle*{3}}
\put(-20,0){\circle*{3}}
\put(-40,-30){\circle*{3}}
\put(-40,-60){\circle*{3}}
\put(0,-40){\circle*{3}}
\put(40,-60){\circle*{3}}
\put(40,-30){\circle*{3}}
\put(20,0){\circle*{3}}
\put(0,-80){\circle*{3}}
\put(0,-70){\circle*{3}}
\put(-40,-90){\circle*{3}}
\put(0,-110){\circle*{3}}
\put(40,-90){\circle*{3}}
\put(0,-165){\circle*{3}}
\put(-110,0){\circle*{3}}
\put(-80,-30){\circle*{3}}
\put(-60,0){\circle*{3}}
\put(60,0){\circle*{3}}
\put(80,-30){\circle*{3}}
\put(110,0){\circle*{3}}

\put(0,20){\circle*{3}}

\put(-80,30){\circle*{3}}
\put(-40,30){\circle*{3}}
\put(40,30){\circle*{3}}
\put(80,30){\circle*{3}}

\put(0,40){\circle*{3}}

\put(-40,60){\circle*{3}}
\put(40,60){\circle*{3}}

\put(0,70){\circle*{3}}

\put(0,80){\circle*{3}}

\put(-40,90) {\circle*{3}}
\put(40,90){\circle*{3}}

\put(0,110){\circle*{3}}
\put(0,165){\circle*{3}}

{\small

\put(-123,-128){
$\begin{array}{ll}
a=z_{11} \\[.05in]
b=z_{12} \\[.05in]
c=z_{21} \\[.05in]
d=z_{22} \\[.05in]
e=z_{23} \\[.05in]
f=z_{32} \\[.05in]
g=z_{33}
\end{array}
$}

\put(60,-128){
$\begin{array}{ll}
A=\Delta_{23,23}\\[.05in]
B=\Delta_{23,13}\\[.05in]
C=\Delta_{13,23}\\[.05in]
D=\Delta_{13,13}\\[.05in]
E=\Delta_{13,12}\\[.05in]
F=\Delta_{12,13}\\[.05in]
G=\Delta_{12,12}
\end{array}
$}

} 

} 

\end{picture}
\vspace{-.05in}
\end{center}
\caption{
Total positivity tests for a $3\times 3$ matrix $z\!=\!(z_{ij})$.
Each test checks $9$ minors.
The frozen minors 
$z_{31}, z_{13}, \Delta_{23,12}(z), \Delta_{12,23}(z), \det(z)$ 
are common to all tests.
The remaining 4 minors form a cluster shown near the corresponding vertex.
To illustrate, the test derived from
Figure~\ref{fig:chamber-sets} involves the cluster
$bfCE=\{z_{32}, z_{12}, \Delta_{13,12}, \Delta_{13,23}\}$.
The edges of the graph correspond to local moves.}
\label{fig:schemes}
\end{figure}

\begin{exercise}
\label{exercise:KL}
For a $3\times 3$ matrix $z=(z_{ij})$, let
\begin{align}
\label{eq:K(z)}
K(z) &= z_{33} \Delta_{12,12}(z) - \det(z), \\
\label{eq:L(z)}
L(z) &= z_{11} \Delta_{23,23}(z)  - \det(z).
\end{align}
Use $K$ and $L$ 
to add $16$ more clusters to the graph in
Figure~\ref{fig:schemes}.
The resulting graph will be regular of degree~$4$.
As an example, the cluster $\{e,f,g,A\}$ at the bottom of 
Figure~\ref{fig:schemes}
will be joined to the new cluster $\{e,f,g,K\}$ by an edge corresponding
to a new exchange relation
\begin{equation}
\label{eq:new-exchange}
\Delta_{23,23} K = \Delta_{12,23} \Delta_{23,12} z_{33} + \det(z)
z_{23} z_{32} \,.
\end{equation}
This construction will yield 16 additional TP tests for $3\times 3$
matrices. 
\end{exercise}

The theory of cluster algebras, to be developed in subsequent
chapters, will unify the three examples we have treated here, 
and will provide a systematic way to produce the ``missing''
clusters and exchange relations, thereby 
generating a large class of new total positivity tests.

\chapter{Mutations of quivers and matrices}
\label{ch:combinatorics-of-mutations}

\vspace{-.4in}

In this chapter we discuss
mutations of quivers and of skew-symmetrizable matrices. 
These constructions 
lie at the heart of the combinatorial framework
underlying the general theory of cluster algebras. 

Quivers (or more generally, skew-symmetrizable matrices) are
the combinatorial data which accompany (extended) clusters 
and determine exchange relations between them. 
The notion of mutation generalizes many examples of ``local" transformations of combinatorial
objects, including those discussed in Chapter~\ref{ch:tp-examples}: 
flips in triangulations, braid moves in wiring diagrams, etc.  

In some guise, quiver mutation appeared in the work of theoretical
physicists 
(cf.\ \cite{cecotti-vafa, seiberg} and the discussion in
\cite[Section~6]{feng-hanany-he-uranga}) 
several years before its discovery by mathematicians~\cite{ca1}. 
However, the systematic study of the combinatorics of mutations 
has only begun with the advent of cluster~algebras.

\section{Quiver mutation}\label{sec:quivermutation}

\begin{definition}
\label{def:quiver}
A \emph{quiver} is a finite oriented graph, consisting of 
	vertices and directed edges (called \emph{arrows}).
We allow multiple edges, but we disallow loops
(i.e., an arrow may not connect a vertex to itself)
and oriented 2-cycles
(i.e., no arrows of opposite orientation may connect the same pair of
vertices).
A quiver does not have to be connected. 

In what follows, we will need a slightly richer notion, with some
vertices in a quiver designated as \emph{frozen}.
The remaining vertices are called \emph{mutable}.
We will always assume that there are no arrows between pairs of frozen vertices.
(Such arrows would make no difference in the future construction of
a cluster algebra associated with a~quiver.)
\end{definition}

The terminology in Definition~\ref{def:quiver} anticipates
the role that quiver muta\-tions play in the forthcoming definition of a cluster
algebra; wherein the ver\-tices of a quiver are labeled by 
the elements of an extended cluster, 
so that the frozen vertices correspond to
frozen variables, and the mutable
vertices to the cluster variables. 
In this chapter, all of this remains in the~background. 

\pagebreak[3]

\begin{definition}
Let $k$ be a mutable vertex in a quiver~$Q$.
The \emph{quiver mutation} $\mu_k$ transforms $Q$ into a new
quiver~$Q'=\mu_k(Q)$ via a sequence of three steps:
\begin{enumerate}[leftmargin=.25in]
\item For each oriented two-arrow path  $i\to k\to j$,
add a new arrow $i\to j$ (unless both
$i$ and~$j$ are frozen, in which case do nothing).
\item Reverse the direction of all arrows incident
to the vertex~$k$.
\item 
Repeatedly remove oriented 2-cycles
until unable to do so. 
\end{enumerate}
An example is given in  Figure~\ref{fig:quiver-mutation}.
\end{definition}

\begin{figure}[ht]
\begin{center}
\setlength{\unitlength}{1.8pt}
\begin{picture}(40,20)(0,0)

\put( 0,0){\makebox(0,0){$a$}}
\put(20,1){\makebox(0,0){$b$}}

\put(0,19.5){\makebox(0,0){$q$}}
\multiput(-3,17)(6,0){2}{\line(0,1){6}}
\multiput(-3,17)(0,6){2}{\line(1,0){6}}
\put(20,21){\makebox(0,0){$k$}}
\put(40,20){\makebox(0,0){$r$}}
\multiput(37,17)(6,0){2}{\line(0,1){6}}
\multiput(37,17)(0,6){2}{\line(1,0){6}}

\thicklines
\put(16,0){\vector(-1,0){12}}
\put(16,20){\vector(-1,0){11}}
\put(35,21){\vector(-1,0){12}}
\put(35,18.5){\vector(-1,0){12}}

\put(20,16.5){\vector(0,-1){12}}
\put(0,3.5){\vector(0,1){11}}

\put(3,3){\vector(1,1){14}}

\end{picture}
\begin{picture}(40,20)(0,0)
\put(20,10){\makebox(0,0){$\stackrel{\displaystyle\mu_k}{\longmapsto}$}}
\end{picture}
\begin{picture}(40,20)(0,0)

\put( 0,0){\makebox(0,0){$a$}}
\put(20,1){\makebox(0,0){$b$}}

\put(0,19.5){\makebox(0,0){$q$}}
\multiput(-3,17)(6,0){2}{\line(0,1){6}}
\multiput(-3,17)(0,6){2}{\line(1,0){6}}

\put(20,21){\makebox(0,0){$k$}}
\put(40,20){\makebox(0,0){$r$}}
\multiput(37,17)(6,0){2}{\line(0,1){6}}
\multiput(37,17)(0,6){2}{\line(1,0){6}}

\thicklines

\put(6,20){\vector(1,0){10}}
\put(23,21){\vector(1,0){12}}
\put(23,19){\vector(1,0){12}}

\put(-1.3,3.5){\vector(0,1){11}}
\put(1.3,3.5){\vector(0,1){11}}
\put(20,5){\vector(0,1){12}}

\put(17,17){\vector(-1,-1){14}}
\put(35.2,16.8){\vector(-1,-1){13}}
\put(37,15){\vector(-1,-1){13}}

\end{picture}
\end{center}
\caption{A quiver mutation~$\mu_k$. Vertices $q$ and $r$ are frozen.
Step~1 adds arrows $a\to b$, $a\to q$, and two arrows $r\to b$.
Step~2 reverses five arrows connecting~$k$ to $a,b,q,r$.
Step~3 removes the arrows $a\to b$ and $b\to a$.}
\label{fig:quiver-mutation}
\end{figure}

\begin{remark}
\label{rem:mut-source-sink}
If a vertex $k$ of a quiver is a sink or a source,
then mutation at $k$ reverses the orientations of all 
arrows incident to~$k$, and does nothing else.  This operation was first considered
in the context of quiver representation theory (the 
reflection functors of Bernstein-Gelfand-Ponomarev~\cite{BGP}). 
\end{remark}

We next formulate some  simple but important properties of
quiver mutation. 

\begin{exercise}
\label{ex:mut-simple}
Verify the following properties of quiver mutation.
\begin{enumerate}[leftmargin=.25in]
\item
Mutation is an involution: 
$\mu_{k}(\mu_k(Q))=Q$.  
\item
Mutation commutes with the simultaneous reversal of orientations of
all arrows of a quiver. 
\item
\label{ex:mut-commute}
Let $k$ and $\ell$ be two mutable vertices such that there is no arrow from $k$ to $\ell$, nor from $\ell$ to $k$.
Then mutations at $k$ and~$\ell$ commute with each other:
$\mu_{\ell}(\mu_k(Q))=\mu_{k}(\mu_{\ell}(Q))$.
\end{enumerate}
\end{exercise}

In particular, mutations in different
connected components of a quiver do not interact with each other. 

\pagebreak[3]

\section{Triangulations of polygons}
\label{sec:triangulations}

Triangulations of polygons  were  discussed
in Section~\ref{sec:Ptolemy}
in the context of studying total positivity in the Grassmannian
of $2$-planes in $m$-space. 

We now associate a quiver to each triangulation of a convex 
$m$-gon~$\mathbf{P}_m$
and explain how flips of such triangulations correspond to quiver mutations.

\begin{definition}
\label{def:Q(T)-polygon}
Let $T$ be a triangulation of the polygon $\mathbf{P}_m$ by pairwise
noncrossing diagonals. 
The quiver~$Q(T)$  associated to $T$
is defined as follows.
The frozen vertices of $Q(T)$ are labeled by the sides of~$\mathbf{P}_m$,
and the mutable vertices of $Q(T)$ are labeled by the diagonals of~$T$.
If~two diagonals, or a diagonal and a boundary segment, belong to the same triangle,
we connect the corresponding vertices in $Q(T)$ by an arrow
whose orientation is determined by the clockwise orientation of the
boundary of the triangle. See Figure~\ref{fig:quiver-triangulation}.
\end{definition}

\begin{figure}[ht]
\begin{center}
\vspace{-10pt}
\setlength{\unitlength}{3pt}
\begin{picture}(60,60)(0,0)
\thinlines
  \multiput(0,20)(60,0){2}{\line(0,1){20}}
  \multiput(20,0)(0,60){2}{\line(1,0){20}}
  \multiput(0,40)(40,-40){2}{\line(1,1){20}}
  \multiput(20,0)(40,40){2}{\line(-1,1){20}}
 
  \multiput(20,0)(20,0){2}{\circle*{1}}
  \multiput(20,60)(20,0){2}{\circle*{1}}
  \multiput(0,20)(0,20){2}{\circle*{1}}
  \multiput(60,20)(0,20){2}{\circle*{1}}
 
\put(40,0){\line(1,2){20}}
\put(0,40){\line(1,0){60}}
\put(0,20){\line(2,-1){40}}
\put(0,40){\line(1,-1){40}}
\put(20,60){\line(2,-1){40}}

\thicklines

\multiput(-1.5,28.5)(3,0){2}{\line(0,1){3}}
\multiput(-1.5,28.5)(0,3){2}{\line(1,0){3}}

\multiput(58.5,28.5)(3,0){2}{\line(0,1){3}}
\multiput(58.5,28.5)(0,3){2}{\line(1,0){3}}

\multiput(28.5,-1.5)(3,0){2}{\line(0,1){3}}
\multiput(28.5,-1.5)(0,3){2}{\line(1,0){3}}

\multiput(28.5,58.5)(3,0){2}{\line(0,1){3}}
\multiput(28.5,58.5)(0,3){2}{\line(1,0){3}}

\multiput(8.5,8.5)(3,0){2}{\line(0,1){3}}
\multiput(8.5,8.5)(0,3){2}{\line(1,0){3}}

\multiput(48.5,8.5)(3,0){2}{\line(0,1){3}}
\multiput(48.5,8.5)(0,3){2}{\line(1,0){3}}

\multiput(8.5,48.5)(3,0){2}{\line(0,1){3}}
\multiput(8.5,48.5)(0,3){2}{\line(1,0){3}}

\multiput(48.5,48.5)(3,0){2}{\line(0,1){3}}
\multiput(48.5,48.5)(0,3){2}{\line(1,0){3}}

{\put(20,10){\circle{3}}}
{\put(20,20){\circle{3}}}
{\put(30,40){\circle{3}}}
{\put(40,50){\circle{3}}}
{\put(50,20){\circle{3}}}

\linethickness{1.5pt}

\put(12,10){{\vector(1,0){6}}}
\put(22,8){{\vector(1,-1){6}}}

\put(2,29){{\vector(2,-1){16}}}
\put(20,18){{\vector(0,-1){6}}}
\put(18,12){{\vector(-1,1){16}}}

\put(21,22){{\vector(1,2){8}}}
\put(32,38){{\vector(1,-1){16}}}
\put(48,20){{\vector(-1,0){26}}}

\put(28,41){{\vector(-2,1){16}}}
\put(12,50){{\vector(1,0){26}}}
\put(38,48){{\vector(-1,-1){6}}}

\put(38,52){{\vector(-1,1){6}}}
\put(48,50){{\vector(-1,0){6}}}

\put(50,12){{\vector(0,1){6}}}
\put(52,22){{\vector(1,1){6}}}

\end{picture}
\vspace{-10pt}
\end{center}

\caption{The quiver $Q(T)$ associated to a triangulation $T$ of an octagon.}
\label{fig:quiver-triangulation}
\end{figure}

\begin{exercise}
\label{ex:flip-is-mutation}
Let $T$ be a triangulation of $\mathbf{P}_m$ as above. 
Let $T'$ be the triangulation obtained from $T$
by flipping a diagonal~$\gamma$. 
Verify that the quiver $Q(T')$ is obtained from $Q(T)$ by mutating at the
vertex labeled~by~$\gamma$. 
\end{exercise}

The construction of Definition~\ref{def:Q(T)-polygon} can be
generalized to triangulations of more general oriented surfaces with
boundary and punctures; this will be discussed in
Chapter~\ref{ch:surfaces}. 
Another generalization~\cite{fock-goncharov-ihes} 
was developed in the study of cluster structures arising in \emph{higher
  Teichm\"uller theory}; 
a very special case is described in the exercise below. 

\begin{exercise} 
\label{exercise:FG-sl3-d4}
To each triangulation $T$ of a convex polygon by noncrossing diagonals, 
let us associate a quiver $Q_3(T)$ as follows.
Place two mutable vertices of $Q_3(T)$ on each diagonal of~$T$,
two frozen vertices on each side of the polygon, 
and one mutable vertex in the interior of each triangle of~$T$. 
For a triangle in~$T$, 
let $A_1, A_2, B_1, B_2, C_1, C_2$ be the vertices of $Q_3(T)$ lying
on the boundary of the triangle, listed clockwise
so that $A_1$ and~$A_2$ (resp., $B_1$ and~$B_2$, $C_1$ and~$C_2$) lie
on the same side of the triangle. 
Let $K$ denote the vertex of $Q_3(T)$ lying inside the triangle.
Draw the arrows $A_1\to K\to B_2\to C_1\to K\to A_2\to B_1\to K\to
C_2\to A_1$. 
Doing so for each triangle of~$T$,
and removing the arrows between frozen vertices, yields the quiver~$Q_3(T)$.
Show that if $T$ and $T'$ are connected by a flip, then 
$Q_3(T)$ and $Q_3(T')$ are connected by a sequence of mutations.
See Figure \ref{fig:Q_3}.
\end{exercise}

\begin{figure}[ht]
\begin{center}
\vspace{-10pt}
\setlength{\unitlength}{2.8pt}
\begin{picture}(60,60)(0,0)
\thinlines
\put(0,0){\line(1,0){60}}
\put(0,60){\line(1,0){60}}
\put(0,0){\line(0,1){60}}
\put(60,0){\line(0,1){60}}
\put(0,60){\line(1,-1){60}}

\put(0,0){\circle*{1}}
\put(60,0){\circle*{1}}
\put(0,60){\circle*{1}}
\put(60,60){\circle*{1}}

\thicklines

\multiput(18.5,-1.5)(3,0){2}{\line(0,1){3}}
\multiput(18.5,-1.5)(0,3){2}{\line(1,0){3}}

\multiput(38.5,-1.5)(3,0){2}{\line(0,1){3}}
\multiput(38.5,-1.5)(0,3){2}{\line(1,0){3}}

\multiput(18.5,58.5)(3,0){2}{\line(0,1){3}}
\multiput(18.5,58.5)(0,3){2}{\line(1,0){3}}

\multiput(38.5,58.5)(3,0){2}{\line(0,1){3}}
\multiput(38.5,58.5)(0,3){2}{\line(1,0){3}}

\multiput(-1.5,18.5)(3,0){2}{\line(0,1){3}}
\multiput(-1.5,18.5)(0,3){2}{\line(1,0){3}}

\multiput(-1.5,38.5)(3,0){2}{\line(0,1){3}}
\multiput(-1.5,38.5)(0,3){2}{\line(1,0){3}}

\multiput(58.5,18.5)(3,0){2}{\line(0,1){3}}
\multiput(58.5,18.5)(0,3){2}{\line(1,0){3}}

\multiput(58.5,38.5)(3,0){2}{\line(0,1){3}}
\multiput(58.5,38.5)(0,3){2}{\line(1,0){3}}

{\put(20,20){\circle{3}}}
{\put(20,40){\circle{3}}}
{\put(40,40){\circle{3}}}
{\put(40,20){\circle{3}}}

\linethickness{1.5pt}

\put(2,40){{\vector(1,0){16}}}
\put(2,20){{\vector(1,0){16}}}
\put(22,20){{\vector(1,0){16}}}

\put(20,38){{\vector(0,-1){16}}}
\put(40,18){{\vector(0,-1){16}}}
\put(20,18){{\vector(0,-1){16}}}

\put(38,2){{\vector(-1,1){16}}}
\put(18,22){{\vector(-1,1){16}}}

\put(20,42){{\vector(0,1){16}}}
\put(40,42){{\vector(0,1){16}}}
\put(40,22){{\vector(0,1){16}}}

\put(38,40){{\vector(-1,0){16}}}
\put(58,40){{\vector(-1,0){16}}}
\put(58,20){{\vector(-1,0){16}}}

\put(22,58){{\vector(1,-1){16}}}
\put(42,38){{\vector(1,-1){16}}}

\end{picture}
\vspace{-10pt}
\end{center}

\caption{The quiver $Q_3(T)$ associated to a triangulation $T$ of a 
quadrilateral.}
\label{fig:Q_3}
\end{figure}

\section{Wiring diagrams}
\label{sec:mut-wiring}

Wiring diagrams were introduced 
in Section~\ref{sec:baseaffine}
in the context of studying total positivity in basic affine spaces. 
We also explained how to label each chamber of a wiring 
diagram by a subset of $[1,n]$, cf.\ Figure \ref{fig:chamber-sets1}.

We now associate a quiver to each wiring diagram and demonstrate that
braid moves in wiring diagrams translate into quiver mutations.  

The \emph{left end} (resp., \emph{right end}) of a chamber is a
crossing point of two wires located at the very left (resp., right)
of the chamber. 
Each bounded chamber has two ends; an unbounded chamber has one. 

\begin{definition}
\label{def:quiverwd}
The quiver $Q(D)$ associated to a wiring diagram~$D$ is defined as
follows. 
The vertices of $Q(D)$ are labeled by the chambers~of~$D$. 
The bounded chambers 
correspond to mutable vertices; 
the unbounded chambers correspond to frozen vertices.
Let $c$ and $c'$ be two chambers, at least one of which is bounded. 
Then there is an arrow $c\to c'$ in $Q(D)$ if and only if 
one of the following conditions is met: 
\begin{itemize}[leftmargin=.3in]
\item[(i)] the right end of~$c$ coincides with the left end of~$c'$;
\item[(ii)] the left end of $c$ lies directly above~$c'$, 
and the right end of~$c'$ lies directly below~$c$; 
\item[(iii)] the left end of~$c$ lies directly  below~$c'$,
and the right end of $c'$ lies directly above~$c$.
\end{itemize}
\end{definition}

An example is shown in Figure~\ref{fig:chamber-sets0}.

\begin{figure}[ht]
\setlength{\unitlength}{1.4pt}
\begin{center}
\vspace{-10pt}
\begin{picture}(210,60)(10,-10)
\thicklines
\dark{

  \put(10,0){\line(1,0){30}}
  \put(60,0){\line(1,0){100}}
  \put(180,0){\line(1,0){30}}
  \put(10,20){\line(1,0){30}}
  \put(60,20){\line(1,0){10}}
  \put(90,20){\line(1,0){70}}
  \put(180,20){\line(1,0){30}}
  \put(10,40){\line(1,0){60}}
  \put(90,40){\line(1,0){120}}

  \put(40,0){\line(1,1){20}}
  \put(70,20){\line(1,1){20}}
  \put(160,0){\line(1,1){20}}

  \put(40,20){\line(1,-1){20}}
  \put(70,40){\line(1,-1){20}}
  \put(160,20){\line(1,-1){20}}

  \put(213,-4){$\mathbf{3}$}
  \put(213,16){$\mathbf{2}$}
  \put(213,36){$\mathbf{1}$}

  \put(2,-4){$\mathbf{1}$}
  \put(2,16){$\mathbf{2}$}
  \put(2,36){$\mathbf{3}$}
}

  \put(18,10){$_{\mathbf{\dark{1}}}$}
  \put(76,10){$_{\mathbf{\dark{2}}}$}
  \put(195,10){$_{\mathbf{\dark{3}}}$}

  \put(46,30){$_{\mathbf{\dark{12}}}$}
  \put(102,30){$_{\mathbf{\dark{23}}}$}

\linethickness{1.5pt}

\put(24,10){\vector(1,0){50}}
\put(82, 10){\vector(1,0){111}}
\put(75, 15){\vector(-3,2){20}}
\put(101,27){\vector(-3,-2){20}}

\end{picture}
\vspace{-.1in}
\end{center}
\caption{A quiver associated with a wiring diagram. 
All vertices but~\scriptsize{$\mathbf{2}$} are frozen 
	(vertex $\mathbf{2}$ corresponds to the only bounded chamber).  
Consequently the quiver does not include the
arrow~{\scriptsize $\mathbf{{12}}\to\mathbf{{23}}$} because 
	{\scriptsize $\mathbf{{12}}$} and 
	{\scriptsize $\mathbf{{23}}$} are both 
	frozen/ unbounded.
}
\label{fig:chamber-sets0}
\end{figure}

The somewhat technical construction 
of Definition~\ref{def:quiverwd} is justified by the
fact that braid moves on wiring diagrams translate into mutations of
associated quivers: 

\begin{proposition} 
\label{pr:braid-mut}
Let $D$ and~$D'$ be wiring diagrams related
by a braid move at chamber~$Y$ (cf.\ Figure~\ref{fig:moves1}).  
Then $Q(D') = \mu_Y(Q(D))$.
\end{proposition}

We leave the proof of Proposition~\ref{pr:braid-mut} as an exercise for the reader. 

\begin{exercise}
Draw the wiring diagrams corresponding to the quivers in
Figure~\ref{fig:two-quivers}. 
Verify that these wiring diagrams are related by a braid move, 
and that the quivers are related by a quiver mutation. 
\end{exercise}

\begin{figure}[ht]
\begin{center}
\setlength{\unitlength}{2.5pt} 
\begin{picture}(60,20)(0,7) 
\thinlines

\put( 0,10){\makebox(0,0){$\mathbf{_{1}}$}}
\put(20,10){\makebox(0,0){$\mathbf{_{2}}$}}
\put(40,10){\makebox(0,0){$\mathbf{_{3}}$}}
\put(60,10){\makebox(0,0){$\mathbf{_{4}}$}}

\put(10,20){\makebox(0,0){$\mathbf{_{12}}$}}
\put(30,20){\makebox(0,0){$\mathbf{_{23}}$}}
\put(50,20){\makebox(0,0){$\mathbf{_{34}}$}}

\put(20,30){\makebox(0,0){$\mathbf{_{123}}$}}
\put(42,30){\makebox(0,0){$\mathbf{_{234}}$}}

\thicklines 

\put(4,10){\vector(1,0){12}}
\put(24,10){\vector(1,0){12}}
\put(44,10){\vector(1,0){12}}

\put(14,20){\vector(1,0){12}}
\put(34,20){\vector(1,0){12}}

\put(28,18){\vector(-1,-1){6}}
\put(48,18){\vector(-1,-1){6}}
\put(38,28){\vector(-1,-1){6}}

\put(17,12){\vector(-1,1){5}}
\put(37,12){\vector(-1,1){5}}
\put(27,22){\vector(-1,1){5}}

\end{picture}
\qquad\qquad
\begin{picture}(60,20)(0,7) 

\put( 0,10){\makebox(0,0){$\mathbf{_{1}}$}}
\put(20,10){\makebox(0,0){$\mathbf{_{13}}$}}
\put(40,10){\makebox(0,0){$\mathbf{_{3}}$}}
\put(60,10){\makebox(0,0){$\mathbf{_{4}}$}}

\put(10,20){\makebox(0,0){$\mathbf{_{12}}$}}
\put(30,20){\makebox(0,0){$\mathbf{_{23}}$}}
\put(50,20){\makebox(0,0){$\mathbf{_{34}}$}}

\put(20,30){\makebox(0,0){$\mathbf{_{123}}$}}
\put(42,30){\makebox(0,0){$\mathbf{_{234}}$}}

\thicklines 

\put(16,10){\vector(-1,0){12}}
\put(36,10){\vector(-1,0){12}}
\put(44,10){\vector(1,0){12}}

\put(34,20){\vector(1,0){12}}

\put(21,12){\vector(1,1){6}}
\put(48,18){\vector(-1,-1){6}}
\put(37,28){\vector(-1,-1){6}}

\put(12,17){\vector(1,-1){5}}
\put(27,22){\vector(-1,1){5}}

\qbezier(37,8)(20,2)(3,8)
\put(35.5,7.6){\vector(3,1){2}}

\end{picture}
\vspace{-10pt}
\end{center}
\caption{Quivers for two wiring diagrams related by a braid move. 
}
\label{fig:two-quivers}
\end{figure}

\begin{remark}\label{rem:reduceddecomp}
We note that the wiring diagrams introduced
in \cref{sec:baseaffine} can be identified with reduced decompositions
of the longest permutation $w_0$ of the symmetric group.  One 
can extend the
 notion of wiring diagram to the setting of 
decompositions (not necessarily reduced) of an arbitrary element
of the symmetric group.  The quiver $Q(D)$ and the correspondence
between braid moves and mutations also make sense in this setting.
\end{remark}

\section{Double wiring diagrams}
\label{sec:mut-double-wiring}

We next extend the constructions of Section~\ref{sec:mut-wiring}
to the double wiring diagrams discussed in
Section~\ref{sec:matrices}. 

Recall that each chamber of a double wiring 
diagram~$D$ is labeled by a pair of subsets of~$[1,n]$, 
cf.\ Figure~\ref{fig:chamber-sets}. 
Similarly to the case of ordinary wiring diagrams, 
each chamber of~$D$ 
has either one or two ``ends," and each end is either
``thick" or ``thin" (formed by two thick lines or two thin 
lines).

\begin{definition}
\label{def:quiverdwd}
The quiver $Q(D)$ associated with a double wiring diagram~$D$ is defined as
follows. 
The vertices of $Q(D)$ are labeled by the chambers of~$D$. 
The bounded chambers 
correspond to mutable vertices; 
the unbounded chambers correspond to frozen vertices.
Let $c$ and $c'$ be two chambers, at least one of which is bounded. 
Then there is an arrow $c\to c'$ in $Q(D)$ if and only if 
one of the following conditions is met
(cf.\ Figure~\ref{fig:chamber-quiver2}): 
\begin{itemize}[leftmargin=.3in]
\item[(i)] the right (resp., left) end of~$c$ is thick (resp., thin), 
and coincides with the left (resp., right) end of~$c'$;
\item[(ii)] the left end of $c'$ is thin, the right end of $c'$ is thick,
and the entire chamber $c'$ lies directly above or directly below~$c$;
\item[(iii)] the left end of $c$ is thick, the right end of $c$ is thin,
and the entire chamber $c$  lies directly above or directly below~$c'$;
\item[(iv)] the left (resp., right) end of $c'$ is above $c$ 
and the right (resp., left) end of $c$ is below $c'$ 
and both of these ends are thin (resp., thick);
\item[(v)] the left (resp., right) end of $c$ is above $c'$ 
and the right (resp., left) end of $c'$ is below $c$ 
and both of these ends are thick (resp., thin). 
\end{itemize}
\end{definition}

\begin{figure}[ht]
\begin{center}
\setlength{\unitlength}{1.4pt}
\begin{picture}(180,60)(6,-10)
\thicklines
\dark{
\linethickness{1.8pt}
  \put(0,0){\line(1,0){40}}
  \put(60,0){\line(1,0){100}}
  \put(180,0){\line(1,0){10}}
  \put(0,20){\line(1,0){40}}
  \put(60,20){\line(1,0){10}}
  \put(90,20){\line(1,0){70}}
  \put(180,20){\line(1,0){10}}
  \put(0,40){\line(1,0){70}}
  \put(90,40){\line(1,0){100}}

  \put(40,0){\line(1,1){20}}
  \put(70,20){\line(1,1){20}}
  \put(160,0){\line(1,1){20}}

  \put(40,20){\line(1,-1){20}}
  \put(70,40){\line(1,-1){20}}
  \put(160,20){\line(1,-1){20}}

  \put(193,-6){$\mathbf{3}$}
  \put(193,14){$\mathbf{2}$}
  \put(193,34){$\mathbf{1}$}

  \put(-6,-6){$\mathbf{1}$}
  \put(-6,14){$\mathbf{2}$}
  \put(-6,34){$\mathbf{3}$}
}

\light{
\thinlines

  \put(0,2){\line(1,0){100}}
  \put(120,2){\line(1,0){70}}
  \put(0,22){\line(1,0){10}}
  \put(30,22){\line(1,0){70}}
  \put(120,22){\line(1,0){10}}
  \put(150,22){\line(1,0){40}}
  \put(0,42){\line(1,0){10}}
  \put(30,42){\line(1,0){100}}
  \put(150,42){\line(1,0){40}}

  \put(10,22){\line(1,1){20}}
  \put(100,2){\line(1,1){20}}
  \put(130,22){\line(1,1){20}}

  \put(10,42){\line(1,-1){20}}
  \put(100,22){\line(1,-1){20}}
  \put(130,42){\line(1,-1){20}}

  \put(193,2){${1}$}
  \put(193,22){${2}$}
  \put(193,42){${3}$}

  \put(-6,2){${3}$}
  \put(-6,22){${2}$}
  \put(-6,42){${1}$}

  \put(14,10){$_{\light{3},\mathbf{\dark{1}}}$}
  \put(76,10){$_{\light{3},\mathbf{\dark{2}}}$}
  \put(136,10){$_{\light{1},\mathbf{\dark{2}}}$}
  \put(181,10){$_{\light{1},\mathbf{\dark{3}}}$}

  \put(-1,30){$_{\light{23},\mathbf{\dark{12}}}$}
  \put(42,30){$_{\light{13},\mathbf{\dark{12}}}$}
  \put(102,30){$_{\light{13},\mathbf{\dark{23}}}$}
  \put(162,30){$_{\light{12},\mathbf{\dark{23}}}$}

  \put(84,50){$_{\light{123},\mathbf{\dark{123}}}$}
}
\end{picture}
\end{center}
\begin{center}
\setlength{\unitlength}{1.4pt}
\begin{picture}(180,40)(6,15)
\thicklines
\put(84,50){\vector(-2,-1){30}}
\put(40,30){\vector(-1,0){22}}
\put(60,30){\vector(1,0){40}}
\put(108,34){\vector(-1,2){6}}
\put(112,26){\vector(3,-2){22}}
\put(162,30){\vector(-1,0){40}}
\put(24,10){\vector(1,0){50}}
\put(76,14){\vector(-3,2){18}}
\put(134,10){\vector(-1,0){44}}
\put(148,10){\vector(1,0){32}}
\light{
\thinlines

  \put(14,10){$_{\light{3},\mathbf{\dark{1}}}$}
  \put(76,10){$_{\light{3},\mathbf{\dark{2}}}$}
  \put(136,10){$_{\light{1},\mathbf{\dark{2}}}$}
  \put(181,10){$_{\light{1},\mathbf{\dark{3}}}$}

  \put(-1,30){$_{\light{23},\mathbf{\dark{12}}}$}
  \put(42,30){$_{\light{13},\mathbf{\dark{12}}}$}
  \put(102,30){$_{\light{13},\mathbf{\dark{23}}}$}
  \put(162,30){$_{\light{12},\mathbf{\dark{23}}}$}

  \put(84,50){$_{\light{123},\mathbf{\dark{123}}}$}
}
\end{picture}
\end{center}
\caption{A double wiring diagram~$D$ and the corresponding quiver~$Q(D)$.}
\label{fig:chamber-quiver2}
\end{figure}

\begin{remark}
One can check that the quiver $Q(D)$ defined as above
depends only on the isotopy type of the double wiring diagram~$D$. 
\end{remark}

As before, local moves translate into quiver mutations: 

\begin{proposition}
Suppose that double wiring diagrams $D$ and $D'$ are related
by a local move (cf.\ Figure~\ref{fig:moves})
at chamber~$Y$. Then 
$Q(D') = \mu_Y(Q(D))$.
\end{proposition}

The verification of this statement, which generalizes
Proposition~\ref{pr:braid-mut}, is left as an exercise for the reader.  

\begin{remark}
While the constructions of the quiver associated to a (double) 
wiring diagram given in \cref{def:quiverwd} and
\cref{def:quiverdwd} appear a bit complicated, we will explain
in Chapter~\ref{ch:plabic}
 how these constructions can be viewed as a special case
of the quiver associated to a planar bicolored graph in a disk.
\cref{urban} below will give a first introduction to quivers
associated to planar bipartite graphs.
\end{remark}

\pagebreak[3]

\section{Spider moves}\label{urban}

The spider move~\cite{kpw} is an operation on bipartite graphs which 
arises in several different contexts including 
statistical mechanics 
~\cite{goncharov-kenyon}, 
gauge theory
(Seiberg duality action on brane tilings~\cite{franco-hanany-et-al}), 
and total positivity (square moves in 
Postnikov's 
plabic 
graphs~\cite{postnikov}).  

We will give spider moves a more thorough treatment in 
Chapter~\ref{ch:plabic}.  
In Chapter \ref{ch:Grassmannians}, 
it will play an important role in the study of 
cluster structures on Grassmannians.

\begin{definition}\label{def:graphquiver}
Let $G$ be a  connected planar bipartite graph, 
properly embedded in a disk, and considered up to isotopy.
More precisely, we require the following: 
\begin{itemize}[leftmargin=.15in]
\item
each vertex in $G$ is colored either white or black
and lies either in the interior of the disk or on
its boundary; 
\item
each edge in $G$ connects two vertices of different colors,
and is represented by a simple curve whose interior 
is disjoint from the other edges and from the boundary;
\item
the closure of each \emph{face} (i.e, a connected component of the complement of~$G$)
is simply connected;
\item 
 each internal vertex has degree at least~$2$;
\item 
each boundary vertex has degree~$1$.
		(This condition 
		could be lifted but we keep it here for the sake of simplicity.)
\end{itemize}
To such a bipartite planar graph~$G$, we associate a quiver~$Q(G)$ as follows.
The vertices of $Q(G)$ are labeled by the faces of~$G$. 
A vertex of $Q(G)$ is frozen if the corresponding face is 
incident to the boundary of the disk, and is mutable otherwise. 
For each edge $e$ in~$G$, 
we introduce an arrow 
connecting the (distinct) faces separated by~$e$;
this arrow is oriented so that it 
``sees'' the white endpoint of~$e$ to the left and the black endpoint to the 
right as it crosses over~$e$, see Figure~\ref{fig:arrow}.
We then remove oriented $2$-cycles from the resulting quiver, 
one by one, to get~$Q(G)$. 
\end{definition}

\begin{figure}[htbp] 
\begin{center}
\setlength{\unitlength}{1.1pt} 
\begin{picture}(30,17)(0,0) 
\thicklines
\put(2,10){\line(1,0){29}} 
\put(0,10){\circle{4}}
\put(30,10){\circle*{4}}
\linethickness{1.5pt}
\put(15,1){\vector(0,1){20}} 
\end{picture} 
\vspace{-8pt}
\end{center}
\vspace{-.1in} 
\caption{Constructing a quiver associated to a bipartite graph.} 
\label{fig:arrow} 
\end{figure} 

We assume that the quiver $Q(G)$ is connected. 

\begin{definition}
Given a bivalent vertex $v$ adjacent to two internal vertices,
the \emph{contraction move} 
contracts the two edges incident to $v$.
	The reverse \emph{decontraction move}
 ``decontracts'' a vertex $v'$ into three vertices of alternating 
	colors, 
	see \cref{fig:contraction}. 
 \begin{figure}[h]
        \centering
	 \resizebox{2.5in}{!}{
        \begin{tikzpicture}
        \node[circle, draw = black, scale = 0.75, fill,
		label={90:$v$}] (A) at (0,1) {};
        \node[circle, draw = black, scale = 0.75,
		] (B) at (-1,1) {};
        \node[circle, draw = black, scale = 0.75,
		] (C) at (1,1) {};

        \draw[thick] (B)--(A)--(C);
        \draw[thick] (-1.73205, 2)--(B)--(-1.73205,0);
        \draw[thick] (2,1)--(C);
        \draw[thick] (1.73205, 2)--(C)--(1.73205,0); 
   
        \draw[>=triangle 45, <->] (2.5,1)--(3.5,1);

        \node[circle, draw = black, scale = 0.75,
		label={180:$v'$}] (D) at (4.5,1) {};

        \draw[thick] (4.5-.73205, 2)--(D)--(4.5-.73205,0);
        \draw[thick] (D)--(5.5,1);
        \draw[thick] (4.5+.73205, 2)--(D)--(4.5+.73205,0); 
    \end{tikzpicture}
	}
\caption{The contraction/ decontraction move.
	(There is a similar move with all colors reversed.)
	\label{fig:contraction} }
    \end{figure}
\end{definition}

\begin{definition}
\label{def:urban-renewal}
The \emph{spider move} is a local transformation of 
a bipartite graph  as in 
\cref{def:graphquiver}
	which contains a quadrilateral
face, whose vertices alternate in color and have  degree at least $3$. The move switches the colors of the four vertices of the quadrilateral and
 adds four ``legs,'' as shown in \cref{fig:urban}.

 \begin{figure}[h]
        \centering
	 \resizebox{4in}{!}{
        \begin{tikzpicture}
		\node[circle, draw = black, scale = 0.75] (A) at (-2,2) {};
        \node[circle, draw = black, scale = 0.75, fill] (B) at (2,2) {};
        \node[circle, draw = black, scale = 0.75] (C) at (2,-2) {};
        \node[circle, draw = black, scale = 0.75,fill] (D) at (-2,-2) {};
	\draw[thick, red] (A)--(B)--(C)--(D)--(A);
        \draw[thick] (-2.1,3)--(A)--(-3,2.1);
        \draw[thick] (-2.8, 2.8)--(A);
        \draw[thick] (B)--(2.8,2.8);
        \draw[thick] (2.1,-3)--(C)--(3,-2.1);
        \draw[thick] (D)--(-2.8,-2.8);
   
        \draw[>=triangle 45, ->] (3,0)--(4,0);
    \end{tikzpicture}
        \begin{tikzpicture}
		\node[circle, draw = black, scale = 0.75] (A) at (-2,2) {};
        \node[circle, draw = black, scale = 0.75, fill] (B) at (2,2) {};
        \node[circle, draw = black, scale = 0.75] (C) at (2,-2) {};
        \node[circle, draw = black, scale = 0.75,fill] (D) at (-2,-2) {};
	\node[circle, draw = black, scale = 0.75, fill] (E) at (-1.5,1.5) {};
        \node[circle, draw = black, scale = 0.75] (F) at (1.5,1.5) {};
        \node[circle, draw = black, scale = 0.75, fill] (G) at (1.5,-1.5) {};
        \node[circle, draw = black, scale = 0.75] (H) at (-1.5,-1.5) {};
	\draw[thick, red] (E)--(F)--(G)--(H)--(E);
		\draw[thick, red] (A)--(E);
		\draw[thick, red] (B)--(F);
		\draw[thick, red] (C)--(G);
		\draw[thick, red] (D)--(H);
        \draw[thick] (-2.1,3)--(A)--(-3,2.1);
        \draw[thick] (-2.8, 2.8)--(A);
        \draw[thick] (B)--(2.8,2.8);
        \draw[thick] (2.1,-3)--(C)--(3,-2.1);
        \draw[thick] (D)--(-2.8,-2.8);
   
        \draw[>=triangle 45, ->] (3,0)--(4,0);
    \end{tikzpicture}
        \begin{tikzpicture}
		\node[circle, draw = black, scale = 0.75] (A) at (-2,2) {};
        \node[circle, draw = black, scale = 0.75, fill] (B) at (2,2) {};
        \node[circle, draw = black, scale = 0.75] (C) at (2,-2) {};
        \node[circle, draw = black, scale = 0.75,fill] (D) at (-2,-2) {};
	\node[circle, draw = black, scale = 0.75, fill] (E) at (-1.5,1.5) {};
        \node[circle, draw = black, scale = 0.75] (F) at (1.5,1.5) {};
        \node[circle, draw = black, scale = 0.75, fill] (G) at (1.5,-1.5) {};
        \node[circle, draw = black, scale = 0.75] (H) at (-1.5,-1.5) {};
		\node[circle, draw = black, scale = 0.75] (I) at (-1,1) {};
        \node[circle, draw = black, scale = 0.75, fill] (J) at (1,1) {};
        \node[circle, draw = black, scale = 0.75] (K) at (1,-1) {};
        \node[circle, draw = black, scale = 0.75,fill] (L) at (-1,-1) {};
	\draw[thick, red] (I)--(J)--(K)--(L)--(I);
		\draw[thick, red] (A)--(E);
		\draw[thick, red] (E)--(I);
		\draw[thick, red] (B)--(F);
		\draw[thick, red] (F)--(J);
		\draw[thick, red] (C)--(G);
		\draw[thick, red] (G)--(K);
		\draw[thick, red] (D)--(H);
		\draw[thick, red] (H)--(L);
        \draw[thick] (-2.1,3)--(A)--(-3,2.1);
        \draw[thick] (-2.8, 2.8)--(A);
        \draw[thick] (B)--(2.8,2.8);
        \draw[thick] (2.1,-3)--(C)--(3,-2.1);
        \draw[thick] (D)--(-2.8,-2.8);
    \end{tikzpicture}
	}
\caption{Two spider moves, performed in 
succession. Note that 
the third graph is related to the first
graph via the contraction move. In this sense  the
spider move is an involution.}
\label{fig:urban} 
    \end{figure}
\end{definition}

\begin{exercise}
\label{ex:urban} 
Verify that if bipartite graphs are related by a 
contraction-decontraction move, the corresponding quivers are identical.
Then verify that if bipartite graphs as above 
are related via the spider move, 
the corresponding quivers are related by a mutation. 
\end{exercise}

\begin{remark} 
It is possible to work with an alternative to the 
spider move, which is shown in 
\cref{fig:spider2}.  In order to apply this move
we require that the two internal black vertices
have degree exactly three.
	\begin{figure}[h]
	 \resizebox{2.2in}{!}{
        \begin{tikzpicture}
        \node[circle, draw = black, scale = 0.5] (A) at (0,0) {};
        \node[circle, draw = black, scale = 0.5] (B) at (2,0) {};
        \node[circle, draw = black, scale = 0.5] (C) at (0,2) {};
        \node[circle, draw = black, scale = 0.5] (D) at (2,2) {};
        \node[circle, fill, draw = black, scale = 0.5] (E) at (0.5, 0.5) {};
        \node[circle, fill, draw = black, scale = 0.5] (F) at (1.5, 1.5) {};

        \draw (A)--(E)--(B)--(F)--(C)--(E);
        \draw (F)--(D);
        \draw (0, -0.3)--(A)--(-0.3, 0); 
        \draw (2, -0.3)--(B)--(2.3, 0); 
        \draw (-0.3, 2)--(C)--(0, 2.3);
        \draw (2.3, 2)--(D)--(2, 2.3);

        \draw[>=triangle 45, <->] (2.5,1)--(3.5,1);

        \node[circle, draw = black, scale = 0.5] (A) at (4,0) {};
        \node[circle, draw = black, scale = 0.5] (B) at (6,0) {};
        \node[circle, draw = black, scale = 0.5] (C) at (4,2) {};
        \node[circle, draw = black, scale = 0.5] (D) at (6,2) {};
        \node[circle, fill, draw = black, scale = 0.5] (E) at (5.5, 0.5) {};
        \node[circle, fill, draw = black, scale = 0.5] (F) at (4.5, 1.5) {};

        \draw (A)--(E)--(D)--(F)--(A);
        \draw (F)--(C);
        \draw (E)--(B);
        \draw (4, -0.3)--(A)--(3.7, 0); 
        \draw (6, -0.3)--(B)--(6.3, 0); 
        \draw (3.7, 2)--(C)--(4, 2.3);
        \draw (6.3, 2)--(D)--(6, 2.3);
    \end{tikzpicture}
		}
		\caption{An alternative version of the
		spider move.  
	(There is  a similar move with all colors reversed.)
		\label{fig:spider2}}
	\end{figure}
\end{remark}

Definition~\ref{def:urban-renewal} can be generalized 
to bipartite graphs properly embedded into an oriented surface.

\section{Mutation equivalence}\label{sec:mutationequivalence}

\begin{definition}
\label{def:mut-equiv-quivers}
Two quivers $Q$ and $Q'$ are called \emph{mutation equivalent}
if $Q$ can be transformed into a quiver isomorphic to~$Q'$ 
by a sequence of mutations. 
(Equivalently, $Q'$ can be transformed into a quiver isomorphic~to~$Q$.) 
The \emph{mutation equivalence class} $[Q]$ of a quiver $Q$ is the set
of all quivers (up to isomorphism) which are mutation equivalent to~$Q$.
\end{definition}

\begin{definition}\label{def:sametype}
Two quivers $Q$ and $Q'$ are said to have the same \emph{type}
if their mutable parts are mutation equivalent. Here the 
	\emph{mutable part} of the quiver refers to the mutable
	vertices together with all arrows between them.  When the mutation
	equivalence class of the mutable part has a name (e.g., type ADE)
then we will use that name to describe the type.
\end{definition}

\begin{example}
Consider the quiver $Q$ at the left of \cref{fig:A3}; this 
is an orientation of the type $A_3$ Dynkin diagram.  The 
	mutation equivalence class $[Q]$ of $Q$
	consists of quivers isomorphic to 
 one of 
those shown in \cref{fig:A3}.  In particular, we consider an 
oriented $3$-cycle to be a quiver of type $A_3$.
\begin{figure}[ht]
\begin{center}
\setlength{\unitlength}{2.8pt}
\begin{picture}(12,8)(0,0)
\put(0,0){\circle*{2}}
\put(12,0){\circle*{2}}
\put(6,10){\circle*{2}}
\thicklines
\put(4.4,8.66){\vector(-3,-5){4}}
\put(7.4,8.66){\vector(3,-5){4}}
\end{picture} \hspace{.5in}
\begin{picture}(12,8)(0,0)
\put(0,0){\circle*{2}}
\put(12,0){\circle*{2}}
\put(6,10){\circle*{2}}
\thicklines
\put(.5,2.06){\vector(3,5){4}}
\put(11.6,2.06){\vector(-3,5){4}}
\end{picture} \hspace{.5in}
\begin{picture}(12,8)(0,0)
\put(0,0){\circle*{2}}
\put(12,0){\circle*{2}}
\put(6,10){\circle*{2}}
\thicklines
\put(.5,2.06){\vector(3,5){4}}
\put(7.4,8.66){\vector(3,-5){4}}
\end{picture} \hspace{.5in}
\begin{picture}(12,8)(0,0)
\put(0,0){\circle*{2}}
\put(12,0){\circle*{2}}
\put(6,10){\circle*{2}}
\thicklines
\put(2,-0.7){\vector(1,0){8}}
\put(11.6,2.06){\vector(-3,5){4}}
\put(4.4,8.66){\vector(-3,-5){4}}
\end{picture}
\end{center}
\caption{The isomorphism classes of the quivers of type $A_3$.  
 All three vertices are mutable.}
\label{fig:A3}
\end{figure}
\end{example}

\begin{example}
The \emph{Markov quiver} is a quiver $Q$ of the form 
shown in Figure~\ref{fig:markov-quiver}. 
Mutating $Q$ at any of its $3$ vertices produces a quiver isomorphic to~$Q$,
so $[Q]$ 
consists of a single element (up to isomorphism).
\end{example}

\begin{figure}[ht]
\begin{center}
\vspace{-10pt}
\setlength{\unitlength}{2.8pt}
\begin{picture}(12,8)(0,0)
\put(0,0){\circle*{2}}
\put(12,0){\circle*{2}}
\put(6,10){\circle*{2}}
\thicklines
\put(2,-0.7){\vector(1,0){8}}
\put(2,0.7){\vector(1,0){8}}
\put(11.6,2.06){\vector(-3,5){4}}
\put(10.4,1.34){\vector(-3,5){4}}
\put(5.6,7.94){\vector(-3,-5){4}}
\put(4.4,8.66){\vector(-3,-5){4}}
\end{picture}
\vspace{-15pt}
\end{center}
\caption{The Markov quiver. All three vertices are mutable.}
\label{fig:markov-quiver}
\end{figure}

\begin{exercise}
\label{ex:orientations-of-a-tree}
Show that all orientations of a tree (with no frozen vertices)
are mutation equivalent to each
other via mutations at sinks and sources.
\end{exercise}

An $a\times b$ \emph{grid quiver} is an orientation of an $a\times b$
grid in which each $4$-cycle is oriented either clockwise or counterclockwise; 
see Figure~\ref{fig:grid}. 
All vertices are mutable.

\begin{exercise}
\label{exercise:grid=triang.grid} 
Show that a grid quiver 
is mutation equivalent to the 
corresponding \emph{triangulated grid quiver} (see Figure~\ref{fig:grid}).
\end{exercise}

\begin{figure}[ht]
\begin{center}
\vspace{-10pt}
\setlength{\unitlength}{2.8pt}
\begin{picture}(40,20)(0,0)
\put(0,0){\circle*{2}}
\put(10,0){\circle*{2}}
\put(20,0){\circle*{2}}
\put(30,0){\circle*{2}}
\put(0,10){\circle*{2}}
\put(10,10){\circle*{2}}
\put(20,10){\circle*{2}}
\put(30,10){\circle*{2}}
\put(0,20){\circle*{2}}
\put(10,20){\circle*{2}}
\put(20,20){\circle*{2}}
\put(30,20){\circle*{2}}
\thicklines
\put(2,0){\vector(1,0){6}}
\put(2,20){\vector(1,0){6}}
\put(12,10){\vector(1,0){6}}
\put(22,0){\vector(1,0){6}}
\put(22,20){\vector(1,0){6}}
\put(18,0){\vector(-1,0){6}}
\put(18,20){\vector(-1,0){6}}
\put(8,10){\vector(-1,0){6}}
\put(28,10){\vector(-1,0){6}}
\put(0,8){\vector(0,-1){6}}
\put(20,8){\vector(0,-1){6}}
\put(10,18){\vector(0,-1){6}}
\put(30,18){\vector(0,-1){6}}
\put(0,12){\vector(0,1){6}}
\put(20,12){\vector(0,1){6}}
\put(10,2){\vector(0,1){6}}
\put(30,2){\vector(0,1){6}}
\end{picture}
\qquad
\setlength{\unitlength}{2.8pt}
\begin{picture}(40,20)(0,0)
\put(0,0){\circle*{2}}
\put(10,0){\circle*{2}}
\put(20,0){\circle*{2}}
\put(30,0){\circle*{2}}
\put(0,10){\circle*{2}}
\put(10,10){\circle*{2}}
\put(20,10){\circle*{2}}
\put(30,10){\circle*{2}}
\put(0,20){\circle*{2}}
\put(10,20){\circle*{2}}
\put(20,20){\circle*{2}}
\put(30,20){\circle*{2}}
\thicklines
\put(2,0){\vector(1,0){6}}
\put(2,10){\vector(1,0){6}}
\put(2,20){\vector(1,0){6}}
\put(12,0){\vector(1,0){6}}
\put(12,10){\vector(1,0){6}}
\put(12,20){\vector(1,0){6}}
\put(22,10){\vector(1,0){6}}
\put(22,0){\vector(1,0){6}}
\put(22,20){\vector(1,0){6}}
\put(0,2){\vector(0,1){6}}
\put(0,12){\vector(0,1){6}}
\put(20,12){\vector(0,1){6}}
\put(10,2){\vector(0,1){6}}
\put(20,2){\vector(0,1){6}}
\put(10,12){\vector(0,1){6}}
\put(30,2){\vector(0,1){6}}
\put(30,12){\vector(0,1){6}}
\put(8,8){\vector(-1,-1){7}}
\put(8,18){\vector(-1,-1){7}}
\put(18,8){\vector(-1,-1){7}}
\put(18,18){\vector(-1,-1){7}}
\put(28,8){\vector(-1,-1){7}}
\put(28,18){\vector(-1,-1){7}}

\end{picture}
\vspace{-10pt}
\end{center}
\caption{The $3 \times 4$ grid quiver, and the corresponding
  triangulated grid quiver.} 
\label{fig:grid}
\end{figure}

\begin{exercise}
\label{exercise:D4-E8}
Verify that in each row of Figure~\ref{fig:E6E8},
the quiver on the left is mutation equivalent to any orientation of
the Dynkin diagram on the right.
\end{exercise}

\begin{figure}[ht]
\begin{center}
\begin{tabular}{ll}
\setlength{\unitlength}{2.8pt}
\begin{picture}(30,10)(0,0)
\thicklines
\multiput(0,0)(10,0){2}{\circle*{2}}
\multiput(0,10)(10,0){2}{\circle*{2}}
\multiput(0,2)(10,0){2}{\vector(0,1){6}}
\multiput(8,8)(10,0){1}{\vector(-1,-1){6}}
\multiput(2,0)(10,0){1}{\vector(1,0){6}}
\multiput(2,10)(10,0){1}{\vector(1,0){6}}
\end{picture}
&
\setlength{\unitlength}{2.8pt}
\begin{picture}(40,10)(0,0)
\multiput(10,10)(10,0){3}{\circle*{2}}
\put(20,0){\circle*{2}}
\thicklines
\put(10,10){\line(1,0){20}}
\put(20,10){\line(0,-1){10}}
\end{picture}
\\[.3in]
\setlength{\unitlength}{2.8pt}
\begin{picture}(30,10)(0,0)
\thicklines
\multiput(0,0)(10,0){2}{\circle*{2}}
\multiput(0,10)(10,0){3}{\circle*{2}}
\multiput(0,2)(10,0){2}{\vector(0,1){6}}
\multiput(8,8)(10,0){2}{\vector(-1,-1){6}}
\multiput(2,0)(10,0){1}{\vector(1,0){6}}
\multiput(2,10)(10,0){2}{\vector(1,0){6}}
\end{picture}
&
\setlength{\unitlength}{2.8pt}
\begin{picture}(40,10)(0,0)
\multiput(0,10)(10,0){4}{\circle*{2}}
\put(20,0){\circle*{2}}
\thicklines
\put(0,10){\line(1,0){30}}
\put(20,10){\line(0,-1){10}}
\end{picture}
\\[.3in]
\setlength{\unitlength}{2.8pt}
\begin{picture}(30,10)(0,0)
\thicklines
\multiput(0,0)(10,0){3}{\circle*{2}}
\multiput(0,10)(10,0){3}{\circle*{2}}
\multiput(0,2)(10,0){3}{\vector(0,1){6}}
\multiput(8,8)(10,0){2}{\vector(-1,-1){6}}
\multiput(2,0)(10,0){2}{\vector(1,0){6}}
\multiput(2,10)(10,0){2}{\vector(1,0){6}}
\end{picture}
&
\setlength{\unitlength}{2.8pt}
\begin{picture}(40,10)(0,0)
\multiput(0,10)(10,0){5}{\circle*{2}}
\put(20,0){\circle*{2}}
\thicklines
\put(0,10){\line(1,0){40}}
\put(20,10){\line(0,-1){10}}
\end{picture}
\\[.3in]
\setlength{\unitlength}{2.8pt}
\begin{picture}(30,10)(0,0)
\thicklines
\multiput(0,0)(10,0){3}{\circle*{2}}
\multiput(0,10)(10,0){4}{\circle*{2}}
\multiput(0,2)(10,0){3}{\vector(0,1){6}}
\multiput(8,8)(10,0){3}{\vector(-1,-1){6}}
\multiput(2,0)(10,0){2}{\vector(1,0){6}}
\multiput(2,10)(10,0){3}{\vector(1,0){6}}
\end{picture}
&
\setlength{\unitlength}{2.8pt}
\begin{picture}(40,10)(0,0)
\multiput(0,10)(10,0){6}{\circle*{2}}
\put(20,0){\circle*{2}}
\thicklines
\put(0,10){\line(1,0){50}}
\put(20,10){\line(0,-1){10}}
\end{picture}
\\[.3in]
\setlength{\unitlength}{2.8pt}
\begin{picture}(40,10)(0,0)
\thicklines
\multiput(0,0)(10,0){4}{\circle*{2}}
\multiput(0,10)(10,0){4}{\circle*{2}}
\multiput(0,2)(10,0){4}{\vector(0,1){6}}
\multiput(8,8)(10,0){3}{\vector(-1,-1){6}}
\multiput(2,0)(10,0){3}{\vector(1,0){6}}
\multiput(2,10)(10,0){3}{\vector(1,0){6}}
\end{picture}
&
\setlength{\unitlength}{2.8pt}
\begin{picture}(60,10)(0,0)
\multiput(0,10)(10,0){7}{\circle*{2}}
\put(20,0){\circle*{2}}
\thicklines
\put(0,10){\line(1,0){60}}
\put(20,10){\line(0,-1){10}}
\end{picture}
\end{tabular}
	\vspace{-5pt}
\end{center}
\caption{Quivers mutation equivalent to orientations of Dynkin
  diagrams of types $D_4, D_5, E_6, E_7, E_8$.} 
\label{fig:E6E8}
\end{figure}

A \emph{triangular grid quiver} with $k$ vertices on each side
is a quiver with $\binom{k+1}{2}$ vertices and $3\binom{k}{2}$ arrows
that has the form shown in Figure~\ref{fig:trigrid}. 
All vertices are mutable. 

\begin{figure}[ht]
\begin{center}
\setlength{\unitlength}{2.8pt}
\begin{picture}(24,20)(0,0)
\put(0,0){\circle*{2}}
\put(12,0){\circle*{2}}
\put(24,0){\circle*{2}}
\put(6,10){\circle*{2}}
\put(18,10){\circle*{2}}
\put(12,20){\circle*{2}}
\put(2,0){\vector(1,0){8}}
\put(14,0){\vector(1,0){8}}
\put(8,10){\vector(1,0){8}}
\put(11.1,1.5){\vector(-3,5){4.2}}
\put(23.1,1.5){\vector(-3,5){4.2}}
\put(17.1,11.5){\vector(-3,5){4.2}}
\put(5.1,8.5){\vector(-3,-5){4.2}}
\put(17.1,8.5){\vector(-3,-5){4.2}}
\put(11.1,18.5){\vector(-3,-5){4.2}}
\end{picture}
	\vspace{-5pt}
\end{center}
\caption{A triangular grid quiver.}
\label{fig:trigrid}
\end{figure}

\begin{exercise}
Show that the triangular grid quiver with three vertices on each side 
(see Figure~\ref{fig:trigrid}) is mutation equivalent to an orientation of a tree.  (However, this is no longer true for a triangular grid
quiver with four vertices on each side.)
\end{exercise}

\begin{exercise}
	\emph{(Difficult but elementary.)}
Show that the $k \times (2k+1)$ grid quiver
is mutation equivalent to the triangular grid 
quiver with $2k$ vertices on each side. 
\end{exercise}

\begin{definition}
\label{def:finitemutationtype}
A quiver $Q$ 
is said to have \emph{finite mutation type} 
if the mutation equivalence class $[Q]$ of $Q$ 
is finite.  
\end{definition}

Quivers with no frozen vertices
of finite mutation type can be completely classified in
explicit combinatorial terms. 
This classification will be described in Chapter~\ref{ch:surfaces}. 

We conclude this section by stating, without proof, an
innocent-looking but rather nontrivial result 
about quiver mutation. 

A quiver is called \emph{acyclic} if it has no oriented cycles.

\begin{theorem}[see \cite{caldero-keller}]
\label{th:acyclic-mut-equiv}
Let $Q$ and $Q'$ be acyclic quivers mutation equivalent to each other.
Then $Q$ can be transformed into a quiver isomorphic to~$Q'$ 
via a sequence of mutations at sources and sinks. 
Consequently (cf.\ Remark~\ref{rem:mut-source-sink}), 
all acyclic quivers in a given mutation equivalence
class have the same underlying undirected graph. 
\end{theorem}

\begin{corollary}
\label{cor:trees-inequivalent}
An acyclic quiver which is mutation equivalent to an orientation of a
tree is itself an orientation of the same tree.  
In particular, orientations of non-isomorphic trees are not mutation equivalent.
\end{corollary}

The proof of \cref{th:acyclic-mut-equiv} 
(and hence \cref{cor:trees-inequivalent})  given in \cite{caldero-keller}
uses the machinary of \emph{cluster categories}; 
it would be very interesting to find a purely combinatorial proof of
either of these results.

\begin{exercise}
Which orientations of an $n$-cycle are
muta\-tion equivalent? 
\end{exercise}

In general, it can be very hard to determine whether two quivers are
mutation equivalent to each other. In what follows, we use the 
designation ``problem'' to refer to an open problem.

\begin{problem}
\label{probem:decide-mut-equiv-quivers}
Design a (reasonably efficient) algorithm for deciding whether two
quivers are mutation equivalent or not. 
\end{problem}

\section{Matrix mutation}
\label{sec:mutmatrix}

In this section, we extend the notion of mutation from quivers to 
a certain class of  matrices. 
We begin by explaining how matrices can be viewed as generalizations of quivers. 

\begin{definition}
\label{def:Bmatrix}
Let $Q$ be a quiver (as in Definition~\ref{def:quiver})
with $m$ vertices, $n$~of them mutable. 
Let us label the vertices of $Q$ by the indices $1,\dots,m$
so that the mutable vertices are labeled $1,\dots,n$. 
The \emph{extended exchange matrix} of~$Q$ is the $m \times n$
matrix $\tilde B(Q) = (b_{ij})$ defined by
\[
b_{ij} = 
\begin{cases}
\ell & \text{if there are $\ell$ arrows from vertex $i$ to vertex $j$ in~$Q$;}\\
-\ell & \text{if there are $\ell$ arrows from vertex $j$ to vertex $i$
  in~$Q$;}\\
0 &\text{otherwise.}
\end{cases}
\]
The \emph{exchange matrix} $B(Q)$ is the 
$n\times n$ skew-symmetric submatrix of $\tilde B(Q)$ occupying
the first $n$ rows:
\[
B(Q)=(b_{ij})_{i,j\in [1,n]}\,.
\]
\end{definition}

To illustrate, consider the Markov quiver~$Q$
shown in Figure~\ref{fig:markov-quiver}.
Then
\[
\tilde B(Q)=B(Q)=
\pm\begin{bmatrix}
0 & 2 & -2\\
-2& 0 & 2\\
2 & -2& 0
\end{bmatrix},
\]
where the sign depends on the labeling of the vertices.

\begin{remark}
While the definition of $\tilde B(Q)$ depends on the choice of labeling of
the vertices of~$Q$ by the integers $1,\dots,m$, 
we often consider
extended exchange matrices up to a simultaneous relabeling of 
rows and columns $1,2,\dots,n$, and a relabeling of the rows
$n+1, n+2, \dots, m$.
\end{remark}

The proof of the following lemma is straightforward. 

\begin{lemma}
\label{lem:adjmatrix}
Let $k$ be a mutable vertex of a quiver~$Q$. 
The extended exchange matrix $\tilde B(\mu_k(Q))=(b'_{ij})$ of the mutated
quiver $\mu_k(Q)$ 
is given by
\begin{equation}
\label{eq:matrix-mutation}
b'_{ij} =
\begin{cases}
-b_{ij} & \text{if $i=k$ or $j=k$;} \\[.05in]
b_{ij}+b_{ik}b_{kj} & \text{if $b_{ik}>0$ and $b_{kj}>0$;}\\[.05in]
b_{ij}-b_{ik}b_{kj} & \text{if $b_{ik}<0$ and $b_{kj}<0$;}\\[.05in]
b_{ij} & \text{otherwise.}
\end{cases}
\end{equation}
\end{lemma}

We next move from skew-symmetric matrices 
to a more general class of matrices.

\begin{definition}
\label{def:skew-symmetrizable}
An $n \times n$ matrix $B=(b_{ij})$ with integer entries is called 
\emph{skew-symmetrizable} if $d_i b_{ij} = -d_j b_{ji}$
for some positive integers $d_1, \dots, d_n$.  
In other words, a matrix is skew-symmetrizable if it differs from a 
skew-symmetric matrix by a rescaling of its rows by positive scalars.

An $m \times n$ integer matrix, 
with $m \geq n$, whose top $n\times n$ submatrix is 
skew-symmetrizable is called an \emph{extended skew-symmetrizable} matrix.
\end{definition}

\begin{exercise}
Show that the class of matrices $B$ described in 
Definition~\ref{def:skew-symmetrizable} would not change if instead of
rescaling the rows of~$B$, we rescale its columns; 
alternatively, we could  conjugate~$B$ by a diagonal matrix with positive real diagonal
entries. 
\end{exercise}

We are now ready to define the notion of matrix mutation.

\begin{definition}
\label{def:matrix-mutation}
Let $\tilde{B} =(b_{ij})$ be an $m \times n$ 
extended skew-symmetrizable integer matrix.  
For $k \in [1,n]$, the
\emph{matrix mutation $\mu_k$ in direction~$k$} transforms 
$\tilde{B}$ into the $m \times n$ matrix $\mu_k(\tilde{B}) = (b'_{ij})$ whose
entries are given by~\eqref{eq:matrix-mutation}.
\end{definition}

By Lemma~\ref{lem:adjmatrix}, matrix mutation 
generalizes quiver mutation.

\begin{exercise}
\label{ex:mat-mut-simple}
Under the conventions of Definitions
\ref{def:skew-symmetrizable} and~\ref{def:matrix-mutation}, 
verify that
\begin{enumerate}[leftmargin=.25in]
\item
the mutated matrix $\mu_k(\tilde B)$ is again extended skew-symmetrizable,
with the same choice of $d_1, \dots, d_n$; 
\item
$\mu_k(\mu_k(\tilde B))=\tilde B$;
\item
\label{ex:mat-mut-simple-3}
$\mu_k(-\tilde B)=-\mu_k(\tilde B)$;
\item
\label{ex:mat-mut-simple-3.5}
$\mu_k(B^T)=(\mu_k(B))^T$, where $B^T$ denotes
the transpose of $B$;
\item
\label{ex:mat-mut-simple-4}
if $b_{ij}=b_{ji}=0$, then $\mu_i(\mu_j(\tilde B))=\mu_j(\mu_i(\tilde B))$.
\end{enumerate}
\end{exercise}

For $b\in\RR$, let ${\rm sgn}(b)$ be $1$, $0$, or $-1$, depending on whether $b$
is positive, zero, or negative.

\begin{definition}
\label{def:skew-symmetrization} 
Let $B$ be a skew-symmetrizable matrix. 
The skew-symmetric matrix $S(B)=(s_{ij})$ defined by
\begin{equation} 
\label{eq:s-matrix} 
s_{ij} = 
{\rm sgn}(b_{ij}) \textstyle\sqrt{|b_{ij} b_{ji}|}
\end{equation} 
is called the \emph{skew-symmetrization} of~$B$. 
Note that $S(B)$ has real (not necessarily integer) entries. 
	\cref{ex:skew-symmetrization}
shows that skew-symmetrization commutes with mutation
(extended verbatim to matrices with real entries). 
\end{definition}

\begin{exercise} 
\label{ex:skew-symmetrization} 
Prove that for any skew-symmetrizable matrix~$B$ and any~$k$, we have
\begin{equation}
\label{eq:skew-symmetrization} 
S(\mu_k(B))=\mu_k(S(B)).
\end{equation}
\end{exercise}

\begin{definition}
\label{def:diagramofB}
The \emph{diagram} of a skew-symmetrizable $n \times n$
matrix \hbox{$B=(b_{ij})$} is the weighted directed graph $\Gamma (B)$
with the vertices $1, \dots, n$ such that there is a directed edge from
$i$ to~$j$
if and only if $b_{ij} > 0$, and this edge is assigned the weight
$|b_{ij}b_{ji}|\,$.
In particular, if $b_{ij}\in\{-1,0,1\}$ for all $i$ and~$j$, then
$\Gamma (B)$ is a quiver whose exchange matrix is~$B$.
\end{definition}

To illustrate Definition~\ref{def:diagramofB}, 
consider $B=\left[\begin{smallmatrix}0 & 2 & -2\\ -1 & 0 & 2\\ 1 & -2 & 0\end{smallmatrix}\right]$.
Then $\Gamma (B)$ is an oriented cycle with edge weights $2$, $4$, and~$2$.

More generally, we use the term \emph{diagram} in the rest of this
chapter to mean a finite directed graph~$\Gamma$
(no loops, multiple edges, or $2$-cycles allowed)
whose edges are assigned positive real weights.

We note that the diagram $\Gamma (B)$ does \emph{not} determine $B$: for
instance,
the matrix $(-B^T)$ has the same diagram as $B$.
Here is another example:
\[
\Gamma\left(\begin{bmatrix}
0 & 1\\
-4 & 0
\end{bmatrix}\right)
=\Gamma\left(
\begin{bmatrix}
0 & 2\\
-2 & 0
\end{bmatrix}\right). 
\]

Note that the diagram $\Gamma(B)$ and the
skew-symmetric matrix $S(B)$ encode the same information about~$B$:
having an edge $i\to j$ in $\Gamma (B)$ supplied
with weight~$c$ is the same as saying that $s_{ij} = \sqrt {c}$ and
$s_{ji} = - \sqrt {c}$.

\begin{proposition}
\label{pr:diagram-mutation}
For a skew-symmetrizable matrix $B$,
the diagram $\Gamma'\!\!=\!\Gamma(\mu_k(B))$ is uniquely determined by
the diagram $\Gamma=\Gamma(B)$ and an index~$k$.
\end{proposition}

\begin{proof}
By Exercise~\ref{ex:skew-symmetrization}, $S(\mu_k(B))=\mu_k(S(B))$.
It remains to translate this property into the language of diagrams.
\end{proof}

In the situation of Proposition~\ref{pr:diagram-mutation},
we write $\Gamma' = \mu_k (\Gamma)$, and call~the transformation
$\mu_k$ a \emph{diagram mutation in direction~$k$}.
A~detailed description of diagram mutation can be found in
\cite[Proposition~8.1]{ca2}. 
Two~diagrams $\Gamma$ and $\Gamma'$ related by a sequence of mutations are called
\emph{mutation equivalent,} and we write $\Gamma\sim\Gamma'$.

\begin{remark}
While the entries of~$B$ are integers,
the entries of $S(B)$ may be irrational,
as the weights of $\Gamma(B)$ may not be perfect squares.
On the other hand, one can deduce from
the skew-symmetrizability of $B$ that the product of weights
over the edges of any cycle in the underlying graph of $\Gamma(B)$
is a perfect square.
\end{remark}

\begin{lemma} \label{lem:vector}
If the diagram $\Gamma(B)$ of an $n \times n$ skew-symmetrizable
matrix $B$ is connected, then the skew-symmetrizing vector
$(d_1,\dots,d_n)$ is unique up to rescaling.
\end{lemma}

\begin{proof}
Let $(d_1,\dots, d_n)$ and $(d'_1,\dots, d'_n)$
be two skew-symmetrizing vectors.  We have 
$d_i b_{ij} = -d_j b_{ji}$ and 
$d'_i b_{ij} = -d'_j b_{ji}$ for all $i$ and~$j$.
So if $b_{ij}$ is nonzero, then  
$\frac{b_{ij}}{b_{ji}} = 
\frac{-d_j}{d_i} = \frac{-d'_j}{d'_i}$ and hence
$\frac{d_j}{d'_j} = \frac{d_i}{d'_i}.$
Since $\Gamma(B)$ is connected, there exists an ordering
$\ell_1, \ell_2,\dots, \ell_n$ of its vertices
such that every vertex~$\ell_j$ with $2 \leq j \leq n$ 
is connected by an edge in $\Gamma(B)$ to a vertex $\ell_i$  with
$i<j$; in other words, $b_{\ell_i \ell_j} \neq 0$.  It follows that 
$\frac{d_{\ell_1}}{d'_{\ell_1}} = 
\frac{d_{\ell_2}}{d'_{\ell_2}} =  \dots = 
\frac{d_{\ell_n}}{d'_{\ell_n}}$, as desired.
\end{proof}

\pagebreak[3]

\section{Invariants of matrix mutations}

The following notion is a straightforward extension of
Definition~\ref{def:mut-equiv-quivers}. 

\begin{definition}
\label{def:mut-equiv-matrices}
Two skew-symmetrizable matrices
$B$ and $B'$ are \emph{mutation equivalent}
if one can get from $B$ to $B'$ by a sequence of mutations,
possibly followed by simultaneous renumbering of rows and columns.  
The \emph{mutation equivalence class} $[B]$ of $B$ is the set
of all matrices 
mutation equivalent to~$B$.
These notions generalize to extended skew-symmetrizable matrices in an obvious way. 
\end{definition}

It is natural to extend Problem~\ref{probem:decide-mut-equiv-quivers}
to the setting of matrix mutations: 

\begin{problem}
\label{problem:decide-mut-equiv}
Find an effective way 
to determine whether two given 
$n\times n$ skew-symmetrizable matrices are mutation equivalent.
\end{problem}

Problem~\ref{problem:decide-mut-equiv} remains wide open, even in the case of
skew-symmetric matrices (or equivalently quivers). 
For $n=2$, the question is trivial, since 
mutation simply negates the entries of the matrix.  
For $n=3$, there is an explicit algorithm for determining
whether two skew-symmetric matrices are mutation equivalent, 
see~\cite{ABBS}.  

Problem~\ref{problem:decide-mut-equiv} is closely related to the
problem of identifying explicit nontrivial invariants of matrix (or quiver) mutation. 
Unfortunately, very few invariants of this kind are known at present. 

\begin{theorem}[{\cite[Lemma 3.2]{ca3}}]
\label{thm:rank-invariant}
Mutations preserve the rank of a matrix. 
\end{theorem}

\begin{proof}
Let~$\tilde B$ be an $m\times n$ extended skew-symmetrizable integer matrix.
Fix an index $k \in [1,n]$ and a sign $\varepsilon \in \{1, -1\}$.
The rule \eqref{eq:matrix-mutation}
describing the matrix mutation in direction~$k$
can be rewritten as follows:
\begin{equation}
\label{eq:matrix-mutation-restated}
b'_{ij} =
\begin{cases}
-b_{ij} & \text{if $i=k$ or $j=k$;} \\[.05in]
b_{ij} 
+\max(0, - \varepsilon b_{ik})\,b_{kj} 
+ b_{ik} \max(0, \varepsilon b_{kj}) 
& \text{otherwise.}
\end{cases}
\end{equation}
(To verify this, examine the four possible sign patterns for $b_{ik}$
and~$b_{kj}$.) \linebreak[3]
Next observe that \eqref{eq:matrix-mutation-restated}
can be restated as
\begin{align}
\label{eq:mutation-product}
\mu_k(\tilde B) &=J_{m,k}\,\tilde B J_{n,k} + J_{m,k}\,\tilde B F_k + E_k \,\tilde B J_{n,k} \\
\notag &= (J_{m,k} + E_k) \,\tilde B\, (J_{n,k} + F_k)
\end{align}
where
\begin{itemize}[leftmargin=.15in]
\item
$J_{m,k}$ (respectively, $J_{n,k}$) denotes the diagonal matrix 
of size $m\times m$ (respectively,~$n\times n$) 
whose diagonal entries are all~$1$, except for the $(k,k)$ entry,
which is $-1$;
\item
$E_k=(e_{ij})$
is the $m\times m$ matrix with
$e_{ik}=\max(0, -\varepsilon b_{ik})$, and all other entries equal to $0$;
\item
$F_k=(f_{ij})$
is the $n\times n$ matrix with
$f_{kj}=\max(0, \varepsilon b_{kj})$, and all other entries equal to $0$.
\end{itemize}
(Here we used that $E_k\,\tilde B F_k=0$ because $b_{ii}=0$ for all~$i$.)
Since 
\begin{equation}
\label{eq:det=-1}
\det(J_{m,k} + {E}_k) = \det(J_{n,k} + F_k) = -1, 
\end{equation}
 it follows that
$\operatorname{rank}(\mu_k(\tilde B))=\operatorname{rank}(\tilde B)$. \end{proof}

\begin{theorem}
\label{thm:invariant}
The determinant of a skew-symmetrizable matrix is invariant under mutation.
\end{theorem}

\begin{proof}
This follows from \eqref{eq:mutation-product} and~\eqref{eq:det=-1}
(taking $m=n$ and $B=\tilde B$). 
\end{proof}

Another invariant of matrix mutations
is the greatest common divisor of the matrix elements of~$B$. 
A finer invariant is the greatest common divisor of the 
matrix elements of the $i$th row (or column) of 
~$B$, for a 
fixed index~$i$~\cite{seven2}.

\begin{remark}
For skew-symmetric matrices (equivalently, quivers with no
frozen vertices), formulas \eqref{eq:mutation-product}
and~\eqref{eq:det=-1} allow us to interpret mutation as a
transformation of a skew-symmetric bilinear form over the integers
under a particular unimodular change of basis. 
One can then use the general theory of invariants of such
transformations (the \emph{skew Smith normal form}, 
see \cite[Section~IV.3]{newman})
to identify some invariants of quiver mutation.
Unfortunately this approach does not yield much beyond the facts
established above. 
\end{remark}

\chapter{Clusters and seeds}

\vspace{-.1in}

This chapter introduces cluster algebras of
\emph{geometric type}. 
A~more general construction of cluster algebras over an arbitrary
semifield will be discussed in Chapter~\ref{ch:general-cluster-algebras}.

\section{Basic definitions}
\label{sec:geometric-type-basic}

Let us recall the three motivating examples discussed in
Chapter~\ref{ch:tp-examples}: 
Grassmannians of $2$-planes, basic affine spaces, and general linear
groups.
In each of these examples, we manipulated two kinds of data:
\begin{itemize}[leftmargin=.15in]
\item
combinatorial data (triangulations, wiring diagrams) and
\item
algebraic data (Pl\"ucker coordinates, chamber minors). 
\end{itemize}
Accordingly, transformations applied to these data occurred on two
levels: 
\begin{itemize}[leftmargin=.15in]
\item
on the ``primary'' level, we saw the combinatorial data 
evolve via local moves
(flips in triangulations, braid moves in wiring diagrams);
as shown in Chapter~\ref{ch:combinatorics-of-mutations},
a unifying description of this dynamics
 can be given using the language of quiver
mutations; 
\item
on the ``secondary" level, we saw the algebraic data 
evolve in a way that was ``driven" by the combinatorial
dynamics, with subtraction-free birational transformations,
		called exchange relations,
encoded by the current combinatorial data.
\end{itemize}
An attempt to write the exchange relations in terms of the quiver at hand
naturally leads to the axiomatic setup of cluster algebras of
geometric type, which we will now describe. 

Let~$m$ and~$n$ be two positive integers such that~$m \geq n$.
As an \emph{ambient field} for a cluster algebra, we take a field $\FFcal$
isomorphic to the field of rational functions over~$\CC$ 
(alternatively, over~$\QQ$) 
in $m$ independent variables.

\begin{definition}
\label{def:seed}
A \emph{labeled seed} of geometric type in $\FFcal$ is a pair
$(\tilde \xx,\! \tilde B)$~where
\begin{itemize}[leftmargin=.15in]
\item $\tilde \xx = (x_1, \dots, x_m)$ is an $m$-tuple of
elements of~$\FFcal$ forming a
\emph{free generating set}; that is, $x_1, \dots, x_m$
are algebraically independent, and $\FFcal = \CC(x_1, \dots, x_m)$;
\item $\tilde B=(b_{ij})$ is an $m \times n$ extended
  skew-symmetrizable integer matrix, see 
Definition~\ref{def:skew-symmetrizable}.
\end{itemize}
We shall use the following terminology:
\begin{itemize}[leftmargin=.15in]
\item 
$\tilde \xx$ is the (labeled)
\emph{extended cluster}
\index{cluster!extended}
of the labeled seed $(\tilde \xx, \tilde B)$;
\item 
the $n$-tuple~$\xx = (x_1, \dots, x_n)$ is the (labeled)
\emph{cluster}
\index{cluster}
of this seed; 
\item 
the elements $x_1, \dots, x_n$ are its \emph{cluster variables}; 
\item 
the remaining elements $x_{n+1}, \dots, x_m$ of $\tilde \xx$ are the 
\emph{frozen variables} (or \emph{coefficient variables}); 
\item 
the matrix~$\tilde B$ is the \emph{extended exchange matrix} of the seed;
\item 
its top $n\times n$ submatrix~$B$ is the \emph{exchange matrix}.
\index{exchange matrix}
\end{itemize}
\end{definition}

\begin{example}
\label{eg:seed-small}
Let $m=3$.
Let $\FFcal=\CC(x_1,x_2,x_3)$ be the field of rational functions in the formal variables $x_1,x_2,x_3$, and 
set $n=2$.
Figure~\ref{fig:table-terminology} illustrates
\cref{def:seed} with
 two seeds $\Sigma=(\tilde\xx,\tilde B)$ and $\Sigma'=(\tilde\xx',\tilde B')$.

\begin{figure}[ht]
{
\vspace{-10pt}
\begin{equation*}
\begin{array}{|c|c|c|}
\hline
&&\\[-4mm]
	 & \Sigma & \Sigma' \\[1mm]
\hline
&&\\[-4mm]
	\text{extended cluster} & \tilde\xx=(x_1, x_2, x_3) & \tilde\xx'=(x_1,\frac{x_1+x_3}{x_2}, x_3) \\[1mm]
\hline
&&\\[-3.5mm]
	\text{cluster variables} & x_1, x_2 & x_1, \frac{x_1+x_3}{x_2} \\[1mm]
\hline
&&\\[-3.5mm]
	\text{frozen variables} & x_3 & x_3 \\[1mm]
\hline
&&\\[-3.5mm]
	\text{	extended exchange matrix} & 
	\tilde B=\begin{bmatrix} 
		0 & 1\\
		-1 & 0\\
		1 & -1
	\end{bmatrix} 
	& 
	\tilde B'=\begin{bmatrix} 
		0 & -1\\
		1 &  0 \\
		0 & 1 
	\end{bmatrix} 
	\\[6mm]
\hline
&&\\[-3.5mm]
	\text{exchange matrix} & 
	B=\begin{bmatrix} 
		0 & 1\\
		-1 & 0
	\end{bmatrix} 
	 & 
	B'=\begin{bmatrix} 
		0 & -1\\
		1 & 0
	\end{bmatrix} 
	 \\[4mm]
\hline
\end{array}
\end{equation*}
}
\vspace{-.2in}
	\caption{Two labeled seeds illustrating 
	\cref{def:seed}.
	}
\label{fig:table-terminology}
\end{figure}
One can alternatively describe these seeds using quivers, see below. 
The quivers on the left encode
the extended exchange matrix whereas the quiver labels on the right 
encode the extended cluster.
The boxed node indicates the frozen vertex (resp., frozen variable). 
\begin{equation}
\label{eq:two-seeds}
\begin{tikzcd}[arrows={-stealth}, sep=2em]
\text{\raisebox{-2pt}{$\Sigma$}} & [3pt]   1  \arrow[r]   & 2 \arrow[r] & \boxed{3} \arrow[ll,bend left=20] &[3pt]  x_1  \arrow[r]   & x_2 \arrow[r] & \boxed{x_3} \arrow[ll,bend left=20] \\[-15pt]
\text{\raisebox{-2pt}{$\Sigma'$}} &   1    & 2 \arrow[l] & \boxed{3}  \arrow[l] & x_1    & \frac{x_1+x_3}{x_2} \arrow[l] & \boxed{x_3}  \arrow[l] 
  \end{tikzcd} 
\vspace{-5pt}
\end{equation} 
\end{example}

\pagebreak[3]

\begin{definition}
\label{def:seed-mutation}
Let $(\tilde \xx, \tilde B)$ be a labeled seed as above. 
Take an index $k\in\{1,\dots,n\}$. 
The \emph{seed mutation} $\mu_k$ in direction~$k$ 
transforms
$(\tilde \xx, \tilde B)$ into the new labeled seed
$\mu_k(\tilde \xx, \tilde B) = (\tilde \xx', \tilde B')$ 
defined as follows:
\begin{itemize}[leftmargin=.15in]
\item
$\tilde B'= \mu_k(\tilde B)$ 
(cf.\ Definition~\ref{def:matrix-mutation}).
\item
the extended cluster $\tilde \xx' =(x_1',\dots,x_m')$ is
given by $x_j'=x_j$ for~$j\neq k$,
whereas $x'_k \in \FFcal$ is determined
by the \emph{exchange relation} 
\begin{equation}
\label{eq:exch-rel-geom}
x_k \, x'_k = \prod_{b_{ik}>0} x_i^{b_{ik}} 
+ \prod_{b_{ik}<0} x_i^{-b_{ik}}.
\end{equation}
\vspace{-10pt}
\end{itemize}
We note that if the indexing set for one of the two monomials above is the empty
set, then by convention we set the corresponding product equal to~$1$.
\end{definition}

\begin{remark}
	For a labeled seed described by a quiver,
	the first (resp., second) monomial on the right-hand side of 
	\eqref{eq:exch-rel-geom} corresponds to 
	the arrows pointing \emph{towards} (resp., \emph{away from}) vertex~$k$.
\end{remark}

\begin{example}
\label{eg:A2-6seeds}
\cref{fig:A2} shows the result of mutating the seed~$\Sigma$ from Example~\ref{eg:seed-small} (cf.~\eqref{eq:two-seeds})
in directions $2,1,2,1,2$.  
Note that the final seed miraculously agrees with the initial seed,  
upon relabeling the vertices. 
\end{example}

\begin{figure}[h]
\vspace{-17pt}
	\begin{center}
\begin{tikzcd}[arrows={-stealth}, sep=2em]
1  \arrow[thick, r]   & 2 \arrow[thick, r] & \boxed{3} \arrow[thick, ll,bend left=20] 
&& x_1  \arrow[thick, r]   & x_2 \arrow[thick, r] & [22pt] \boxed{x_3} \arrow[thick, ll,bend left=15] \\[-21pt]
&{\ } \arrow[d,blue,<->,"\mu_2"]&&&& {\ } \arrow[d,blue,<->,"\mu_2"] \\[-5pt]
&{\ }&&&& {\ } \\[-28pt]
1    & 2 \arrow[thick, l] & \boxed{3}  \arrow[thick, l] && x_1    & \dfrac{x_1\!+\!x_3}{x_2} \arrow[thick, l] & [5pt] \boxed{x_3}  \arrow[thick, l] \\[-28pt]
&{\ } \arrow[d,blue,<->,"\mu_1"]&&&& {\ } \arrow[d,blue,<->,"\mu_1"] \\[-5pt]
&{\ }&&&& {\ } \\[-28pt]
1  \arrow[thick, r]  & 2   & \boxed{3} \arrow[thick, l] && \dfrac{x_1\!+\!x_2\!+\!x_3}{x_1x_2}  \arrow[thick, r]   & \dfrac{x_1\!+\!x_3}{x_2}   &  [5pt] \boxed{x_3} \arrow[thick, l] \\[-28pt]
&{\ } \arrow[d,blue,<->,"\mu_2"]&&&& {\ } \arrow[d,blue,<->,"\mu_2"] \\[-5pt]
&{\ }&&&& {\ } \\[-28pt]
1     & 2  \arrow[thick, r] \arrow[thick, l]&  \boxed{3} && \dfrac{x_1\!+\!x_2\!+\!x_3}{x_1x_2}     & \dfrac{x_2\!+\!x_3}{x_1} \arrow[thick, r] \arrow[thick, l] & [5pt] \boxed{x_3} \\[-28pt]
&{\ } \arrow[d,blue,<->,"\mu_1"]&&&& {\ } \arrow[d,blue,<->,"\mu_1"] \\[-5pt]
&{\ }&&&& {\ } \\[-28pt]
1  \arrow[thick, r]   & 2  \arrow[thick, r] & \boxed{3} && x_2  \arrow[thick, r]   & \dfrac{x_2\!+\!x_3}{x_1}  \arrow[thick, r] & [5pt] \boxed{x_3}  \\[-28pt]
&{\ } \arrow[d,blue,<->,"\mu_2"]&&&& {\ } \arrow[d,blue,<->,"\mu_2"] \\[-5pt]
&{\ }&&&& {\ } \\[-31pt]
1  \arrow[thick, rr,bend right=20]  & 2  \arrow[thick, l] & \boxed{3} \arrow[thick, l] && x_2  \arrow[thick, rr,bend right=15]  & x_1  \arrow[thick, l] & [5pt] \boxed{x_3}  \arrow[thick, l] 
  \end{tikzcd} 
\vspace{-15pt}
\end{center} 
\caption{A sequence of five consecutive seed mutations. The seeds appearing in the top two rows are the seeds $\Sigma$ and $\Sigma'$ from~\eqref{eq:two-seeds}. 
}
\label{fig:A2}
\end{figure}

\begin{exercise}
\label{exercise:quiverexchange}
Consider each of the three settings that we discussed in
Sections \ref{sec:Ptolemy}, \ref{sec:baseaffine},
and~\ref{sec:matrices}. 
Construct a seed $(\tilde \xx, \tilde B(Q))$ where $Q$ is a quiver 
associated with a particular triangulation, wiring diagram, or double wiring
diagram 
(see Definitions~\ref{def:Q(T)-polygon}, \ref{def:quiverwd},
and~\ref{def:quiverdwd}, respectively),
and $\tilde \xx$ is the extended cluster consisting of the
corresponding Pl\"ucker coordinates or chamber minors. 
 Verify that applying the recipe \eqref{eq:exch-rel-geom} to these
 data recovers the appropriate exchange relations
\eqref{eq:grassmann-plucker-3term}, \eqref{eq1:ac+bd=yz},
and~\eqref{eq:ac+bd=yz}, respectively. 
\end{exercise}

\begin{definition}
\label{def:Tn}
Let~$\TT_n$ denote the \emph{$n$-regular tree}
whose edges are labeled by the numbers $1, \dots, n$,
so that the $n$ edges incident to each vertex receive
different labels.
We shall write $t \overunder{k}{} t'$ to indicate that vertices
$t,t'\in\TT_n$ are joined by an edge with label~$k$. 
See Figure~\ref{fig:treeT3}. 
\end{definition}

\begin{figure}[ht]
\begin{center}
\vspace{-10pt}
\setlength{\unitlength}{2pt}
\begin{picture}(140,25)(0,-3)
\thicklines
\put(-7,20){\makebox(0,0){$\TT_1$}}

  \put(10,20){\line(1,0){20}}

\multiput(10,20)(20,0){2}{\circle*{2}}

{\tiny
\put(20,22.5){\makebox(0,0){$1$}}
}

  \put(0,5){\line(1,0){140}}

\multiput(10,5)(20,0){7}{\circle*{2}}

\put(-7,5){\makebox(0,0){$\TT_2$}}

{\tiny
\multiput(20,7.5)(40,0){3}{\makebox(0,0){$2$}}
\multiput(40,7.5)(40,0){3}{\makebox(0,0){$1$}}
}
\end{picture}

\begin{picture}(120,53)(0,-5)
\thicklines
  \put(0,20){\line(3,1){30}}
  \put(30,0){\line(3,1){30}}
  \put(60,20){\line(3,1){30}}

  \put(60,20){\line(-3,1){30}}
  \put(90,0){\line(-3,1){30}}
  \put(120,20){\line(-3,1){30}}

  \put(30,30){\line(0,1){10}}
  \put(60,10){\line(0,1){10}}
  \put(90,30){\line(0,1){10}}

  \put(30,40){\line(3,1){10}}
  \put(90,0){\line(3,1){10}}
  \put(90,40){\line(3,1){10}}
  \put(120,20){\line(3,1){10}}

  \put(0,20){\line(-3,1){10}}
  \put(30,40){\line(-3,1){10}}
  \put(90,40){\line(-3,1){10}}
  \put(30,0){\line(-3,1){10}}

  \put(0,20){\line(0,-1){5}}
  \put(30,0){\line(0,-1){5}}
  \put(90,0){\line(0,-1){5}}
  \put(120,20){\line(0,-1){5}}

  \put(0,20){\circle*{2}}
  \put(30,0){\circle*{2}}
  \put(30,30){\circle*{2}}
  \put(30,40){\circle*{2}}
  \put(60,10){\circle*{2}}
  \put(60,20){\circle*{2}}
  \put(90,0){\circle*{2}}
  \put(90,30){\circle*{2}}
  \put(90,40){\circle*{2}}
  \put(120,20){\circle*{2}}

\put(-16,40){\makebox(0,0){$\TT_3$}}

{\tiny
\put(15,28){\makebox(0,0){$2$}}
\put(75,28){\makebox(0,0){$2$}}
\put(45,8){\makebox(0,0){$2$}}

\put(45,28){\makebox(0,0){$1$}}
\put(105,28){\makebox(0,0){$1$}}
\put(75,8){\makebox(0,0){$1$}}

\put(32,35){\makebox(0,0){$3$}}
\put(92,35){\makebox(0,0){$3$}}
\put(62,15){\makebox(0,0){$3$}}
}

\end{picture}
\vspace{-20pt}
\end{center}
\caption{The $n$-regular trees $\TT_n$ for $n=1,2,3$.}
\label{fig:treeT3}
\end{figure}

We will use notation associated with the tree $\TT_n$ to keep track of the 
various labeled seeds that can be obtained by iterated mutations from a given initial seed.

\begin{figure}[ht]
\begin{center}
\vspace{-8pt}
\setlength{\unitlength}{1.5pt}
\begin{picture}(200,75)(0,13)

\put(100,50){\makebox(0,0){$x_1,x_2,x_3$}}
\put(100,50){\oval(36,12)}

\put(100,18){\makebox(0,0){$x_1,x_2,x'_3$}}
\put(100,18){\oval(32,12)}
\put(103,34){\makebox(0,0){\scriptsize 3}}
\put(100,44){\line(0,-1){20}}
\put(114.3,13.7){\line(1,-1){7}}
\put(85.7,13.7){\line(-1,-1){7}}

\put(114.5,55.5){\line(1,1){17.3}}
\put(145,78){\makebox(0,0){$x_1,x'_2,x_3$}}
\put(145,78){\oval(32,12)}
\put(120,66.5){\makebox(0,0){\scriptsize 2}}
\put(145,84){\line(0,1){10}}
\put(159.3,73.7){\line(1,-1){7}}

\put(85.5,55.5){\line(-1,1){17.3}}
\put(55,78){\makebox(0,0){$x'_1,x_2,x_3$}}
\put(55,78){\oval(32,12)}
\put(55,84){\line(0,1){10}}

\put(40.7,73.7){\line(-1,-1){7}}

\put(78,68){\makebox(0,0){\scriptsize 1}}

\end{picture}
\vspace{-8pt}
\end{center}
\caption{Clusters in a seed pattern, cf.\ Definition~\ref{def:seed-pattern}.} 
\label{fig:seed-pattern}
\end{figure}

\pagebreak[3]

\begin{definition}
\label{def:seed-pattern}
A \emph{seed pattern of rank $n$} is defined by assigning
a labeled seed $(\tilde \xx(t), \tilde B(t))$
to every vertex $t \in \TT_n$, so that the seeds assigned to the
endpoints of any edge $t \overunder{k}{} t'$ are obtained from each
other by the seed mutation in direction~$k$.
A~seed pattern is uniquely determined
by any one of its seeds.
See Figure~\ref{fig:seed-pattern}. 
\end{definition}

Now everything is in place for defining cluster algebras.

\begin{definition}
\label{def:cluster-algebra}
Let $(\tilde \xx(t), \tilde B(t))_{t\in\TT_n}$ be a seed pattern.
Let $\Xcal$ be the set of all cluster variables
	appearing in the various seeds $\xx(t)$ for $t\in \TT_n$.
We let the \emph{ground ring} be $R = \CC[x_{n+1}, \dots, x_{m}]$,
the polynomial ring generated by the frozen variables.
(A~common alternative is to take $R = \CC[x_{n+1}^{\pm 1}, \dots, x_{m}^{\pm 1}]$,
the ring of Laurent polynomials in the frozen variables, see 
Section~\ref{sec:rings-matrices} for an example.  
Sometimes the scalars are restricted to~$\QQ$, or even to~$\ZZ$.)

The \emph{cluster algebra}~$\AA$ (of geometric type, over~$R$)
associated with the given seed pattern 
is the $R$-subalgebra of the ambient field~$\FFcal$ 
generated by all cluster variables: $\AA= R[\Xcal]$.
To be more precise, a cluster algebra is the $R$-subalgebra~$\AA$
as above  together with a fixed seed pattern in it. 
	By definition, the \emph{rank} of~$\AA$ is the rank of the underlying
	seed pattern, i.e. the cardinality of any cluster of~$\AA$.
\end{definition}

A common way to describe a cluster algebra 
is to pick an \emph{initial (labeled) seed} $(\tilde \xx_\circ, \tilde B_\circ)$
in~$\FFcal$ and build a seed pattern from it. 
The corresponding cluster algebra, denoted $\AA(\tilde \xx_\circ,
\tilde B_\circ)$,
is generated over the ground ring~$R$ by all cluster variables
appearing in the seeds mutation
equivalent to~$(\tilde \xx_\circ, \tilde B_\circ)$.

\begin{remark}[cf.\ Exercise \ref{exercise:quiverexchange}] 
It can be shown that applying this construction
in each of the three settings discussed in
Chapter~\ref{ch:tp-examples},
one obtains cluster algebras naturally identified with 
the Pl\"ucker ring~$R_{2,m}$ (cf. Section~\ref{sec:plucker-rings}), 
the ring of invariants~$\CC[\SL_k]^U$ (cf. Section~\ref{sec:rings-baseaffine}),
	and the polynomial ring~$\CC[z_{11},z_{12},\dots,z_{kk}]$ 
(cf. Section~\ref{sec:rings-matrices}), respectively.

\end{remark}

\begin{remark}
As \cref{fig:A2} suggests, it 
is often more natural to work with 
\emph{(unlabeled) seeds}, 
which differ from the labeled ones in that we identify two seeds
$(\tilde \xx,\tilde B)$
and $(\tilde \xx',\tilde B')$
in which $\xx'$ is a permutation of~$\xx$,
and $\tilde B'$ is obtained from~$\tilde B$
by the corresponding permutations of rows and columns. 
We note that ignoring the labeling does not affect the resulting cluster
algebra in a meaningful way.
Accordingly, we will occasionally treat 
	$\xx$ (and/or~$\tilde\xx$) as a set rather than as a labeled
	sequence. 
\end{remark}

\begin{remark}
Many questions arising in cluster algebra theory and its applications 
do not really concern cluster algebras as such. 
These are questions which are not about commutative rings carrying a cluster
structure;
rather, they are about seed patterns and the birational
transformations that relate extended clusters to each other. 
For those questions, the choice of the ground ring 
is immaterial: the formulas remain the same regardless. 
\end{remark}

\begin{remark}\label{rem:opposite}
Since any free generating collection of $m$ elements in~$\FFcal$ 
can be mapped to any other such collection by an automorphism
of~$\FFcal$, the choice of the initial extended cluster~$\tilde
\xx_\circ$ is largely inconsequential: 
the cluster algebra $\AA(\tilde \xx_\circ, \tilde B_\circ)$ is determined, 
up to an isomorphism preserving all the matrices~$\tilde B(t)$,  
by the initial extended exchange matrix~$\tilde B_\circ$,
and indeed by its mutation equivalence class.
Also, replacing $\tilde{B}_\circ$ by $-\tilde{B}_\circ$ yields
essentially the same cluster algebra (all matrices $\tilde{B}(t)$
change their sign). 
\end{remark}

\begin{remark}
The same commutative ring (or two isomorphic rings) 
can carry very different cluster structures.
One can construct two seed patterns whose 
sets of exchange matrices are disjoint from each other,
yet the two rings generated by their respective sets of cluster variables
are isomorphic. 
See for example \cite[Example 6.3.1 and 6.3.2]{chapter6}.
\end{remark}

\begin{remark}
We will soon encounter many examples in which different vertices of the
tree~$\TT_n$ correspond to identical labeled or unlabeled seeds.
In spite of that, the set $\mathcal{X}$ of cluster variables will
typically be infinite. 
Note that this does not preclude a cluster algebra $\AA$ from being finitely
generated (which is often the case). 
We shall also see in Section~\ref{sec:generators+relations} 
	(see Example~\ref{ex:Gr36}) that even when $\mathcal{X}$ is finite, 
the exchange relations \eqref{eq:exch-rel-geom} do not always
generate the defining ideal of~$\AA$, i.e.
	 the ideal of all relations satisfied
	by the cluster variables $\mathcal{X}$.
\end{remark}

\section{Examples of rank $1$ and~$2$}
\label{sec:rank12}

In this section, we look at some examples of cluster algebras of small
rank. 

\subsection*{Rank~1} 
This case is very simple.  
The tree $\TT_1$ has two vertices, so we~only have two seeds, and two
clusters $(x_1)$ and~$(x_1')$.  
The extended exchange matrix $\tilde B_\circ$ can be any $m\times 1$ matrix whose top
entry is~$0$.

The~single exchange relation has the form $x_1\,x'_1=M_1+M_2$ 
where 
$M_1$~and~$M_2$ are mono\-mials in the frozen variables $x_2,\dots,x_m$ which do not
share a common factor~$x_i$. 
The cluster algebra is generated by $x_1, x'_1, x_2,\dots,x_m$,
subject to this relation, and lies inside 
the ambient field $\FFcal=\CC(x_1, x_2,\dots, x_m)$. 

Simple as they might be, cluster algebras of rank~1
do arise ``in nature," 
cf.\ Examples~\ref{ex:SL_2} and~\ref{ex:U-in-SL_3}. 
We  give two more examples here.
Additional examples will appear in Chapter~\ref{ch:rings}.

\begin{example}[cf.\ Section~\ref{sec:baseaffine}]
\label{ex:base-affine-SL3}
Let $U\subset G=\SL_3(\CC)$ be the subgroup of unipotent lower
triangular $3 \times 3$ matrices.
The ring $\CC[G]^U$
is generated by the six flag minors~$P_J$,
for $J$ a nonempty proper subset of $\{1,2,3\}$.
This ring has the structure of a cluster algebra of rank~1 in which
\begin{itemize}[leftmargin=.15in]
\item
the ambient field is $\CC(P_1,P_2,P_3,P_{12},P_{23})$;
\item
the frozen variables are $P_1,P_3,P_{12},P_{23}$;
\item
the cluster variables are $P_2$ and $P_{13}$;
\item
the single exchange relation is
$P_2 P_{13} = P_1 P_{23} + P_3 P_{12}$.
\end{itemize}
The two seeds of this cluster algebra correspond to the two wiring diagrams with 
$3$ strands.
Their respective sets of chamber minors are the two extended clusters
$\{P_2, P_1, P_3, P_{12}, P_{23}\}$ (cf.\ Figure~\ref{fig:chamber-sets0})
and $\{P_{13}, P_1, P_3, P_{12}, P_{23}\}$.
\end{example}

\begin{example}[cf.\ Example~\ref{ex:U-in-SL_3}]
\label{ex:U}
The coordinate ring of the subgroup $U^+$ of unipotent upper-triangular $3\times 3$ matrices
\[
\begin{bmatrix}
1  & a & b \\
0  & 1 & c \\
0  & 0 & 1
\end{bmatrix}
\in\SL_3(\CC)
\]
is 
$\CC[a,b,c]$.
This ring has the structure of a cluster algebra of rank~1 in~which
\begin{itemize}[leftmargin=.15in]
\item
	the ambient field is $\FFcal=\CC(a,b,c)=\CC(a,b,ac-b)$;
\item
the frozen variables are $b$ and $P=ac-b$;
\item
the cluster variables are $a$ and~$c$;
\item
the single exchange relation is
$ac = P+b$.
\end{itemize}
\end{example}

\subsection*{Rank~2} 

Any $2\times 2$ skew-symmetrizable matrix looks like this: 
\begin{equation}
\label{eq:Bbc}
\pm 
\begin{bmatrix}
0 & b\\
-c & 0
\end{bmatrix}, 
\end{equation}
for some integers $b$ and $c$ which are either both positive, or both
equal to~$0$. 
Applying a mutation $\mu_1$ or $\mu_2$
to a matrix of the form~\eqref{eq:Bbc} simply changes its sign. 

\begin{example}
Let $b=c=0$, i.e. the top two rows of the $m\times 2$ extended 
exchange matrix consist entirely of $0$'s.
Then the two mutations commute, because each 
$\mu_k$ changes the sign of the entries in column $k$ of the extended
exchange matrix while leaving the other column untouched; as the two
matrix columns do not affect each other,
 the story reduces to two rank~1 exchange patterns. 
We get four cluster variables $x_1,x_2,x_1',x_2'$, 
 four clusters
$(x_1,x_2)$, $(x_1',x_2)$, $(x_1,x_2')$, and $(x_1',x_2')$, and 
two exchange relations of the form $x_1x_1'=M_1+M_2$ and
\hbox{$x_2x_2'=M_3+M_4$},
where $M_1,M_2,M_3,M_4$ are monomials in the frozen variables.
\end{example}

For the rest of this section, we assume that $b>0$ and $c>0$. 
We denote the cluster variables in our cluster algebra $\AA$ of rank~2 by
\[
\dots, z_{-2}, z_{-1}, z_0, z_1, z_2, \dots,
\] 
so that the seed pattern looks like this:

\begin{equation*}
\cdots 
\overunder{1}{}
\begin{blockarray}{cc}
\begin{block}{(cc)}
z_1 & z_0 \\
\end{block}\\[-.2in]
\begin{block}{[cc]}
0 & -b\\
c & 0\\
\end{block}
\end{blockarray}
\overunder{2}{}
\begin{blockarray}{cc}
\begin{block}{(cc)}
z_1 & z_2 \\
\end{block}\\[-.2in]
\begin{block}{[cc]}
0 & b\\
-c & 0\\
\end{block}
\end{blockarray}
\overunder{1}{}
\begin{blockarray}{cc}
\begin{block}{(cc)}
z_3 & z_2 \\
\end{block}\\[-.2in]
\begin{block}{[cc]}
0 & -b\\
c & 0\\
\end{block}
\end{blockarray}
\overunder{2}{}
\begin{blockarray}{cc}
\begin{block}{(cc)}
z_3 & z_4 \\
\end{block}\\[-.2in]
\begin{block}{[cc]}
0 & b\\
-c & 0\\
\end{block}
\end{blockarray}
\overunder{1}{}
\cdots
\vspace{-.1in}
\end{equation*}
where we placed each cluster on top of the corresponding exchange
matrix.
(The extended exchange matrix may have additional rows.)  

We denote by $\AA=\AA(b,c)$ a cluster algebra of rank~2
which has exchange matrices 
$\pm\left[\begin{smallmatrix}
0 & b\\
-c & 0
\end{smallmatrix}\right]$
and no frozen variables. 
(Cluster algebras without frozen variables are generally said to have 
\emph{trivial coefficients}.) 
\index{trivial coefficients}
The exchange relations in $\AA(b,c)$ are, in the notation introduced above: 
\begin{equation}
\label{eq:Abc}
z_{k-1}\,z_{k+1}=
\begin{cases}
z_k^c + 1 & \text{if $k$ is even;}\\[.05in]
z_k^b + 1 & \text{if $k$ is odd.}
\end{cases}
\end{equation}

\begin{example}
\label{example:A(1,1)}
The cluster variables in the cluster algebra $\AA(1,1)$ 
with trivial coefficients satisfy the recurrence 
\begin{equation}
\label{eq:pentagon-recurrence}
z_{k-1}\,z_{k+1}=z_k+1. 
\end{equation}
Expressing everything in terms of the initial cluster $(z_1,z_2)$, we get: 
\begin{equation}
\label{eq:z3...z6-A2}
z_3=\frac{z_2+1}{z_1},\ \ 
z_4
=\frac{z_1+z_2+1}{z_1 z_2},\ \ 
z_5=\frac{z_1+1}{z_2},\ \ 
z_6=z_1,\ \ 
z_7=z_2, \ \dots, 
\end{equation}
so the sequence is $5$-periodic! 
Thus in this case, we have only 5 distinct cluster variables.
In the seed pattern, we will have:
\begin{equation*}
\cdots 
\overunder{2}{}
\begin{blockarray}{cc}
\begin{block}{(cc)}
z_1 & z_2 \\
\end{block}\\[-.2in]
\begin{block}{[cc]}
0 & 1\\
-1 & 0\\
\end{block}
\end{blockarray}
\overunder{1}{}
\begin{blockarray}{cc}
\begin{block}{(cc)}
z_3 & z_2 \\
\end{block}\\[-.2in]
\begin{block}{[cc]}
0 & -1\\
1 & 0\\
\end{block}
\end{blockarray}
\overunder{2}{}
\cdots
\overunder{1}{}
\begin{blockarray}{cc}
\begin{block}{(cc)}
z_7 & z_6 \\
\end{block}\\[-.2in]
\begin{block}{[cc]}
0 & -1\\
1 & 0\\
\end{block}
\end{blockarray}
\overunder{2}{}
\cdots
\vspace{-.1in}
\end{equation*}
Note that even though the labeled seeds containing the clusters
$(z_1,z_2)$ and $(z_7,z_6)$ are different,
the corresponding unlabeled seeds coincide.
Just switch $z_6$ and~$z_7$,
and interchange the rows and the columns in the associated exchange
matrix. 
Thus, this exchange pattern has 5 distinct (unlabeled)~seeds.
\end{example}

\begin{remark}
The recurrence \eqref{eq:pentagon-recurrence} arises in different
mathematical contexts such as dilogarithm identities 
(cf., e.g., bibliographical pointers in~\cite[Section~1.1]{pcmi}), 
the Napier-Gauss \emph{Pentagramma Mirificum}
(cf.\ \cite{cayley-pentagramma} and
\cite[Section~12.7]{coxeter-non-euclidean}) 
and Coxeter's frieze patterns~\cite{coxeter-frieze}. 
\end{remark}

\begin{example}
We now keep the same exchange matrices but introduce a single frozen variable~$y$. 
Consider a seed pattern that looks like this: 
{\small \vspace{-5pt}
\begin{equation*}
\cdots
\begin{blockarray}{cc}
z_1 & \!\!\!z_2 \\
\begin{block}{[cc]}
0 & \!\!\!1\\
-1 &\!\! \!0\\
p & \!\!\!q \\
\end{block}
\end{blockarray}
\smalloverunder{1}{}
\begin{blockarray}{cc}
z_3 & \!\!\!z_2 \\
\begin{block}{[cc]}
0 & \!\!\!-1\\
1 & \!\!\!0\\
-p & \!\!\!p\!+\!q\\
\end{block}
\end{blockarray}
\smalloverunder{2}{}
\begin{blockarray}{cc}
z_3 & \!\!\!z_4 \\
\begin{block}{[cc]}
0 & \!\!\!1\\
-1 & \!\!\!0\\
q & \!\!\!-p\!-\!q\\
\end{block}
\end{blockarray}
\smalloverunder{1}{}
\begin{blockarray}{cc}
z_5 & \!\!\!z_4 \\
\begin{block}{[cc]}
0 & \!\!\!-1\\
1 & \!\!\!0\\
-q & \!\!\!-p\\
\end{block}
\end{blockarray}
\smalloverunder{2}{}
\begin{blockarray}{cc}
z_5 & \!\!\!z_6 \\
\begin{block}{[cc]}
0 & \!\!\!1\\
-1 & \!\!\!0\\
-q & \!\!\!p\\
\end{block}
\end{blockarray}
\smalloverunder{1}{}
\begin{blockarray}{cc}
z_7 & \!\!\!z_6 \\
\begin{block}{[cc]}
0 & \!\!\!-1\\
1 & \!\!\!0\\
q & \!\!\!p\\
\end{block}
\end{blockarray}
\cdots,
\vspace{-.15in}
\end{equation*}}where 
$p$ and $q$ are nonnegative integers. 
Relabeling the rows and columns to keep the $2\times 2$ exchange matrices
invariant, we get 
{\small \vspace{-5pt}
\begin{equation*}
\cdots
\begin{blockarray}{cc}
z_1 & \!\!\!z_2 \\
\begin{block}{[cc]}
0 & \!\!\!1\\
-1 &\!\! \!0\\
p & \!\!\!q \\
\end{block}
\end{blockarray}
\smalloverunder{}{}
\begin{blockarray}{cc}
z_2 & \!\!\!z_3 \\
\begin{block}{[cc]}
0 & \!\!\!1\\
-1 & \!\!\!0\\
p\!+\!q & \!\!\!-p\\
\end{block}
\end{blockarray}
\smalloverunder{}{}
\begin{blockarray}{cc}
z_3 & \!\!\!z_4 \\
\begin{block}{[cc]}
0 & \!\!\!1\\
-1 & \!\!\!0\\
q & \!\!\!-p\!-\!q\\
\end{block}
\end{blockarray}
\smalloverunder{}{}
\begin{blockarray}{cc}
z_4 & \!\!\!z_5 \\
\begin{block}{[cc]}
0 & \!\!\!1\\
-1 & \!\!\!0\\
-p & \!\!\!-q\\
\end{block}
\end{blockarray}
\smalloverunder{}{}
\begin{blockarray}{cc}
z_5 & \!\!\!z_6 \\
\begin{block}{[cc]}
0 & \!\!\!1\\
-1 & \!\!\!0\\
-q & \!\!\!p\\
\end{block}
\end{blockarray}
\smalloverunder{}{}
\begin{blockarray}{cc}
z_6 & \!\!\!z_7 \\
\begin{block}{[cc]}
0 & \!\!\!1\\
-1 & \!\!\!0\\
p & \!\!\!q\\
\end{block}
\end{blockarray}
\cdots,
\vspace{-.15in}
\end{equation*}}so 
the sequence of extended exchange matrices remains 5-periodic. 
We then compute the cluster variables: \vspace{-5pt}
\[
z_3 = \frac{z_2+y^p}{z_1},{\ }
z_4 = \frac{y^{p+q} z_1 + z_2 + y^p}{z_1 z_2},{\ }
z_5 = \frac{y^q z_1 + 1}{z_2},{\ }
z_6 = z_1, {\ }
z_7 = z_2\,;
\]
the 5-periodicity persists! 
Just as in the case of trivial coefficients, there are five
distinct cluster variables overall, and five distinct
unlabeled~seeds.

The above computations were based on the assumption that 
both entries in the third row of the initial extended exchange matrix 
are nonnegative. 
In fact, this condition is not required for 5-periodicity. 
Note that we could start with an initial seed containing the
cluster $(z_i,z_{i+1})$,
for any $i\in\{1,2,3,4,5\}$, together with the associated extended exchange matrix in the relabeled sequence above, 
and get the same 5-periodic behavior. 
Since any row vector in $\ZZ^2$ has the form $(p,q)$, $(p+q,-p)$,
$(q,-p-q)$, $(-p,-q)$, or $(-q,p)$, for some $p,q\ge 0$
(see  Figure~\ref{fig:5-partition}), we conclude that 
any seed pattern with extended exchange matrices of the form 
$\pm\left[\begin{smallmatrix}
0 & 1\\
-1 & 0\\
* & *
\end{smallmatrix}\right]$
has exactly five seeds. 

\begin{figure}[h]
\begin{center}
\setlength{\unitlength}{1.6pt}
\begin{picture}(80,60)(0,3)
\thicklines
\put(20,20){\makebox(0,0){\small $(-p,-q)$}}
\put(20,55){\makebox(0,0){\small $(-q,p)$}}
\put(60,55){\makebox(0,0){\small $(p,q)$}}
\put(57,3){\makebox(0,0){\small $(q,-p\!-\!q)$}}
\put(75,28){\makebox(0,0){\small $(p\!+\!q,-p)$}}
\put(0,40){\line(1,0){80}}
\put(40,0){\line(0,1){70}}
\put(40,40){\line(1,-1){40}}
\end{picture}
\vspace{-8pt}
\end{center}
\caption{Five types of ``frozen rows'' in extended exchange matrices
with top rows $(0,1)$ and $(-1,0)$. 
Within each of the five cones, 
the points are parameterized by $p,q\ge 0$.}
\label{fig:5-partition}
\end{figure}

As we shall later see, the general case of a seed pattern with exchange
matrices $\pm\left[\begin{smallmatrix}
0 & 1\\
-1 & 0
\end{smallmatrix}\right]$
and an arbitrary number of frozen variables
exhibits the~same qualitative behaviour: 
there will still be five cluster variables and five~seeds.
\end{example}

\begin{example}
\label{example:A(1,2)}
The cluster variables in the cluster algebra $\AA(1,2)$ 
satisfy the recurrence 
\begin{equation}
\label{eq:exch-rel-B2}
z_{k-1}\,z_{k+1}=
\begin{cases}
z_k^2 + 1 & \text{if $k$ is even;}\\[.05in]
z_k + 1 & \text{if $k$ is odd.}
\end{cases}
\end{equation}
Expressing everything in terms of the initial cluster $(z_1,z_2)$, we get: 
\[
z_3=\frac{z_2^2+1}{z_1},\ \ 
z_4
=\frac{z_2^2+z_1+1}{z_1 z_2},\ \ 
z_5=\frac{z_1^2+z_2^2+2z_1 +1}{z_1 z_2^2},\ \ 
z_6=\frac{z_1+1}{z_2},\ \ 
\]
and then 
$z_7=z_1$ and 
$z_8=z_2$,
so the sequence is $6$-periodic! 
Thus in this case, we have only 6 distinct cluster variables,
and 6 distinct seeds.
\end{example}

\begin{exercise}
Compute the cluster variables for the cluster algebra with the initial
extended exchange matrix $\left[\begin{smallmatrix} 
0 & 1\\
-2 & 0\\
p & q
\end{smallmatrix}\right]$. 
\end{exercise}

\begin{exercise}
Compute the cluster variables of the 
cluster algebra~$\mathcal A(1,3)$.
(Start by evaluating them in the specialization
$z_1=z_2 = 1$; notice
that all the numbers will be integers.) 
\end{exercise}

\begin{example}
Consider the cluster algebra $\AA(1,4)$.
Setting $z_1 = z_2 = 1$ and applying the recurrence~\eqref{eq:Abc}, 
we see that the cluster variables~$z_3,z_4,\dots$ 
specialize to the following values: 
\[
2, 3, 41, 14, 937, 67, 21506, 321, 493697, 1538,
11333521, 7369, 260177282, \dots
\]
It is not hard to show that this sequence is not periodic ---
so the (unspecialized) sequence of cluster variables is not periodic
either.

The good news is that all these numbers are integers. 
Why does this happen? 
To understand this, let us recur\-sively compute the 
cluster variables $z_3,z_4,\dots$ in terms of $z_1$ and~$z_2$: 
\begin{align*}
z_3 &= \frac{z_2^4+1}{z_1},\\
z_4 &= \frac{z_3+1}{z_2} = \frac{z_2^4+z_1+1}{z_1 z_2},\\
z_5 &= \frac{z_4^4+1}{z_3} \\
&= 
\frac{z_2^{12} + 4 z_1 z_2^8 + 3 z_2^8 + 6 z_1^2 z_2^4+8 z_1 z_2^4
+z_1^4+3 z_2^4+4 z_1^3+6 z_1^2+4 z_1 + 1}
{z_1^3 z_2^4},\\
z_6 &= \frac{z_5+1}{z_4} 
= \frac{z_2^8+3z_1 z_2^4+2z_2^4+z_1^3+3 z_1^2+
3z_1 + 1}{z_1^2 z_2^3},\ \text{etc.}
\end{align*}
Now we see what is going on: 
the evaluations of these expressions at $z_1 \!=\! z_2 =1$ are integers 
because they are 
\emph{Laurent polynomials} in~$z_1$ and~$z_2$, i.e., 
their denominators are monomials. 
(This is by no means to be expected: 
for example, the computation of $z_6$ involves
dividing by $z_4 \!=\! \frac{z_2^4+z_1+1}{z_1 z_2}$.) 
\end{example}

\newpage

\section{The Laurent phenomenon}
\label{sec:laurent-geometric}

The examples of Laurentness that we have seen before are special cases
of the following general phenomenon. 

\begin{theorem}
\label{thm:Laurent}
In a cluster algebra of geometric type, 
each cluster variable can be expressed as a Laurent polynomial with
integer coefficients in the elements
of any extended cluster.
\end{theorem}

The rest of this section is devoted to the proof of Theorem~\ref{thm:Laurent}. 
First, we state a simple auxiliary lemma which can be obtained by direct inspection of
the exchange relations~\eqref{eq:exch-rel-geom}. 

\begin{lemma}
\label{lem:trivialization}
Let $\tilde B_\circ$ be an $m \times n$ extended exchange matrix. 
Let $\tilde B'_\circ$ be the matrix obtained from $\tilde B_\circ$
by deleting the rows labeled by a subset $I\subset\{n+1,\dots,m\}$.
Then the formulas expressing the cluster variables in a cluster algebra 
$\AA(\tilde\xx'_\circ,\tilde B'_\circ)$ in terms of the initial extended
cluster~$\tilde\xx'_\circ$ 
can be obtained from their counterparts for 
$\AA(\tilde\xx_\circ,\tilde B_\circ)$
by specializing the frozen variables $x_i$ ($i\in I$) 
to~$1$, and relabeling the remaining variables accordingly. 
\end{lemma}

\begin{remark}
\label{rem:trivialization}
A specialization of the kind described in
Lemma~\ref{lem:trivialization} sends Laurent polynomials to Laurent
polynomials. This means that if we add extra frozen variables to
the initial seed and establish Laurentness of an arbitrary  cluster
variable in this modified setting,
it would then imply the Laurentness of the cluster variable's counterpart in the original setting.
\end{remark}

Let us set up the notation needed for the proof of
Theorem~\ref{thm:Laurent}:
\begin{itemize}[leftmargin=.15in]
\item
$t_\circ\in\TT_n$ is an (arbitrarily chosen) initial vertex; 
\item
$(\tilde\xx_\circ,\tilde B_\circ)$ is the initial seed;
\item
$\tilde\xx_\circ=(x_1,\dots,x_m)$ is the initial extended cluster;
\item
$\tilde B_\circ=(b_{ij}^0)$ is the initial $m\times n$ extended exchange matrix;
\item 
$t\in\TT_n$ is an arbitrary vertex;
\item
$x\in\xx(t)$ is a cluster variable at~$t$;
\item $t_1$ and $t_2$ are the first two vertices on the unique
	path 
 in~$\TT_n$ connecting~$t_\circ$~to~$t$, obtained via 
		mutations in direction $j$ then $k$, 
		so that locally we have
$t_\circ
\!\!\begin{array}{c}
\scriptstyle{j}\\[-.1in]
-\!\!-\\[-.1in]
\scriptstyle{}
\end{array}
\! 
t_1
\!\!\begin{array}{c}
\scriptstyle{k}\\[-.1in]
-\!\!-\\[-.1in]
\scriptstyle{}
\end{array}
\! 
t_2$; 
\item
$d$ is the length of this path,
i.e., the distance in $\TT_n$ between
  $t_\circ$ and~$t$;
\item
$\tilde\xx(t_1)=(\tilde\xx(t_\circ)-\{x_j\}) \cup\{x_j'\}$; 
\item
$\tilde\xx(t_2)=(\tilde\xx(t_1)-\{x_k\}) \cup\{x_k'\}$. 
\end{itemize}
We will prove the Laurentness of~$x$, viewed as a function
of~$\xx_\circ$, by induction on~$d$. 
(More precisely, the statement we prove by induction concerns
arbitrary seeds at distance $d$ from each other in arbitrary cluster
algebras of geometric type.) 
The base cases $d=1$ and $d=2$ are trivial. 

There are two possibilities to consider.

\smallskip

\noindent
\textbf{Case~1}: $b_{jk}^0=b_{kj}^0=0$.
Let $t_3$ be the vertex in~$\TT_n$ connected to~$t_\circ$ by an edge
labeled~$k$. 
Since $\mu_j$ and $\mu_k$ commute at~$t_\circ$
(cf.\ Exercise~\ref{ex:mat-mut-simple}(\ref{ex:mat-mut-simple-4})),
each of the two seeds at $t_1$ and~$t_3$, respectively,
lies at distance $d-1$ from a seed containing~$x$, 
and $\tilde\xx(t_3)=(\tilde\xx(t_\circ)-\{x_k\}) \cup\{x_k'\}$.

By the induction assumption, the cluster variable~$x$ is expressed as
a Laurent polynomial in terms of the extended
cluster~$\tilde\xx(t_1)=(x_1,\dots,x_j',\dots,x_m)$. 
Also, $x'_j= \frac{M_1+M_2}{x_j}$, where $M_1$ and $M_2$ are monomials in $x_1,\dots,x_m$. 
Substituting this into the aforementioned Laurent polynomial, 
we obtain a formula expressing $x$ in terms of~$\tilde\xx_\circ$.  
Another such formula is obtained by taking the Laurent polynomial
expression for $x$ in terms of~$\tilde\xx(t_3)$,
and substituting $x'_k= \frac{M_3+M_4}{x_k}$,
with $M_3$ and $M_4$ some monomials in $x_1,\dots,x_m$. 
Removing common factors, we obtain 
(necessarily identical) expressions for~$x$ as a ratio of coprime
polynomials in $x_1,\dots,x_m$, with monic denominator. 

Note that in the first computation, all non-monomial factors that
can potentially remain in the denominator must come from
$M_1+M_2$; in the second one, they can only come from~$M_3+M_4$.
If $M_1+M_2$ and $M_3+M_4$ were coprime to each other,
the Laurentness of~$x$ would follow. 
This coprimality however does not hold in general.  
(For example, if columns $j$ and~$k$ of~$\tilde B_\circ$ are equal to
each other, then $M_1+M_2=M_3+M_4$.) 
We can however~use a trick based on Lemma~\ref{lem:trivialization},
cf.\ Remark~\ref{rem:trivialization}. 
Let us introduce a new frozen variable~$x_{m+1}$ and
extend the matrix~$\tilde B_\circ$ by an extra row 
in which the $(m+1,j)$-entry is~1, and all other entries are~0. 
Now $M_1+M_2$ has become a binomial which has degree~1
in the variable~$x_{m+1}$. 
Note that the supports of $M_1$ and $M_2$ are disjoint,
so $M_1+M_2$ does not have a monomial factor, 
 and if it had a nontrivial factorization, that factorization
would give a nontrivial factorization of $1+x_{m+1}$ after specializing
the other cluster variables to $1$.  Therefore $M_1+M_2$ is irreducible;
moreover it 
cannot divide $M_3+M_4$ as the latter does not depend on~$x_{m+1}$.
So $M_1+M_2$ and $M_3+M_4$ are coprime to each other, and we are done with Case~1.

\smallskip

\noindent
\textbf{Case~2}: $b_{jk}^0 b_{kj}^0<0$.
This case is much harder. The general shape of the proof 
remains the same: we use induction on~$d$ together with a coprimality
argument assisted by the introduction of additional frozen variables. 
One new aspect of the proof is that we need to separately 
consider the case $d=3$ since the induction step relies on it. 

Without loss of generality we assume that $b_{jk}^0<0$ and $b_{kj}^0>0$.
Otherwise, change the signs of all extended exchange matrices; this
will not affect the formulas relating extended clusters to each
other, see
Exercise~\ref{ex:mat-mut-simple}(\ref{ex:mat-mut-simple-3}).  

We denote by $t_3\in\TT_n$ the vertex connected to~$t_2$ by an edge
labeled~$j$, and introduce notation 
\[
\tilde\xx(t_3)=(\tilde\xx(t_2)-\{x_j'\}) \cup\{x_j''\}
=(x_1,\dots,x_j'',\dots,x_k',\dots,x_m). 
\]
(Whether $j<k$ or $k<j$ is immaterial.)
See Figure~\ref{fig:3mutations}. 
Note that $t_3$ may or may not lie
on the unique path in $T_n$ connecting $t_0$ to~$t$.

\begin{figure}[ht]
\begin{center}
\setlength{\unitlength}{3pt}
\begin{picture}(65,11)(-5,16)
\thicklines
  \put(0,20){\line(1,0){60}}
  \multiput(0,20)(20,0){4}{\circle*{1.5}}
\put(10,22){\makebox(0,0){$j$}}
\put(30,22){\makebox(0,0){$k$}}
\put(50,22){\makebox(0,0){$j$}}
\put(0,23){\makebox(0,0){$t_\circ$}}
\put(20,23){\makebox(0,0){$t_1$}}
\put(40,23){\makebox(0,0){$t_2$}}
\put(60,23){\makebox(0,0){$t_3$}}
\put(0,15){\makebox(0,0){$x_j, x_k$}}
\put(20,15){\makebox(0,0){$x'_j, x_k$}}
\put(40,15){\makebox(0,0){$x'_j, x'_k$}}
\put(60,15){\makebox(0,0){$x''_j, x'_k$}}
\end{picture}
\end{center}
\caption{Cluster variables obtained via successive mutations $\mu_j$,
  $\mu_k$, $\mu_j$.}
\label{fig:3mutations}
\end{figure}

Note that among cluster variables obtained by at most three
mutations from the initial seed, 
those of the form~$x_j''$ are the only ones whose Laurentness is not
 obvious.  
Therefore to establish Case 2 when $d=3$, it is enough 
to prove \cref{lem:3mutations} below.

\begin{lemma}
\label{lem:3mutations}
The cluster variable $x_j''$ 
is a Laurent polynomial in~$\tilde\xx_\circ$. 
\end{lemma}

\proof
Let $\mu_j(\tilde B_\circ)=\tilde{B}(t_1)=(b_{ij})$
and $\mu_k(\tilde{B}(t_1))=\tilde{B}(t_2)= (b'_{ij})$ be the extended exchange
matrices at $t_1$ and~$t_2$, respectively. 
Our assumption $b_{jk}^0<0$ implies that $b_{jk}>0$ and $b^{'}_{kj}>0$. 

We will view each of the cluster variables $x_j'$, $x_k'$, $x_j''$ as
a rational function in the elements $x_1,\dots,x_m$ of the initial
extended cluster~$\tilde\xx_\circ$. 
The notation $P \sim Q$ will mean that 
$P$ and $Q$ differ by a monomial factor, i.e., 
$P=QM$ where $M$ is a Laurent monomial in $x_1,\dots,x_m$.
Given a polynomial $P$, the notation $Q\equiv R \mod P$ will
mean that  $Q-R=PS$ for
some Laurent polynomial~$S$.
We also define 
\[
P_j=P_j(x_1,\dots,x_m)= \prod_i x_i^{b_{ij}}+1.
\]

The relevant instances of the exchange
relation~\eqref{eq:exch-rel-geom} imply that
\begin{align}
\label{eq:exch-rel-xj'}
x'_j &\sim x_j^{-1} \biggl( \prod_i x_i^{b_{ij}}+1 \biggr) \sim P_j,\\
\label{eq:exch-rel-xk'}
x'_k &= x_k^{-1} \biggl(\! (x'_j)^{b_{jk}} \prod_{\substack{b_{ik}>0\\ i\neq j}} x_i^{b_{ik}}
+ 
\prod_{b_{ik}<0} x_i^{-b_{ik}} 
\!\biggr) \equiv 
x_k^{-1} 
\prod_{b_{ik}<0} x_i^{-b_{ik}} \bmod P_j 
,\\
\label{eq:exch-rel-xj''}
x''_j &\sim (x'_j)^{-1} \biggl( (x'_k)^{b^{'}_{kj}} \prod_{i\neq k} x_i^{b^{'}_{ij}}+1 \biggr).
\end{align}
To establish that $x_j''$ is a Laurent polynomial in $x_1,\dots,x_m$, we
need to show
that the second factor in \eqref{eq:exch-rel-xj''} is
divisible by $P_j$.

Working  modulo~$P_j$ we obtain 
\begin{align*}
(x'_k)^{b^{'}_{kj}} \prod_{i\neq k} x_i^{b^{'}_{ij}} +1 
 &\equiv
  \biggl( x_k^{-1} \prod_{b_{ik}<0} x_i^{-b_{ik}} \biggr)^{b^{'}_{kj}} \prod_{i \neq k} x_i^{b^{'}_{ij}} 
   + 1\\
&= x_k^{b_{kj}} 
\prod_{b_{ik}<0} x_i^{b_{ik} b_{kj}}  
\prod_{i \neq k} x_i^{b^{'}_{ij}} 
  + 1
= \prod_i x_i^{b_{ij}} + 1 \equiv 0 ,
\end{align*}
as desired. In the last line, we used the fact that 
$b_{kj}<0$ and consequently 
\[
b^{'}_{ij} =
\begin{cases}
b_{ij}-b_{ik}b_{kj} & \text{if $b_{ik}<0$;}\\
b_{ij} & \text{if $b_{ik}\ge 0$ and $i\neq k$.\quad \qed}
\end{cases} 
\]

After proving the following technical lemma, we will be ready to 
complete the inductive proof of Case 2.

\begin{lemma}
\label{lem:coprime}
Suppose that distinct indices $q,r \in \{n+1,\dots,m\}$ are such that 
$b^0_{qj}=1$ and $b^0_{rk}=1$,
and moreover all other entries in rows 
$q$ and $r$ of $\tilde{B_\circ}$ are equal to~$0$.  
Then $x'_j$ is coprime to both $x'_k$ and~$x''_j$.
\end{lemma}

Here ``coprime'' means that those cluster variables, viewed as Laurent polynomials
in $\xx_\circ$, have no common non-monomial factor. 

\begin{proof}
Let us denote $b_{jk}^0=-b$ and $b^0_{kj}=c$. 
Recall that $b^0_{kj}>0$, so $b,c>0$. 
The local structure of the extended exchange matrices 
at $t_\circ$, $t_1$, and~$t_2$ 
at the intersections of rows $j,k,q, r$ and columns $j,k$ 
is as follows:
\[
\begin{array}{ccc}
\begin{array}{c|cccc}
&&j&k&\\
\hline
 & &\vdots & \vdots\\
j&\cdots&0&-b&\cdots\\
k&\cdots&c&0&\cdots\\
 & &\vdots & \vdots\\
q&\cdots&1&0&\cdots\\
r&\cdots&0&1&\cdots
\end{array}
{\ \ }
&
\begin{array}{c|cccc}
&&j&k&\\
\hline
 & &\vdots & \vdots\\
j&\cdots&0&b&\cdots\\
k&\cdots&-c&0&\cdots\\
 & &\vdots & \vdots\\
q&\cdots&-1&0&\cdots\\
r&\cdots&0&1&\cdots
\end{array}
&
{\ \ }
\begin{array}{c|cccc}
&&j&k&\\
\hline
 & &\vdots & \vdots\\
j&\cdots&0&-b&\cdots\\
k&\cdots&c&0&\cdots\\
 & &\vdots & \vdots\\
q&\cdots&-1&0&\cdots\\
r&\cdots&0&-1&\cdots
\end{array}
\\[.7in]
\ \tilde B_\circ &
\ \ \quad \mu_j(\tilde{B}) &
\qquad \mu_k(\mu_j(\tilde{B}))
\end{array}\!
\]

\noindent
We then have
\begin{align*}
x'_j &= x_j^{-1} (x_k^c \, x_q M_1 + M_2),\\
x'_k &= x_k^{-1} ((x'_j)^b \, x_r M_3 + M_4),\\
x''_j &= (x'_j)^{-1} (x_q M_5 + (x'_k)^c M_6),
\end{align*}
where $M_1,\dots,M_6$ are monomials in the $x_i$'s, with $i\notin\{j,k,q,r\}$. 
We see that 
$x'_j$ is linear in $x_q$
and hence irreducible (as a Laurent polynomial in~$\tilde\xx$), 
i.e., it cannot be written as a product
of two non-monomial factors.
Since $x'_j$ does not depend on~$x_r$, 
we conclude that $x'_k$ is linear in $x_r$ and hence irreducible,
and moreover coprime with~$x'_j$.

It remains to show that $x_j'$ and $x_j''$ are coprime. 
Note that $x'_k$ and $x''_j$ can be regarded as polynomials in~$x_r$;
we denote by $x'_k(0)$ and $x''_j(0)$ their specializations at
$x_r=0$. 
If we show that $x''_j(0)$ is coprime to $x'_j=x_j'(0)$, then we'll be done.
To this end, note that 
$x'_k(0) = x_k^{-1} M_4$.  Therefore
\[
x''_j(0) = x_j\,\frac{x_q M_5 + (x_k^{-1} M_4)^c M_6}{x_k^c \, x_q M_1 + M_2}\,.
\]
Here both the numerator and denominator 
are linear in $x_q$, and therefore the denominator 
(essentially,~$x'_j$) cannot divide the 
numerator more than once.
Also, $x'_j$ is irreducible.  This means that, in order for 
$x''_j(0)$ and $x'_j$ to fail to  be coprime, we would need 
the denominator of $x''_j(0)$ to divide the numerator at least twice.
Hence $x''_j(0)$ and $x'_j$ are coprime, as desired. 
\end{proof}

We are now ready to complete the proof of Case~2 of
Theorem~\ref{thm:Laurent}. 
We begin by augmenting the initial extended exchange 
matrix $\tilde{B}_\circ$ by two additional rows
corresponding to two new frozen variables $x_q$ and~$x_r$.~We~set 
\[
b^0_{qi}=\begin{cases}
1 & \text{if $i=j$,}\\
0 & \text{if $i\neq j$;}
\end{cases}
\qquad
b^0_{ri}=\begin{cases}
1 & \text{if $i=k$,}\\
0 & \text{if $i\neq k$,}
\end{cases}
\]
so as to satisfy the conditions of Lemma~\ref{lem:coprime}. 

By the induction assumption, the cluster variable~$x$ is expressed as
a Laurent polynomial in terms of each of the extended
clusters~$\tilde\xx(t_1)$ and~$\tilde\xx(t_3)$. 
The only elements of these clusters which do not appear
in~$\tilde\xx_\circ$ are $x_j'$, $x_k'$, and~$x_j''$, so
\[
x=\frac{\text{Laurent polynomial in~$\tilde\xx_\circ$}}{(x_j')^a}
=\frac{\text{Laurent polynomial in~$\tilde\xx_\circ$}}{(x_k')^b(x_j'')^c}\,,
\]
for $a,b,c\!\in\!\ZZ$. 
By Lemma~\ref{lem:3mutations}, $x_j'$, $x_k'$, and~$x_j''$ are Laurent polynomials~in~$\tilde\xx_\circ$. 
By Lemma~\ref{lem:coprime}, 
$x_j'$ is coprime to both $x_k'$ and~$x_j''$. 
The theorem~now follows by the same argument 
(based on Lemma~\ref{lem:trivialization}) 
that we used in~Case~1. 
\endproof

Theorem~\ref{thm:Laurent} can be sharpened as follows. 

\begin{theorem}
\label{th:Laurent-sharper}
In a cluster algebra of geometric type, 
frozen variables do not appear in the denominators 
of the Laurent polynomials expressing cluster variables in terms of an
initial extended cluster. 
\end{theorem}

Stated in our standard notation, \cref{th:Laurent-sharper}
asserts that each cluster variable is a Laurent polynomial 
in the initial cluster variables $x_1,\dots,x_n$, 
with coefficients in $\ZZ[x_{n+1},\dots,x_m]$. 

\begin{proof}
We borrow the notation from the proof of Theorem~\ref{thm:Laurent} above. 
Let $x$ be a cluster variable from a distant seed, and $x_r$ a 
frozen variable ($n<r\le m$). 
We will think of $x$ as a Laurent polynomial
$x(x_r)$ whose coefficients are integral Laurent polynomials in the
variables~$x_i$, with $i\neq r$. 
We want to show that $x$ is in fact a polynomial in~$x_r$;
Theorem~\ref{th:Laurent-sharper} will then follow by varying~$r$.

We will make use of the following trivial lemma.

\begin{lemma}
\label{lem:laurent-implies-poly}
Let $P$ and $Q$ be two polynomials (in any number of variables)
with coefficients in a domain~$S$, and with nonzero constant terms $a$ and~$b$,
respectively. If the ratio $P/Q$ is a Laurent polynomial over~$S$, then it
is in fact a polynomial over $S$ with the constant term~$a/b$.
\end{lemma}

Our proof of Theorem~\ref{th:Laurent-sharper}
proceeds by induction on~$d$, the smallest distance in~$\TT_n$
between the initial seed and a seed containing~$x$. 
We will inductively prove the following
strengthening of the desired statement:
\begin{center}
\begin{minipage}{4.5in}
\noindent
\emph{$x(x_r)$ is a polynomial in $x_r$
whose constant term $x(0)$ can be written as a subtraction-free rational
expression in the elements of $\tilde\xx_\circ-\{x_r\}$; 
in particular, $x(0)\neq 0$.}
\end{minipage}
\end{center}
\noindent
If $d= 0$, then $x \in \tilde\xx_\circ$, and there is nothing to prove.
If $d> 0$, then $x$ appears on the left-hand side of 
an exchange relation~\eqref{eq:exch-rel-geom}
$x x' = M_1 + M_2$, where $M_1$ and $M_2$ denote the monomials in the 
	exchange relation, and where $x'$ and the cluster variables in 
	 $M_1$ and $M_2$
 come from a seed located at distance $d-1$ from~$\tilde\xx_\circ$.
By definition of the exchange relation, 
the frozen variable $x_r$ will appear in at most one of $M_1$ and $M_2$.
Therefore if we express $x', M_1$, and $M_2$ in terms of~$\tilde\xx_\circ$,
	viewing them as Laurent polynomials in $x_r$ (whose coefficients
	are integral Laurent polynomials in the variables 
	$x_i$ with $i \neq r$), 
the inductive assumption implies that $M_1+M_2$ is a  polynomial
	in $x_r$ with nonzero constant term.
It now follows 
from Lemma~\ref{lem:laurent-implies-poly} 
	that  
	$x = \frac{M_1+M_2}{x'}$ is a polynomial in $x_r$ with 
	nonzero constant term.
\end{proof}

\newpage

\section{Connections to number theory}

\begin{example}[\emph{Markov triples}]
Consider the cluster algebra defined by the Markov quiver given in
Figure~\ref{fig:markov-quiver}. 
Since the quiver is invariant under mutations, exchange relations for
any cluster $(x_1,x_2,x_3)$ will look the same:
\begin{align*}
x_1'\,x_1 &=x_2^2+x_3^2,\\
x_2'\,x_2 &=x_1^2+x_3^2,\\
x_3'\,x_3 &=x_1^2+x_2^2.
\end{align*} \pagebreak[3]
If we start with the triple $(1,1,1)$ and mutate in all possible
directions, we will get an infinite set of triples in~$\ZZ^3$, 
including those shown in Figure~\ref{fig:markov-triples}.

\begin{figure}[ht]
\begin{center}
\setlength{\unitlength}{2pt}
\begin{picture}(150,110)(0,0)
\put(0,55){\makebox(0,0){(1, 1, 1)}} 
\put(30,55){\makebox(0,0){(1, 1, 2)}} 
\put(60,55){\makebox(0,0){(1, 5, 2)}} 
\put(90,25){\makebox(0,0){(29, 5, 2)}}
\put(90,85){\makebox(0,0){(1, 5, 13)}} 
\put(120,10){\makebox(0,0){(29, 169, 2)}} 
\put(120,40){\makebox(0,0){(29, 5, 433)}} 
\put(120,70){\makebox(0,0){(194, 5, 13)}} 
\put(120,100){\makebox(0,0){(1, 34, 13)}} 
\put(150,0){\makebox(0,0){(985,169,2)}} 
\put(150,20){\makebox(0,0){(29,169,14701)}} 
\put(150,30){\makebox(0,0){(29,37666,433)}} 
\put(150,50){\makebox(0,0){(6466,5,433)}} 
\put(150,60){\makebox(0,0){(194,5,2897)}} 
\put(150,80){\makebox(0,0){(194,7561,13)}} 
\put(150,90){\makebox(0,0){(1325,34,13)}} 
\put(150,110){\makebox(0,0){(1, 34, 89)}}

\put(10,55){\line(1,0){10}}
\put(40,55){\line(1,0){10}}
\put(70,60){\line(1,2){10}}
\put(70,50){\line(1,-2){10}}
\put(100,90){\line(1,1){7}}
\put(100,80){\line(1,-1){7}}
\put(100,30){\line(1,1){7}}
\put(100,20){\line(1,-1){7}}

\put(132,103){\line(1,1){5}}
\put(132,74){\line(2,3){2}}
\put(132,44){\line(2,3){3}}
\put(132,14){\line(2,3){2}}
\put(132,6){\line(1,-1){4.5}}
\put(132,36.5){\line(2,-3){2}}
\put(132,66){\line(2,-3){3}}
\put(131,97){\line(1,-1){4.5}}

\end{picture}
\end{center}
\caption{Markov triples.}
\label{fig:markov-triples}
\end{figure}

We next observe that all these triples satisfy the diophantine \emph{Markov
  equation} 
\[
x_1^2+x_2^2+x_3^2 = 3 x_1 x_2 x_3.
\]
To see this, verify that each mutation in our cluster algebra transforms
a solution of this equation (a \emph{Markov triple}) into another
solution. 
(This is an instance of \emph{Vieta jumping}, which replaces one root
of a quadratic equation by another root.)
Starting with the solution $(1,1,1)$, we get the tree of
Markov triples above. 
In fact, every Markov triple appears in this tree.
The celebrated (still open) Uniqueness Conjecture asserts 
that the maximal elements of Markov triples are all distinct. 
See~\cite{aigner-markov} for a detailed account. 

Returning to the cluster-algebraic interpretation of this example, we note
that more generally, the quantity
\[
\frac{x_1^2+x_2^2+x_3^2}{x_1x_2x_3}
\]
is invariant under mutations in this seed pattern. 
This is closely related to \emph{integrability} of the Markov
recurrence; cf.\ Chapter~\ref{ch:dynamical}.
\end{example}

\begin{example}[\emph{Fermat numbers}]

Sometime in the 1640's,
Pierre Fermat conjectured that for every positive integer~$n$,
the number $F_n=2^{2^n}+1$ is prime.
This was disproved in 1732 by Leonhard Euler, who discovered that
\[
F_5=2^{32}+1=641\cdot 6700417.
\]
Curiously, this factorization can be obtained by observing cluster
mutations. 

Consider the rank~$2$ cluster algebra
with the initial seed 
$(\tilde\xx,\tilde B)$ where
\[
\tilde\xx=(x_1,x_2,x_3),\quad 
\tilde B =
\left[
\begin{array}{rr}
0 &4 \\
-1& 0\\
1&-3
\end{array}
\right].
\]
The mutation $\mu_1$ produces a new extended cluster $\tilde\xx'=(x_1',x_2,x_3)$
where 
\[
x'_1\,x_1=x_2+x_3\,.
\]
Taking the specialization
\[
(x_1,x_2,x_3)=(3,-1,16), 
\]
we see that the mutated extended cluster specializes to
\[
(x_1',x_2,x_3)=(5,-1,16). 
\]
Applying the sharp version of the Laurent phenomenon 
(Theorems~\ref{thm:Laurent} and~\ref{th:Laurent-sharper})
to the initial extended cluster~$\tilde\xx$, 
we see that every cluster variable specializes to 
an integer (possibly) divided by a power of~$3$; applying the
same result to~$\tilde\xx'$, we conclude that every cluster variable is
an integer (possibly)  divided by a power of $5$.
Thus, every cluster variable specializes to an integer!
Now let us see which integers we get.
Alternately applying the mutations $\mu_1$
and~$\mu_2$, we obtain 
the following sequence,
with specialized cluster variables written on top of the extended exchange
matrices:
\begin{align*}
\begin{array}{rr}
3 &-1 \\
\hline
0 &4 \\
\!\!-1& 0\\
1&-3
\end{array}
\stackrel{\textstyle\mu_1}{\longrightarrow}
\begin{array}{rr}
5 &-1\\
\hline
0 &-4\\
1& 0\\
\!\!-1&1
\end{array}
\stackrel{\textstyle\mu_2}{\longrightarrow}
\begin{array}{rr}
5 &\!-641 \\
\hline
0 & 4 \\
\!\!-1& 0\\
0 &-1
\end{array}&
\stackrel{\textstyle\mu_1}{\longrightarrow}
\begin{array}{rr}
\!\!-128 &\!-641 \\
\hline
0 & -4 \\
1& 0\\
0 & -1
\end{array}
\stackrel{\textstyle\mu_2}{\longrightarrow}
\begin{array}{rr}
\!\!-128 &\!\frac{-F_5}{641} \\
\hline
0 & 4 \\
-1& 0\\
0 & 1
\end{array}
\end{align*}
This shows that $F_5/641$ is an integer, reproducing  Euler's discovery.
\end{example}

\pagebreak[3]

\begin{example} 
\label{ex:Somos4}
The \emph{Somos-4 sequence} $z_0,z_1,z_2,\dots$ 
is defined by
the initial conditions $z_0=z_1=z_2=z_3=1$ and the recurrence
\[
z_{m+2}z_{m-2}=z_{m+1}z_{m-1}+z_m^2\,.
\]
This sequence is named after M.~Somos who discovered it (and its various
generalizations) sometime in the 1980s;
see, e.g., \cite{bousquet-propp-west,hone-superintegrable} and references therein. 
 
The first several terms of the Somos-4 sequence are
\[
1,1,1,1,2,3,7,23,59,314,1529,8209,83313,620297,7869898,\dots
\]
---all integers! 
An explanation of the integrality of this sequence
can be given using cluster algebras.

Consider the exchange matrix shown in Figure~\ref{fig:4vertexquiver} along with the corresponding quiver.  
It is easy to check that mutating at the vertex labeled~1 produces a quiver that differs
from the original one by clockwise rotation by~$\pi/2$. 
It follows that subsequent quiver mutations at $2,3,4,1,2,3,4,1,\dots$ 
will generate a sequence of cluster variables satisfying the
Somos-4 recurrence above. 
In view of the Laurent phenomenon, the initial conditions
$z_0=z_1=z_2=z_3=1$ 
will result in a sequence of integers. 
\end{example}

\begin{figure}[h]
\begin{center}
\vspace{-.2in}
\begin{tabular}{cc}
$B = \tilde B = \begin{bmatrix}0 & -1 & 2 & -1 \\
                   1 & 0 & -3 & 2  \\
                   -2 & 3 & 0 & -1 \\
                   1 & -2 & 1 & 0 
  \end{bmatrix}
$
\hspace{.5in}
&
\setlength{\unitlength}{2.4pt}
\begin{picture}(20,20)(0,9)
\put(0,20){\makebox(0,0){$1$}} 
\put(20,20){\makebox(0,0){$2$}} 
\put(20,0){\makebox(0,0){$3$}} 
\put(0,0){\makebox(0,0){$4$}} 
\thicklines
\put(18,20){\vector(-1,0){16}}
\put(0,2){\vector(0,1){16}}
\put(17,18){\vector(-1,-1){16}}
\put(18,17){\vector(-1,-1){16}}
\put(2,0){\vector(1,0){16}}
\put(20,2){\vector(0,1){16}}
\put(21,2){\vector(0,1){16}}
\put(19,2){\vector(0,1){16}}
\put(2,19){\vector(1,-1){16}}
\put(1,18){\vector(1,-1){16}}
\end{picture}
\end{tabular}
\end{center}
		\caption{The exchange matrix and the quiver~\cite{fordy-marsh} associated with the Somos-4 sequence. 
		All four vertices of the quiver are mutable. } 
		\label{fig:4vertexquiver}
	\end{figure}

\begin{remark}
An alternative approach to establishing integrality of the Somos-4 and
other related sequences is based on explicit combinatorial
interpretations of their terms. 
In particular, the numbers~$z_m$ defined above can be shown to count
\emph{perfect matchings} in certain planar bipartite graphs,
see~\cite{speyer-perfect-matchings} and Figure~\ref{fig:somos-4}. 
\end{remark}

\begin{figure}[h]
\centering
\resizebox{5in}{!}{
        \begin{tikzpicture}
        
        \node[circle, very thick, draw = black, scale = 0.75, fill] (A) at (0, 0) {}; 

        \node[circle, very thick, draw = black, scale = 0.75] (B) at (1, 0) {};

        \node[circle, very thick, draw = black, scale = 0.75, fill] (C) at (1, 1) {}; 

        \node[circle, very thick, draw = black, scale = 0.75] (D) at (0, 1) {};

        \draw[very thick] (A)--(B)--(C)--(D)--(A); 

        picture 2

        \node[circle, very thick, draw = black, scale = 0.75, fill] (A') at (2.5, 0) {}; 

        \node[circle, very thick, draw = black, scale = 0.75] (B') at (3.5, 0) {};

        \node[circle, very thick, draw = black, scale = 0.75, fill] (C') at (3.5, 1) {}; 

        \node[circle, very thick, draw = black, scale = 0.75] (D') at (2.5, 1) {};

        \draw[very thick] (A')--(B')--(C')--(D')--(A'); 

        \node[circle, very thick, draw = black, scale = 0.75] (E') at (4.5, 2) {}; 

        \node[circle, very thick, draw = black, scale = 0.75, fill] (F') at (5.5, 2) {};  

        \node[circle, very thick, draw = black, scale = 0.75] (G') at (5.5, -1) {}; 

        \node[circle, very thick, draw = black, scale = 0.75, fill] (H') at (4.5, -1) {}; 

        \draw[very thick] (C')--(E')--(F')--(G')--(H')--(B'); 

        picture 3

        \node[circle, very thick, draw = black, scale = 0.75, fill] (A') at (7, 0) {}; 

        \node[circle, very thick, draw = black, scale = 0.75] (B') at (8, 0) {};

        \node[circle, very thick, draw = black, scale = 0.75, fill] (C') at (8, 1) {}; 

        \node[circle, very thick, draw = black, scale = 0.75] (D') at (7, 1) {};

        \draw[very thick] (A')--(B')--(C')--(D')--(A'); 

        \node[circle, very thick, draw = black, scale = 0.75] (E') at (9, 2) {}; 

        \node[circle, very thick, draw = black, scale = 0.75, fill] (F') at (10, 2) {};  

        \node[circle, very thick, draw = black, scale = 0.75] (G') at (10, -1) {}; 

        \node[circle, very thick, draw = black, scale = 0.75, fill] (H') at (9, -1) {}; 

        \draw[very thick] (C')--(E')--(F')--(G')--(H')--(B');

        \node[circle, very thick, draw = black, scale = 0.75] (I') at (11, 2) {};

        \node[circle, very thick, draw = black, scale = 0.75, fill] (J') at (11, 3) {}; 

        \node[circle, very thick, draw = black, scale = 0.75] (K') at (10, 3) {};

        \draw[very thick] (F')--(I')--(J')--(K')--(F'); 

        \node[circle, very thick, draw = black, scale = 0.75, fill] (O') at (12, 1) {}; 

        \draw[very thick] (O')--(I');

        \node[circle, very thick, draw = black, scale = 0.75] (I'') at (11, -2) {};

        \node[circle, very thick, draw = black, scale = 0.75, fill] (J'') at (11, -1) {}; 

        \node[circle, very thick, draw = black, scale = 0.75, fill] (K'') at (10, -2) {};

        \draw[very thick] (K'')--(I'')--(J'')--(G')--(K''); 

        \node[circle, very thick, draw = black, scale = 0.75] (L'') at (12, 0) {}; 

        \draw[very thick] (L'')--(J'');

        \draw[very thick] (L'')--(O');

picture 4 

        \node[circle, very thick, draw = black, scale = 0.75, fill] (A''') at (13.5, 0) {}; 

        \node[circle, very thick, draw = black, scale = 0.75] (B''') at (14.5, 0) {};

        \node[circle, very thick, draw = black, scale = 0.75, fill] (C''') at (14.5, 1) {}; 

        \node[circle, very thick, draw = black, scale = 0.75] (D''') at (13.5, 1) {};

        \draw[very thick] (A''')--(B''')--(C''')--(D''')--(A'''); 

        \node[circle, very thick, draw = black, scale = 0.75] (E''') at (15.5, 2) {}; 

        \node[circle, very thick, draw = black, scale = 0.75, fill] (F''') at (16.5, 2) {};  

        \node[circle, very thick, draw = black, scale = 0.75] (G''') at (16.5, -1) {}; 

        \node[circle, very thick, draw = black, scale = 0.75, fill] (H''') at (15.5, -1) {}; 

        \draw[very thick] (C''')--(E''')--(F''')--(G''')--(H''')--(B''');

        top right

        f' is bottom left 

        \node[circle, very thick, draw = black, scale = 0.75] (I''') at (17.5, 2) {};

        \node[circle, very thick, draw = black, scale = 0.75, fill] (J''') at (17.5, 3) {}; 

        \node[circle, very thick, draw = black, scale = 0.75] (K''') at (16.5, 3) {};

        \draw[very thick] (F''')--(I''')--(J''')--(K''')--(F'''); 

        \node[circle, very thick, draw = black, scale = 0.75] (L''') at (18.5, 4) {}; 

        \node[circle, very thick, draw = black, scale = 0.75, fill] (M''') at (19.5, 4) {};  

        \node[circle, very thick, draw = black, scale = 0.75] (N''') at (19.5, 1) {}; 

        \node[circle, very thick, draw = black, scale = 0.75, fill] (O''') at (18.5, 1) {}; 

        \draw[very thick] (J''')--(L''')--(M''')--(N''')--(O''')--(I''');

        \node[circle, very thick, draw = black, scale = 0.75] (I-) at (17.5, -2) {};

        \node[circle, very thick, draw = black, scale = 0.75, fill] (J-) at (17.5, -1) {}; 

        \node[circle, very thick, draw = black, scale = 0.75, fill] (K-) at (16.5, -2) {};

        \draw[very thick] (K-)--(I-)--(J-)--(G''')--(K-); 

        \node[circle, very thick, draw = black, scale = 0.75] (L-) at (18.5, 0) {}; 

        \node[circle, very thick, draw = black, scale = 0.75, fill] (M-) at (19.5, 0) {};  

        \node[circle, very thick, draw = black, scale = 0.75] (N-) at (19.5, -3) {}; 

        \node[circle, very thick, draw = black, scale = 0.75, fill] (O-) at (18.5, -3) {}; 

        \draw[very thick] (J-)--(L-)--(M-)--(N-)--(O-)--(I-);

        \draw[very thick] (O''')--(L-); 

        \draw[very thick] (N''')--(M-);

        \node[circle, very thick, draw = black, scale = 0.75] (Z) at (20.5, 0) {};

        \node[circle, very thick, draw = black, scale = 0.75, fill] (Y) at (20.5, 1) {};

        \draw[very thick] (N''')--(Y)--(Z)--(M-);

    \end{tikzpicture}
    }
\caption{The number of perfect matchings in each of these bipartite graphs is 2,
3, 7, 23, respectively (cf.\ the Somos-4
sequence).}
\label{fig:somos-4}
    \end{figure}

\begin{remark}\label{rem:elliptic}
Somos sequences and their various generalizations are intimately
related to the arithmetic of elliptic curves;
see, e.g., \cite{hone-swart} and references therein. 
Here is a typical result, stated here without proof 
(this version is due to D.~Speyer). 
Consider the elliptic curve
\[
y^2=1-8x+12x^2-4x^3,
\]
and let $x_m$ be the $x$-coordinate of the point 
$P+mQ$ where $P=(0,1)$, $Q=(1,-1)$,
and we are using the standard group law of the elliptic curve.
Then 
\[
x_m=\frac{z_{m-1}z_{m+1}}{z_m^2},
\]
where $(z_m)$ is the Somos-4 sequence above. 
\end{remark}

\begin{exercise}
Use the Laurent phenomenon to show that the sequence  
$z_0,z_1,z_2,\dots$ defined by
the initial conditions $z_0=z_1=z_2=1$ and the recurrence
\[
z_{m+3}z_{m}=z_{m+2}z_{m+1}+1
\]
consists entirely of integers. 
\end{exercise}

\begin{exercise}
Show that all elements of the sequence  
$z_0,z_1,z_2,\dots$ defined by the generalized Somos-4 recurrence
\[
z_{m+2}z_{m-2}=az_{m+1}z_{m-1}+bz_m^2
\]
are Laurent polynomials in $z_0,z_1,z_2,z_3$, with coefficients in $\ZZ[a,b]$. 
\end{exercise}

\begin{example} 
\label{ex:gen-somos}
The \emph{Somos-5} sequence 
\[
1, 1, 1, 1, 1, 2, 3, 5, 11, 37, 83, 274, 1217, 6161, 22833,
165713,\dots
\]
is defined by the recurrence relation
\begin{equation}
\label{eq:somos-5-recurrence}
z_m z_{m+5}  = z_{m+1} z_{m+4} + z_{m+2} z_{m+3} \quad (m = 1, 2, \dots)
\end{equation}
with the initial conditions $z_1 = \dots = z_5 = 1$.
\emph{A~priori}, one expects the numbers~$z_m$ to be rational---but
in fact, all of them are integers.
Once again, this is a consequence of a stronger statement: 
viewed as a function of $z_1,\dots, z_5$, every $z_m$ is a 
Laurent polynomial with integer coefficients.
To~prove this, we need to find a cluster algebra with an initial
cluster 
$(z_1, \dots, z_5)$ (no frozen variables) which has all
relations~\eqref{eq:somos-5-recurrence} 
among its exchange relations, so that
all the $z_m$ are among its cluster variables.
\end{example}

\begin{exercise}
Establish the integrality of all terms of the Somos-5 sequence
by examining the sequence of mutations 
\[
\mu_1,\mu_2,\mu_3,\mu_4,\mu_5,\mu_1,\mu_2,\mu_3,\mu_4,\mu_5,\dots,
\]
in the cluster algebra whose initial exchange matrix
(and the corresponding quiver) are shown below.

	\begin{figure}[h]
$B = \tilde B = \begin{bmatrix}0 & -1 & 1 & 1 & -1\\
                   1 & 0 & -2 & 0 & 1 \\
                   -1 & 2 & 0 & -2 & 1\\
                   -1 & 0 & 2 & 0 & -1\\
                   1 & -1 & -1 & 1 & 0
  \end{bmatrix}
$
\hspace{1in}
\setlength{\unitlength}{2.8pt}
\begin{picture}(20,20)(0,8)
\put(0,20){\makebox(0,0){$5$}} 
\put(20,20){\makebox(0,0){$2$}} 
\put(20,0){\makebox(0,0){$3$}} 
\put(0,0){\makebox(0,0){$4$}} 
\put(10,10){\makebox(0,0){$1$}} 
\thicklines
\qbezier(22,0)(35,35)(2,23)
\put(2,23){{\vector(-3,-1){1}}}
\put(18,20){\vector(-1,0){16}}
\put(17.5,17.5){\vector(-1,-1){7}}
\put(8.5,8.5){\vector(-1,-1){7}}
\put(0,18){\vector(0,-1){16}}
\put(2,0){\vector(1,0){16}}
\put(2,-1){\vector(1,0){16}}
\put(20,2){\vector(0,1){16}}
\put(21,2){\vector(0,1){16}}
\put(1.5,18.5){\vector(1,-1){7}}
\put(11,9){\vector(1,-1){7}}
\end{picture}
		\caption{The exchange matrix and the quiver associated with the Somos-5 sequence.}
	\end{figure}
\end{exercise}
\medskip

\begin{example}[see 
\cite{fordy-marsh}] 
Fix a positive integer $n$, and let $(a_1, \dots, a_{n-1})$ be a
\emph{palindromic} integer 
vector, that is, $a_i = a_{n-i}$ for $i = 1, \dots, n-1$.
Consider the sequence $z_1, z_2, \dots$ given by the recurrence
\begin{equation}
\label{eq:fordy-marsh-recurrence}
z_m z_{m+n}  = \prod_{i=1}^{n-1} z_{m+i}^{[a_i]_+} + \prod_{i=1}^{n-1} z_{m+i}^{[-a_i]_+}
 \quad (m = 1, 2, \dots)
\end{equation}
with indeterminates $z_1, \dots, z_n$ as the initial terms;
here we use the notation
\begin{equation}
\label{eq:positive-part}
[a]_+=\max(a,0). 
\end{equation}
(The recurrence \eqref{eq:somos-5-recurrence} is a special
case with $n=5$ and $(a_1, \dots, a_{n-1}) = (1,-1,-1,1)$.)
Then all the terms $z_m$ are integer Laurent polynomials in $z_1, \dots, z_n$.

To show this, we find an $n \times n$ skew-symmetric
integer matrix~$B=(b_{ij})$ 
such that $\mu_1(B)$ is obtained from~$B$ by the cyclic permutation
of its rows and columns, 
and such that its first column is given by $b_{i1} = a_{i-1}$ for
$i = 2, \dots, n$. If we can do so, then the sequence of mutations
\[
\mu_1,\mu_2,\mu_3,\dots,\mu_n,\mu_1,\mu_2,\dots,
\]
will produce the recurrence 
\eqref{eq:fordy-marsh-recurrence}.
It is not hard to see that setting 
\begin{equation*}
- b_{ji} = b_{ij} = a_{i-j} + \sum_{k=1}^{j-1}
([- a_{i-k}]_+ [a_{j-k}]_+  \ - \ [a_{i-k}]_+ [- a_{j-k}]_+) \quad
\end{equation*}
for $1 \leq j < i \leq n$ produces a matrix $B$ with 
	the requisite properties.
\end{example}

\section{$Y$-patterns}
\label{sec:y-patterns}

We keep the notational conventions used in
Section~\ref{sec:laurent-geometric}. 

One of our goals is to show that many structural properties of a seed pattern
are determined by (the mutation class formed by) 
its $n\times n$ exchange matrices~$B(t)$,
and do not depend on the bottom $m-n$ rows of the matrices~$\tilde B(t)$.

Theorem~\ref{th:y-hat} below concerns certain Laurent monomials in the
elements of a given seed;
each of these Laurent monomials is simply a ratio of the two terms
appearing on the
right-hand side of an exchange relation~\eqref{eq:exch-rel-geom}.
Surprisingly, the evolution of these ratios is
completely controlled by the matrices~$B(t)$. 
That is, the laws governing this evolution do not depend on the bottom
parts of the matrices~$\tilde B(t)$.

\begin{theorem}
\label{th:y-hat}
Let $(\tilde\xx, \tilde B)$ and $(\tilde\xx', \tilde B')$
be two labeled seeds related by mutation at~$k$, 
with extended clusters 
\begin{equation*}
\tilde\xx = (x_1, \dots, x_m),\quad
\tilde\xx' = (x'_1, \dots, x'_m)
\end{equation*}
and $m\times n$ extended exchange matrices 
\begin{equation*}
\tilde B= (b_{ij}),\quad 
\tilde B'= (b'_{ij}). 
\end{equation*}
Define the $n$-tuples $\hat\yy=(\hat y_1, \dots, \hat y_n)$
and $\hat\yy'=(\hat y'_1, \dots, \hat y'_n)$~by
\begin{equation}
\label{eq:yhat-yhat'}
\hat y_j = 
\prod_{i=1}^m x_i^{b_{ij}} ,\quad
\hat y'_j = 
\prod_{i=1}^m (x'_i)^{b'_{ij}}.
\end{equation}
Then 
\vspace{-.1in}
\begin{equation}
\label{eq:y-hat-mutation}
\hat y'_j =
\begin{cases}
\hat y_k^{-1}  & \text{if $j = k$};\\[.05in]
\hat y_j \, 
(\hat y_k + 1)^{- b_{kj}} & \text{if $j \neq k$ and $b_{kj}\le
  0$;}\\[.05in]
\hat y_j \, 
(\hat y_k^{-1} + 1)^{- b_{kj}} & \text{if $j \neq k$ and $b_{kj}\ge
  0$.}
\end{cases}
\end{equation}
\end{theorem}

\begin{proof}
We check~\eqref{eq:y-hat-mutation} case by case. 
The case $j=k$ is easy: 
\[
\hat y'_k 
= \prod_{i} (x'_i)^{b'_{ik}}
= \prod_{i \neq k} x_i^{b'_{ik}}
= \prod_{i \neq k} x_i^{-b_{ik}} =
\hat y_k^{-1} .
\]
If $j \neq k$ and  $b_{kj}\le 0$, then 
\begin{align*}
\hat y'_j =& (x'_k)^{b'_{kj}}\prod_{i \neq k} x_i^{b'_{ij}} 
= (x'_k)^{-b_{kj}}\prod_{i \neq k} x_i^{b_{ij}} \prod_{b_{ik}<0} x_i^{-b_{ik}b_{kj}} 
\\
=&x_k^{b_{kj}}\left(\prod_{b_{ik}>0} x_i^{b_{ik}} 
+ \prod_{b_{ik}<0} x_i^{-b_{ik}} \right)^{-b_{kj}}
\prod_{i \neq k} x_i^{b_{ij}} \prod_{b_{ik}<0} x_i^{-b_{ik}b_{kj}} 
\\
=& \hat y_j (\hat y_k + 1)^{-b_{kj}} \,.
\end{align*}
If $j \neq k$ and $b_{kj}\ge 0$, then 
we can check~\eqref{eq:y-hat-mutation} directly as before; alternatively,
it follows from the previous case, using the fact that mutation is an 
involution and switching the roles of the two seeds.
\end{proof}

Theorem~\ref{th:y-hat} suggests the following definitions. 

\begin{definition}
\label{def:y-seed}
A \emph{$Y$-seed} of rank~$n$ in a field~$\FFcal$
is a pair $(Y,B)$ where
\begin{itemize}[leftmargin=.15in]
\item
$Y
$ 
is an $n$-tuple of elements of~$\FFcal$; 
\item
$B$ is a skew-symmetrizable $n\times n$ integer matrix. 
\end{itemize}
We say that two $Y$-seeds $(Y,B)$ and~$(Y',B')$ of rank~$n$
are related by 
a \emph{$Y$-seed mutation} $\mu_k$ in direction~$k$ (here $1\le k\le
n$) if
\begin{itemize}[leftmargin=.15in]
\item
the matrices $B=(b_{ij})$ and $B'=(b'_{ij})$ 
are related via mutation at~$k$; 
\item
the $n$-tuple $Y'\!=\!(Y'_1,\dots,Y'_n)$ 
is obtained from $Y\!=\!(Y_1,\dots,Y_n)$ by 
\vspace{-5pt}
\begin{equation}
\label{eq:y-mutation}
Y_j' =
\begin{cases}
Y_k^{-1}  & \text{if $j = k$;}\\[.05in]
Y_j \, (Y_k + 1)^{- b_{kj}} & \text{if $j \neq k$ and $b_{kj}\le
  0$;}\\[.05in]
Y_j \, (Y_k^{-1} + 1)^{- b_{kj}} & \text{if $j \neq k$ and $b_{kj}\ge
  0$.}
\end{cases}
\end{equation}
\end{itemize}
\vspace{-5pt}
It is easy to check that mutating $(Y',B')$ at~$k$ recovers~$(Y,B)$. 

A \emph{$Y$-pattern} of rank~$n$ 
is a collection of $Y$-seeds $(Y(t),B(t))_{t\in\TT_n}$ labeled by 
the vertices of the $n$-regular tree~$\TT_n$, 
such that for any edge $t\overunder{k}{}t'$ in~$\TT_n$, 
the $Y$-seeds $(Y(t),B(t))$ and $(Y(t'),B(t'))$ are related to each
other by the $Y$-seed mutation in direction~$k$. 
\end{definition}

\begin{remark}
In Definition~\ref{def:y-seed},
we do not require the elements~$Y_i$ to be
algebraically independent,
one reason being that this condition does not always hold 
for the monomials $\hat y_j$ in Theorem~\ref{th:y-hat}. 
Consequently, one can not \emph{a~priori} guarantee that the mutation process can
propagate to all vertices in~$\TT_n$ (what if $Y_k=0$
in~\eqref{eq:y-mutation}?). 
To ensure the existence of a $Y$-pattern with a given initial seed $(Y,B)$, 
one can for example
require all elements of~$Y$ to be given by
subtraction-free expressions in some set of variables,
or alternatively take positive values under a particular
specialization of these variables. 
As each of these conditions reproduces under mutations of $Y$-seeds,
the mutation process can then proceed without hindrance. 
\end{remark}

\begin{example}
[\emph{$Y$-pattern of type~$A_2$
}]
\label{example:A2-Y-pattern}
Consider the $Y$-pattern of rank~$2$ 
\vspace{-5pt}
\[
\cdots
\overunder{2}{} (Y(0),B(0))
\overunder{1}{} (Y(1),B(1))
\overunder{2}{} (Y(2),B(2))
\overunder{1}{} \cdots \,
\]
with the exchange matrices 
\vspace{-5pt}
\begin{equation}
\label{eq:B(t)-A2}
B(t)=(-1)^t \left[
\begin{matrix}
0 & 1\\
-1&0
\end{matrix}
\right]. 
\end{equation}
(The corresponding quivers are orientations of the type 
$A_2$ Dynkin diagram.) 

\pagebreak[3]

The rule~\eqref{eq:y-mutation} of $Y$-seed mutation gives the
following recurrence for the $Y$-seeds $Y(t)=(Y_{1;t},Y_{2;t})$. 
For $t$ even, we have 
\[
Y(t+1)=\mu_1(Y(t)),\quad   
Y_{1;t+1} = Y_{1;t}^{-1}, \quad
Y_{2;t+1} = Y_{2;t} \,(Y_{1;t}^{-1}+1)^{-1}, 
\]
whereas for $t$ odd, we have
\[
Y(t+1)=\mu_2(Y(t)),\quad 
Y_{1;t+1} = Y_{1;t} \,(Y_{2;t}^{-1}+1)^{-1},\quad
Y_{2;t+1} = Y_{2;t}^{-1}. 
\]
We then recursively obtain the pairs
$Y(t)=(Y_{1;t},Y_{2;t})$ 
listed in Figure~\ref{fig:y-seeds-A2}.
\end{example}
\vspace{-.1in}

\begin{figure}[ht]
{
\begin{equation*}
\begin{array}{|c|c|c|}
\hline
&&\\[-4mm]
t & Y_{1;t} & Y_{2;t} \\[1mm]
\hline
&&\\[-4mm]
0 & y_1 & y_2 \\[1mm]
\hline
&&\\[-3.5mm]
1 & y_1^{-1} & y_1 \, y_2 (y_1+1)^{-1} \\[1mm]
\hline
&&\\[-3.5mm]
2 & y_2 (y_1y_2+y_1+1)^{-1} & (y_1+1)\,  y_1^{-1} y_2^{-1} \\[1mm]
\hline
&&\\[-3.5mm]
3 & (y_1y_2+y_1+1)\,  y_2^{-1} & y_1^{-1}(y_2+1)^{-1} \\[1mm]
\hline
&&\\[-3.5mm]
4 & y_2^{-1} & y_1(y_2+1) \\[1mm]
\hline
&&\\[-4mm]
5 & y_2 & y_1 \\[1mm]
\hline
&&\\[-3.5mm]
6 & y_1 y_2 (y_1+1)^{-1} & y_1^{-1} \\[1mm]
\hline
&&\\[-3.5mm]
7 & (y_1+1)\,  y_1^{-1} y_2^{-1} & y_2 (y_1y_2+y_1+1)^{-1} \\[1mm]
\hline
&&\\[-3.5mm]
8 & y_1^{-1}(y_2+1)^{-1} & (y_1y_2+y_1+1)\,  y_2^{-1} \\[1mm]
\hline
&&\\[-3.5mm]
9 & y_1(y_2+1) & y_2^{-1} \\[1mm]
\hline
&&\\[-4mm]
10 & y_1 & y_2  \\[1mm]
\hline
\end{array}
\end{equation*}
}
\vspace{-.17in}
\caption{The $Y$-seeds $(Y(t),B(t))=((Y_{1;t},Y_{2;t}),B(t))$
in type~$A_2$. 
The exchange matrices $B(t)$ are given by~\eqref{eq:B(t)-A2}. 
The initial $Y$-seed is $(Y(0),B(0))$, with 
$Y(0)=(Y_{1;0},Y_{2;0})=(y_1,y_2)$. 
This sequence of $Y$-seeds is $10$-periodic: $Y(t+10)=Y(t)$.}
\label{fig:y-seeds-A2}
\end{figure}

Using the terminology introduced in Definition~\ref{def:y-seed}, 
we can state the following direct corollary of 
Theorem~\ref{th:y-hat}. 

\begin{corollary}
\label{cor:y-hat}
Let $(\tilde\xx(t), \tilde B(t))_{t\in\TT_n}$ be a seed pattern
in~$\FFcal$, 
with
\begin{equation*}
\tilde\xx(t) = (x_{1;t}, \dots, x_{m;t}),\quad
\tilde B(t) = (b^t_{ij}). 
\end{equation*}
Let $B(t)= (b^t_{ij})_{i,j\le n}$ denote the exchange matrix at a vertex~$t \in \TT_n$, 
and let $\hat \yy(t) = (\hat y_{1;t}, \dots, \hat y_{n;t})$ be the $n$-tuple of
elements in $\FFcal$ given~by 
\begin{equation}
\label{eq:y-hat-pattern}
\hat y_{k;t} = 
\prod_{i=1}^m x_{i;t}^{b^t_{ik}} \,.
\end{equation}
Then $(\hat\yy(t),B(t))_{t\in\TT_n}$ is a $Y$-pattern in~$\FFcal$. 
\end{corollary}

\pagebreak[3]

\begin{remark}
The rules governing the evolution of $Y$-seeds 
may seem simpler than the corresponding rules of seed mutation:
\begin{itemize}[leftmargin=.15in]
\item
$Y$-seed mutations are driven by the $n\times n$
matrices~$B$ whereas \linebreak[3]
ordinary seed mutations require the extended 
$m\times n$ matrices~$\tilde B$;
\item
in the $Y$-seed setting, there are no frozen variables; 
\item
each recurrence~\eqref{eq:y-mutation} only involves two
variables $Y_j$ and~$Y_k$
whereas the exchange relation~\eqref{eq:exch-rel-geom}
potentially involves all cluster variables of the current seed. 
\end{itemize}
On the other hand,
\begin{itemize}[leftmargin=.15in]
\item
a seed mutation only changes one cluster variable whereas a $Y$-seed
mutation may potentially change all the variables $Y_1,\dots,Y_n$;
\item
consequently, we end up getting ``more'' $Y$-variables than cluster variables 
(if the number of seeds is finite, then this is a precise statement); 
\item
the $Y$-pattern recurrences do not, generally speaking, 
exhibit the Laurent phenomenon. 
\end{itemize}
\end{remark}

\begin{remark}
In various examples, including many cluster algebras arising as 
	coordinate rings of algebraic varieties, the cluster
	algebra under investigation 
 has a distinguished (multi-)grading,
and its exchange relations are all (multi-)homo\-geneous. 
It follows that the rational expressions $\hat y_{k;t}$ 
defined by~\eqref{eq:y-hat-pattern} have (multi-)degree~$0$. 
It~is not surprising, then, that $Y$-patterns naturally arise in the
study of configurations (of points, lines, flags, etc.) in projective
spaces. 
See Examples~\ref{example:configs-of-points-on-proj-line}
and~\ref{example:pentagram} below. 
\end{remark}

\pagebreak[3]

\begin{example}[\emph{Configurations of points on the projective
      line}] 
\label{example:configs-of-points-on-proj-line}
The \emph{cross-ratio} is a quantity associated with an 
ordered quadruple of collinear points, particularly points on
the projective line~$\mathbb{P}^1$ (say over~$\CC$). 
For our purposes, it will be convenient to use the following version
of the cross-ratio. 
Let $P_1, P_2, P_3, P_4\in\mathbb{P}^1$
be four distinct points on the projective line, 
with projective coordinates
$(a_1:b_1)$, $(a_2:b_2)$, $(a_3:b_3)$, and $(a_4:b_4)$, 
respectively.
We then define
\begin{equation}
\label{eq:Y1234}
Y(P_1,P_2, P_3, P_4)=\frac{P_{14}\,P_{23}}{P_{12}\,P_{34}},
\end{equation}
where we use the notation 
\begin{equation}
\label{eq:Pij}
P_{ij} = \det\begin{pmatrix}
a_i & a_j \\
b_i & b_j 
\end{pmatrix}
=a_i b_j - a_j b_i\,.
\end{equation}
This quantity is related to the conventional cross-ratio via the formula 
\[
Y(P_1,P_2, P_3, P_4)=-(P_1,P_3;P_4,P_2).
\]

The symmetric group~$\mathcal{S}_4$ acts on quadruples of collinear points 
by permuting the points in a quadruple. 
The permutations in~$\mathcal{S}_4$ which preserve the cross-ratio 
form a subgroup isomorphic to the Klein four-group. 
There are therefore six different versions of the cross-ratio,
all of which are uniquely determined by any one of them. 

The cross-ratio is 
essentially the only projective invariant of a quadruple of collinear points.
More generally (see, e.g., \cite[Section~7.4]{richter-gebert}), 
any rational function of an ordered $m$-tuple of points on the projective line 
which is invariant under projective transformations 
can be expressed in terms of cross-ratios associated to various
quadruples of points.
In~fact, one only needs cross-ratios associated with $m-3$ quadruples to get 
all $\binom{m}{4}$ of them. 
One way to make this explicit is by using the machinery of
$Y$-patterns. 
A~configuration of $m$ distinct ordered points
$P_1,\dots, P_m\in\mathbb{P}^1$ with projective coordinates
$(a_1:b_1),\dots,(a_m:b_m)$ 
can be encoded by a $2 \times m$ matrix
\begin{equation*}
z = \begin{bmatrix}
a_1 & a_2 & \dots & a_m\\
b_1 & b_2 & \dots & b_m
\end{bmatrix}. 
\end{equation*}
Recall that the Pl\"ucker coordinates $P_{ij}=P_{ij}(z)$ are defined
by the formula~\eqref{eq:Pij},
for $1\le i<j\le m$. 

We now associate a $Y$-seed 
to an arbitrary triangulation~$T$ of 
a convex $m$-gon $\mathbf{P}_m$
(cf.\ Sections~\ref{sec:Ptolemy}
and~\ref{sec:triangulations})
by $m-3$ pairwise noncrossing diagonals. 
Recall that $\mathbf{P}_m$ has $m$ vertices labeled $1,\dots,m$, 
in clockwise order.  
We label the diagonals of $T$ by the numbers $1, \dots, m-3$, 
and define the exchange matrix $B_T$ to be the $(m-3) \times (m-3)$
matrix associated to the mutable part of the quiver $Q(T)$, see 
Definition~\ref{def:Q(T)-polygon}.  (Ignore the frozen 
vertices associated with the sides of the polygon.)
Consider a diagonal of~$T$ labeled~$d$. 
This diagonal triangulates a quadrilateral with vertices
labeled $i, j, k, \ell$ in clockwise order, connecting vertices
$i$ and~$k$, cf.\ Figure~\ref{fig:type-a-exch}.  
Define 
\[
Y_d = Y(P_i,P_j,P_k,P_\ell),
\] 
cf.~\eqref{eq:Y1234}.
Note that since 
$Y(P_i,P_j,P_k,P_\ell)=Y(P_k,P_\ell,P_i,P_j)$,
there is no ambiguity in this definition.
Finally, define the $Y$-seed 
associated with $T$ to be 
the pair $(Y_T, B_T)$, where 
$Y_T = (Y_1,\dots, Y_{m-3})$.

It is now an exercise to verify that 
these $Y$-seeds transform under flips 
by the $Y$-seed mutation rule 
\eqref{eq:y-mutation}.  
Note that this example is nothing but the application of the
construction in Theorem~\ref{th:y-hat} to 
the seed pattern associated to the $\Gr_{2,m}$ example
(cf.\ Section~\ref{sec:triangulations}).
\end{example}

\begin{exercise}
Given six points $P_1,\dots,P_6$ on the projective line,
express the cross-ratios for the quadruples 
$\{P_i,P_4,P_5,P_6\}$ in terms of the cross-ratios for the quadruples
$\{P_i,P_1,P_2,P_3\}$. 
\end{exercise}

\pagebreak[3]

\begin{example}[\emph{The pentagram map}] 
\label{example:pentagram}
The \emph{pentagram map}, introduced in~\cite{Schwartz}, 
is a transformation of generic projective polygons 
(i.e., cyclically ordered tuples of
points on the projective plane~$\mathbb{P}^2$) 
defined by the following construction: given a polygon~$A$
as input, draw all of its ``shortest" diagonals, and output
the ``smaller'' polygon~$A'$ which they cut out.
See Figure~\ref{fig:pentagram}.

\begin{figure}[ht]
\begin{center}
\hspace{-.2in}
\setlength{\unitlength}{2.2pt}
\begin{picture}(60,66)(0,-3)

\linethickness{2pt}
  \multiput(0,20)(60,0){2}{\line(0,1){20}}
  \multiput(20,0)(0,60){2}{\line(1,0){20}}
  \multiput(0,40)(40,-40){2}{\line(1,1){20}}
  \multiput(20,0)(40,40){2}{\line(-1,1){20}}

  \multiput(20,0)(20,0){2}{\circle*{1}}
  \multiput(20,60)(20,0){2}{\circle*{1}}
  \multiput(0,20)(0,20){2}{\circle*{1}}
  \multiput(60,20)(0,20){2}{\circle*{1}}

\thicklines
  \put(70, 30){{\vector(1,0){8}}}
\end{picture}
\hspace{2cm}
\begin{picture}(60,66)(0,-3)
\put(-4,20){\makebox(0,0){$A_7$}}
\put(20,-4){\makebox(0,0){$A_0\!=\!A_8$}}
\put(40,-4){\makebox(0,0){$A_1$}}
\put(64,20){\makebox(0,0){$A_2$}}
\put(64,40){\makebox(0,0){$A_3$}}
\put(40,64){\makebox(0,0){$A_4$}}
\put(20,64){\makebox(0,0){$A_5$}}
\put(-4,40){\makebox(0,0){$A_6$}}

\put(29,11){\makebox(0,0){$A'_{\frac12}$}}
\put(42,17){\makebox(0,0){$A'_{1\!\frac12}$}}
\put(49,30){\makebox(0,0){$A'_{2\!\frac12}$}}
\put(43,42){\makebox(0,0){$A'_{3\!\frac12}$}}
\put(30,49){\makebox(0,0){$A'_{4\!\frac12}$}}
\put(18,42){\makebox(0,0){$A'_{5\!\frac12}$}}
\put(11,29){\makebox(0,0){$A'_{6\!\frac12}$}}
\put(17,16){\makebox(0,0){$A'_{7\!\frac12}$}}

\thinlines
  \multiput(0,20)(60,0){2}{\line(0,1){20}}
  \multiput(20,0)(0,60){2}{\line(1,0){20}}
  \multiput(0,40)(40,-40){2}{\line(1,1){20}}
  \multiput(20,0)(40,40){2}{\line(-1,1){20}}

  \multiput(20,0)(20,0){2}{\circle*{1}}
  \multiput(20,60)(20,0){2}{\circle*{1}}
  \multiput(0,20)(0,20){2}{\circle*{1}}
  \multiput(60,20)(0,20){2}{\circle*{1}}

\put(40,0){\line(1,2){20}}
\put(0,20){\line(1,2){20}}
\put(20,60){\line(2,-1){40}}
\put(40,60){\line(1,-2){20}}
\put(20,0){\line(2,1){40}}
\put(0,20){\line(2,-1){40}}
\put(0,40){\line(1,-2){20}}
\put(40,60){\line(-2,-1){40}}

\linethickness{2pt}
\put(13.3,13.3){\line(2,-1){16.6}}
\put(13.3,13.3){\line(-1,2){8.3}}
\put(13.3,46.3){\line(-1,-2){8.3}}
\put(13.3,46.3){\line(2,1){16.6}}
\put(46.6,13.3){\line(-2,-1){16.7}}
\put(46.6,13.3){\line(1,2){8.3}}
\put(46.6,46.3){\line(1,-2){8.3}}
\put(46.6,46.3){\line(-2,1){16.6}}
\end{picture}

\end{center}
\caption{The pentagram map.}
\label{fig:pentagram}
\end{figure}

As shown in 
\cite{Glick}, the pentagram map is related to $Y$-seed mutation.
To explain the connection, one needs to describe the pentagram map in
properly chosen coordinates. 
We shall view a polygon with $n$ vertices as an $n$-periodic sequence 
$A = (A_i)_{i\in \ZZ}$ of points in~$\mathbb{P}^2$. 
Given two polygons related 
by the pentagram map,
it is convenient to index the points of one of them by the integers $\ZZ$ 
and the points of the other by the half-integers
$\ZZ+\frac{1}{2}$, as shown at the right in Figure \ref{fig:pentagram}.

Recall the definition~\eqref{eq:Y1234} of 
the projective invariant $Y(P_1,P_2, P_3, P_4)$
(a negative cross-ratio)
associated with a quadruple of distinct collinear points $P_1,P_2, P_3, P_4$. 
One can associate 
a similar invariant to a quadruple of distinct concurrent lines 
$L_1,L_2,L_3,L_4$ in~$\mathbb{P}^2$ 
passing through a point~$Q$: 
any line~$L$ not passing
through~$Q$ intersects these lines in four distinct points
$P_1,P_2,P_3,P_4$, and the number $Y(L_1,L_2,L_3,L_4):=Y(P_1,P_2, P_3, P_4)$
does not depend on the choice of the line~$L$.

\begin{definition} 
\label{def:y-params}
Let $A$ be a polygon with $n$ vertices indexed as above by 
	either  $\ZZ$ or $\ZZ+\frac{1}{2}$.
The \emph{$y$-parameters of $A$} 
	are the numbers $y_j(A)$ (for $1 \leq j \leq 2n$) defined by
\begin{align}
\label{eq:y-pentagram}
y_{2k}(A) &= 
Y(\overleftrightarrow{A_k A_{k-1}},
               \overleftrightarrow{A_k A_{k+2}},
               \overleftrightarrow{A_k A_{k+1}},
               \overleftrightarrow{A_k A_{k-2}})^{-1},
\\
\nonumber
y_{2k+1}(A) &= Y(A_k,
                 \overleftrightarrow{A_{k+2}A_{k+3}} \cap L, 
                 A_{k+1}, 
                 \overleftrightarrow{A_{k-2}A_{k-1}} \cap L),
\end{align}
where $L \!=\! \overleftrightarrow{A_k A_{k+1}}$.
(Here $\overleftrightarrow{A_iA_j}$ denotes the line passing through
$A_i$ and~$A_j$.)
See Figure~\ref{fig:y-parameters}. 

\begin{figure}[ht]
\begin{center}

\setlength{\unitlength}{2.4pt}

\begin{picture}(60,66)(2,-3)
\put(20,-4){\makebox(0,0){$A_{k-2}$}}
\put(40,-4){\makebox(0,0){$A_{k-1}$}}
\put(64,20){\makebox(0,0){$A_k$}}
\put(67,40){\makebox(0,0){$A_{k+1}$}}
\put(40,64){\makebox(0,0){$A_{k+2}$}}
\put(20,64){\makebox(0,0){$A_{k+3}$}}

\thinlines
  \multiput(0,20)(60,0){2}{\line(0,1){20}}
  \multiput(20,0)(0,60){2}{\line(1,0){20}}
  \multiput(0,40)(40,-40){2}{\line(1,1){20}}
  \multiput(20,0)(40,40){2}{\line(-1,1){20}}

  \multiput(20,0)(20,0){2}{\circle*{1}}
  \multiput(20,60)(20,0){2}{\circle*{1}}
  \multiput(0,20)(0,20){2}{\circle*{1}}
  \multiput(60,20)(0,20){2}{\circle*{1}}

\linethickness{1.5pt}
\put(40,0){\line(1,1){20}}
\put(60,20){\line(0,1){20}}
 \put(40,60){\line(1,-2){20}}
 \put(20,0){\line(2,1){40}}
\end{picture}
\hspace{.6in}
\begin{picture}(60,66)(5,-3)
\put(20,-4){\makebox(0,0){$A_{k-2}$}}
\put(40,-4){\makebox(0,0){$A_{k-1}$}}
\put(64,20){\makebox(0,0){$A_k$}}
\put(67,40){\makebox(0,0){$A_{k+1}$}}
\put(40,64){\makebox(0,0){$A_{k+2}$}}
\put(20,64){\makebox(0,0){$A_{k+3}$}}

\put(63,8){\makebox(0,0){$L$}}

\thinlines
  \multiput(0,20)(60,0){2}{\line(0,1){20}}
  \multiput(20,0)(0,60){2}{\line(1,0){20}}
  \multiput(0,40)(40,-40){2}{\line(1,1){20}}
  \multiput(20,0)(40,40){2}{\line(-1,1){20}}

  \multiput(20,0)(20,0){2}{\circle*{1}}
  \multiput(20,60)(20,0){2}{\circle*{1}}
  \multiput(0,20)(0,20){2}{\circle*{1}}
  \multiput(60,20)(0,20){2}{\circle*{1}}

\put(63,0){\line(-1,0){6}}
\multiput(55,0)(-6,0){3}{\line(-1,0){4}}
\put(63,60){\line(-1,0){6}}
\multiput(55,60)(-6,0){3}{\line(-1,0){4}}

\linethickness{0.8pt}

\put(60,-3){\line(0,1){66}}
\multiput(60,0)(0,20){4}{\circle*{2.5}}

\end{picture}

\end{center}
\caption{The $y$-parameters of a polygon.}
\label{fig:y-parameters}
\end{figure}

We next define the $2n \times 2n$
exchange matrix $B  = (b_{ij})$ by 
\begin{equation}
\label{eq:B-pentagram}
b_{ij} = 
\begin{cases}
(-1)^j & \text{if $i-j \equiv \pm 1 \mod 2n$;}\\
(-1)^{j+1} & \text{if $i-j \equiv \pm 3 \mod 2n$;}\\
0 & \text{otherwise.}
\end{cases}
\end{equation}
We set $Y(A)=(y_1(A),\dots,y_{2n}(A))$. 
Thus $(Y(A), B)$ is a $Y$-seed of rank~$2n$.
\end{definition}

The following result, obtained in~\cite{Glick}, is included without
proof. 

\begin{proposition} 
Let $A$ be an $n$-gon indexed by~$\ZZ$, and let $A'$ 
be the $n$-gon (indexed by $\ZZ+\frac12$)
obtained from $A$ via the pentagram map. 
Then applying the composition of $Y$-seed mutations
\[
\mu_{even} = \mu_2 \circ \mu_4 \circ \dots \circ \mu_{2n} 
\]
to the $Y$-seed $(Y(A),B)$ 
(cf.\ \eqref{eq:y-pentagram}--\eqref{eq:B-pentagram})
produces the $Y$-seed $(Y(A'),-B)$. 

Similarly, let $A'$ be an $n$-gon indexed by $\ZZ+\frac12$. 
Then applying
\[
\mu_{odd} = \mu_1 \circ \mu_3 \circ \dots \circ \mu_{2n-1} 
\]
to the $Y$-seed $(Y(A'),-B)$ produces the 
$Y$-seed $(Y(A''),B)$ associated with the $n$-gon~$A''$ 
obtained from $A'$ via the pentagram map. 
\end{proposition}
(Note that the individual mutations in each of $\mu_{even}$
and $\mu_{odd}$ commute.)  
\end{example}

\pagebreak[3]

$Y$-patterns have arisen in many other mathematical contexts.  
An incomplete list includes: 
\begin{itemize}[leftmargin=.15in]
\item
Thurston's \emph{shear coordinates} in Teichm\"uller spaces 
and their generalizations (see, e.g.,
\cite{fock-goncharov-dual} and references therein); 
\item
recursively defined sequences of points on \emph{elliptic curves},
and associated Somos-like sequences, 
cf.\ Example~\ref{ex:Somos4} and Remark~\ref{rem:elliptic};
\item
\emph{wall-crossing formulas} for motivic Donaldson-Thomas invariants 
introduced by M.~Kontsevich and Y.~Soibelman (see,
e.g.,~\cite{kontsevich-soibelman-survey}), and related wall-crossing
phenomena for BPS states in theoretical physics; 
\item
Fock-Goncharov varieties~\cite{fock-goncharov-ihes}, including moduli
spaces of point configurations in basic affine spaces; 
\item
Zamolodchikov's $Y$-systems \cite{yga, zamolodchikov} in the theory of the Thermodynamic Bethe Ansatz. 
\end{itemize}
We will return to some of the aforementioned applications in the
subsequent chapters. 

\pagebreak[3]

\section{Tropical semifields}

In this section, we re-examine the combinatorics of matrix mutations,
relating it to the concept of $Y$-seeds and their mutations
discussed in Section~\ref{sec:y-patterns}. 
We begin by introducing the notion of semifield,
and in particular, the tropical semifield, which 
will give us an important alternative way to encode the
bottom part of an extended exchange matrix $\tilde B$.

\begin{definition}
A \emph{semifield} is an abelian group~$P$, written
multiplicatively, endowed with an operation of ``auxiliary
addition''~$\oplus$
which is required to be commutative and associative, and satisfy the distributive law with
respect to the multiplication in~$P$. 
\end{definition}

We emphasize that $(P,\oplus)$ does not have to be a group, just a semigroup. 
Since every element of~$P$ has a multiplicative inverse,
  $P$ does not contain an additive identity element (unless $P$ is trivial).

\begin{definition}
Let $\Trop(q_1,\dots, q_{\ell})$
denote the multiplicative group of Laurent
monomials 
in the variables $q_1,\dots,q_{\ell}$. 
We equip $\Trop(q_1,\dots,q_{\ell})$ with the binary operation of
\emph{tropical addition}~$\oplus$ defined by
\begin{equation}
\label{eq:tropical-addition}
\prod_{i=1}^{\ell} q_i^{a_i} \oplus \prod_{i=1}^{\ell} q_i^{b_i} =
\prod_{i=1}^{\ell} q_i^{\min (a_i, b_i)} \ .
\end{equation}
\end{definition}

\begin{lemma}\label{lem:distributive}
Tropical addition 
is commutative and associative, and it satisfies the distributive law with
respect to the ordinary multiplication:
\begin{equation*}
(p\oplus q)r=pr\oplus qr. 
\end{equation*}
\end{lemma}

Thus $\Trop(q_1,\dots,q_\ell)$ is a semifield,
which we call 
the \emph{tropical semifield} generated by
$q_{1},\dots,q_{\ell}$. 

\begin{remark}
The above terminology differs from the one used in \emph{tropical
  geometry} by what is essentially a notational convention:
replacing Laurent monomials by the corresponding vectors of exponents,
one gets a semifield in which multiplication is the ordinary addition,
and auxiliary addition amounts to taking the minimum.
\end{remark}

The formalism of the tropical semifield and its auxiliary addition
allows us to 
restate the rules of matrix mutation in the following way.

Let~$\tilde B$ be an $m\times n$ extended exchange matrix.
As before, $x_{n+1},\dots,x_m$ are the frozen variables. 
We encode the bottom $(m-n)\times n$ submatrix of~$\tilde B$ 
by the \emph{coefficient tuple} $\yy = (y_1, \dots, y_n)\in\Trop(x_{n+1},\dots,x_m)^n$ defined by
\begin{equation}
\label{eq:yj}
y_j = \prod_{i=n+1}^m x_i^{b_{ij}} \quad (j \in \{1,\dots,n\}) \ .
\end{equation}
Thus the matrix $\tilde B$ contains the same information as its top $n\times n $
submatrix~$B$ together with the coefficient tuple~$\yy$.

\begin{proposition}
\label{pr:matrix-mutation-trop}
Let $\tilde B=(b_{ij})$ and $\tilde B'$
be two extended skew-symmetri\-zable matrices related by a
mutation~$\mu_k$, 
and let $\yy=(y_1,\dots,y_n)$ and  $\yy'=(y'_1,\dots,y'_n)$ 
be the corresponding coefficient tuples (cf.~\eqref{eq:yj}). 
Then 
\begin{equation}
\label{eq:y-mutation-trop}
y'_j =
\begin{cases}
y_k^{-1} & \text{if $j = k$;}\\[.05in]
y_j (y_k \oplus 1)^{- b_{kj}} & \text{if $j \neq k$ and $b_{kj}\le
  0$;}\\[.05in]
y_j (y_k^{-1} \oplus 1)^{- b_{kj}} & \text{if $j \neq k$ and $b_{kj}\ge
  0$.}
\end{cases}
\end{equation}
\end{proposition}

Comparing \eqref{eq:y-mutation-trop} with~\eqref{eq:y-mutation},
we can informally say that the coefficient tuple~$\yy$ undergoes a
``tropical $Y$-seed mutation'' at~$k$. 

Proposition~\ref{pr:matrix-mutation-trop} can be proved by
translating the rules of matrix mutation into the
language of the tropical semifield.
We outline a different 
proof which explains the connection 
between the formulas
\eqref{eq:y-hat-mutation}--\eqref{eq:y-mutation} and~\eqref{eq:y-mutation-trop},
and introduces some notions that will be useful in the sequel. 

\begin{definition}
Let $\Qsf(x_1,\dots,x_m)$ denote the set of nonzero rational functions in 
$x_1,\dots,x_m$ which can be written as subtraction-free rational
expressions in these variables, with positive rational coefficients. 
Thus, each element of $\Qsf(x_1,\dots,x_m)$ 
can be written in the form 
$\frac{P(x_1,\dots,x_m)}{Q(x_1,\dots,x_m)}$, where $P$ and $Q$ are 
polynomials with positive coefficients.
The set $\Qsf(x_1,\dots,x_m)$ is a semifield with respect to the
ordinary operations of addition and multiplication. 
We call it the \emph{universal semifield} generated by
$x_1,\dots,x_m$. 
\end{definition}

This terminology is justified by the following easy lemma,
whose proof we omit (see~\cite[Lemma~2.1.6]{bfz-tp}).
Informally speaking, this lemma says that the generators 
$x_1,\dots,x_m$ of the semifield $\Qsf(x_1,\dots,x_m)$ do not 
satisfy any relations, save for those which are implied by the axioms 
of a semifield.

\begin{lemma} 
\label{lem:universal-semifield}
For any semifield~$\mathcal{S}$, any map 
$f:\{x_1,\dots,x_m\}\to\mathcal{S}$ extends uniquely to a semifield
homomorphism $\Qsf(x_1,\dots,x_m)\to\mathcal{S}$. 
\end{lemma}

\begin{proof}[Proof of Proposition~\ref{pr:matrix-mutation-trop}]
Let $\tilde\xx=(x_1,\dots,x_m)$ be a collection of indeterminates. 
Define the semifield homomorphism
\[
f: \Qsf(x_1,\dots,x_m)\to\Trop(x_{n+1},\dots,x_m)
\]
by setting (cf.\ Lemma~\ref{lem:universal-semifield}) 
\[
f(x_i)=
\begin{cases}
1 & \text{if $i\le n$;}\\
x_i & \text{if $i>n$.}
\end{cases}
\]
Applying mutation at~$k$ to the seed $(\tilde\xx,\tilde B)$, 
we get a new seed $(\tilde\xx',\tilde B')$,
in which the only new cluster variable $x_k'$ satisfies an exchange
relation 
of the form $x_k x_k'=M_1+M_2$. 
The two monomials $M_1$ and~$M_2$ are coprime, and in particular do not
share a frozen variable~$x_i$. 
Applying the semifield homomorphism~$f$, we obtain
$1\cdot f(x_k')=f(M_1)\oplus f(M_2)=1$, so $f(x_k')=1$. 

Now let $\hat \yy$ and $\hat \yy'$ 
(resp., $\yy$ and~$\yy'$) 
be defined 
by~\eqref{eq:yhat-yhat'} (resp.,~\eqref{eq:yj}). 
Since all cluster
variables in $\tilde\xx$ and~$\tilde\xx'$ are sent to~1 by~$f$, 
we conclude that $f(\hat \yy)=\yy$ and $f(\hat \yy')=\yy'$. 
Therefore applying $f$ to~\eqref{eq:y-hat-mutation}
yields~\eqref{eq:y-mutation-trop}. 
\end{proof}

Let $(\tilde \xx,\tilde B)$ be a labeled seed as before.
Since the extended exchange matrix~$\tilde B$ contains the same information 
as the exchange matrix $B$ together with the coefficient tuple~$\yy$
defined by~\eqref{eq:yj}, we can identify the seed $(\tilde \xx,\tilde B)$
with the triple $(\xx, \yy, B)$.  Abusing notation, we will 
also refer to such triples as (labeled) seeds: 

\begin{definition}
\label{def:seed-triple}
Let $\FFcal$ be a field of rational functions (say over~$\CC$) 
in some $m$ variables which include the \emph{frozen variables}
$x_{n+1},\dots,x_m$. 
A labeled \emph{seed} (of geometric type)
of rank~$n$ 
is a triple
$\Sigma = (\xx, \yy, B)$ consisting of 
\begin{itemize}[leftmargin=.15in]
\item
a \emph{cluster} $\xx$, an $n$-tuple of elements of~$\FFcal$
such that the \emph{extended cluster}
$\xx\cup\{x_{n+1},\dots,x_m\}$ freely generates~$\FFcal$;
\item
an \emph{exchange matrix}~$B$, a skew-symmetrizable integer matrix;
\item
a \emph{coefficient tuple} $\yy$, an $n$-tuple of Laurent monomials in
the tropical semifield $\Trop(x_{n+1},\dots,x_m)$. 
\end{itemize}
\end{definition}

We can now restate the rules of seed mutation in this language. 

\pagebreak[3]

\begin{proposition}
\label{pr:xyB-mutation}
Let $(\xx, \yy, B)$, with $B=(b_{ij})$ and $\yy=(y_1,\dots,y_n)$, 
and \linebreak[3]
$(\xx', \yy', B')$, with 
$\yy'=(y'_1,\dots,y'_n)$, 
be  
two labeled seeds 
related by a mutation~$\mu_k$. 
Then $(\xx', \yy', B')$ is obtained from $(\xx, \yy, B)$ as follows: 
\begin{itemize}[leftmargin=.15in]
\item
$B'=\mu_k(B)$;
\item
$\yy'$ is given by the ``tropical $Y$-seed mutation rule''~\eqref{eq:y-mutation-trop}; 
\item
$\xx'=(\xx-\{x_k\}) \cup\{x'_k\}$, where $x'_k$ is defined by
the exchange relation 
\begin{equation}
\label{eq:exchange-rel-xx}
x_k \, x_k'= \frac{y_k}{y_k \oplus 1}
\prod_{b_{ik}>0} x_i^{b_{ik}}
+ 
\frac{1}{y_k \oplus 1}
\prod_{b_{ik}<0} x_i^{-b_{ik}} \,; 
\end{equation}
\end{itemize}
\end{proposition}

\begin{proof}
The only statement requiring proof is~\eqref{eq:exchange-rel-xx},
which can be easily seen to be a rewriting of~\eqref{eq:exch-rel-geom}.
\end{proof}

We will use \cref{eq:exchange-rel-xx} in 
Chapter~\ref{ch:general-cluster-algebras} to define cluster 
algebras over an arbitrary semifield.

\pagebreak[3]

We can now re-define the notion of a labeled seed pattern.

\begin{definition} \label{def:triple-pattern}
A \emph{labeled seed pattern} of rank~$n$ is obtained  
by assigning a triple $\Sigma(t) = (\xx(t), \yy(t), B(t))$
as above to every vertex $t$ in the $n$-regular tree~$\TT_n$,
and requiring that the triples assigned to adjacent vertices of the tree are related by 
the corresponding mutation, as described in Proposition~\ref{pr:xyB-mutation}.
\end{definition} 

The advantage of the latest version of the definition of seed pattern
is that it enables us to perform calculations for arbitrary 
extensions of a given exchange matrix to an extended exchange matrix;
see  
\eqref{eq:yj}
and the surrounding discussion.

\begin{example}
[\emph{Type~$A_2$}]
\label{example:A2-seed-pattern}
Consider the seed pattern of rank~$2$ 
\begin{equation*}
\cdots
\overunder{2}{} \Sigma(0)
\overunder{1}{} \Sigma(1)
\overunder{2}{} \Sigma(2)
\overunder{1}{} \Sigma(3)
\overunder{2}{} \Sigma(4)
\overunder{1}{} \Sigma(5)
\overunder{2}{} \cdots \,
\end{equation*}
formed by the seeds 
$\Sigma(t)=(\xx(t),\yy(t),B(t))$, for $t\in\TT_2\cong\ZZ$,
with the exchange matrices
\begin{equation}
\label{eq:exchange-matrices-A2}
B(t)=(-1)^t\,\begin{bmatrix}
 0 & 1\\
-1 & 0
\end{bmatrix}
\end{equation}
(cf.\ Examples~\ref{eg:A2-6seeds} and~\ref{example:A2-Y-pattern}).
Note that we do not specify the bottom part of the initial exchange
matrix, nor even the number of frozen variables. 
Still, we can express all the seeds in terms of the initial one using
the language of the tropical semifield,
and following the recipe formulated in Proposition~\ref{pr:xyB-mutation}. 
The~results of the 
computation 
are shown in Figure~\ref{fig:seeds-A2}.
\end{example}

\begin{figure}[ht]
{\small
\begin{equation*}
\begin{array}{|c|cc|cc|}
\hline
&&&&\\[-4mm]
t & \multicolumn{2}{c}{\yy(t)} & \multicolumn{2}{c}{\xx(t)}  \\[0mm]
\hline
&&&&\\[-3mm]
0 & 
y_1&y_2&x_1& x_2
\\[.5mm]
\hline
&&&&\\[-3mm]
1 & 
\dfrac{1}{y_1}&
\dfrac{y_1y_2}{y_1\oplus 1} &
\dfrac{y_1+x_2}{x_1(y_1\oplus 1)}&
x_2 
\\[4.5mm]
\hline
&&&&\\[-3mm]
2 & 
\dfrac{y_2}{y_1y_2\oplus y_1\oplus 1}&
\dfrac{y_1\oplus 1}{y_1 y_2}&
\dfrac{y_1+x_2}{x_1(y_1\oplus 1)}&
\dfrac{x_1y_1y_2+y_1+x_2}{(y_1y_2\oplus y_1\oplus 1)x_1x_2} 
\\[4.5mm]
\hline
&&&&\\[-3mm]
3 & 
\dfrac{y_1y_2\oplus y_1\oplus 1}{y_2}&
\dfrac{1}{y_1(y_2\oplus 1)}&
\dfrac{x_1y_2+1}{x_2(y_2\oplus 1)}&
\dfrac{x_1y_1y_2+y_1+x_2}{(y_1y_2\oplus y_1\oplus 1)x_1x_2} 
\\[4.5mm]
\hline
&&&&\\[-3mm]
4 & 
\dfrac{1}{y_2} & 
y_1(y_2\oplus 1)& 
\dfrac{x_1y_2+1}{x_2(y_2\oplus 1)} &
x_1 
\\[4.5mm]
\hline
&&&&\\[-3mm]
5 & 
y_2 & y_1 & x_2&x_1
\\[.5mm]
\hline
\end{array}
\end{equation*}
}
	\caption{Seeds in type~$A_2$. 
Exchange matrices are given by~\eqref{eq:exchange-matrices-A2}. 
The initial cluster is $\xx_\circ\!=\!(x_1,x_2)$;  
the initial coefficient tuple is $\yy_\circ\!=\!(y_1,y_2)$. 
The formulas for the coefficient tuples~$\yy(t)$ are 
the tropical versions of the formulas in Figure~\ref{fig:y-seeds-A2}. 
Note that the labeled seed $\Sigma(5)$ is obtained from $\Sigma(0)$
by interchanging the indices~$1$ and $2$; the sequence then continues
by obvious periodicity, so that~$\Sigma(10)$ is identical to~$\Sigma(0)$,~etc.}
\label{fig:seeds-A2}
\end{figure}


\bibliographystyle{acm}
\bibliography{bibliography}
\label{sec:biblio}


\end{document}